\documentclass[a4paper, 10pt, reqno]{amsart}
\allowdisplaybreaks

\usepackage{amsmath}
\usepackage{amssymb}
\usepackage{amsthm}
\usepackage{mathtools, verbatim}
\usepackage[utf8]{inputenc}
\usepackage{indentfirst}
\usepackage{enumerate}
\usepackage[colorlinks=true]{hyperref}
\hypersetup{urlcolor=blue, citecolor=red}
     
\usepackage{graphicx}

\usepackage[usenames,dvipsnames]{pstricks}
\usepackage{epsfig}
\usepackage{pst-grad} 
\usepackage{pst-plot} 

\usepackage{ifthen}
\usepackage{color,xcolor}

\def\sideremark#1{\ifvmode\leavevmode\fi\vadjust{\vbox to0pt{\vss
 \hbox to 0pt{\hskip\hsize\hskip1em
 \vbox{\hsize2.1cm\tiny\raggedright\pretolerance10000
  \noindent #1\hfill}\hss}\vbox to15pt{\vfil}\vss}}}%

\allowdisplaybreaks

\usepackage{fullpage}
\numberwithin{equation}{section}

\def\polhk#1{\setbox0=\hbox{#1}{\ooalign{\hidewidth\lower1.5ex\hbox{`}\hidewidth\crcr\unhbox0}}}

\def\Xint#1{\mathchoice
{\XXint\displaystyle\textstyle{#1}}%
{\XXint\textstyle\scriptstyle{#1}}%
{\XXint\scriptstyle\scriptscriptstyle{#1}}%
{\XXint\scriptscriptstyle\scriptscriptstyle{#1}}%
\!\int}
\def\XXint#1#2#3{{\setbox0=\hbox{$#1{#2#3}{\int}$ }
\vcenter{\hbox{$#2#3$ }}\kern-.6\wd0}}

\def\dashint{\Xint-}

\newcommand{\supp}{\operatorname{supp}}

\renewcommand{\div}{\operatorname{div}}

\newcommand{\Om}{\Omega}
\newcommand{\mean}[1]{\,-\hskip-1.08em\int_{#1}}

\renewcommand{\div}{\operatorname{div}}

\newcommand{\Per}{\operatorname{Per}}

\newcommand{\R}{\mathbb{R}}
\newcommand{\eps}{\varepsilon}

\newcommand{\Reg}{{\operatorname{Reg}}}
\newcommand{\Sing}{{\operatorname{Sing}}}

\theoremstyle{plain}
\newtheorem{theorem}{Theorem}[section]

\newtheorem{definition}[theorem]{Definition}
\newtheorem{lemma}[theorem]{Lemma}
\newtheorem{corollary}[theorem]{Corollary}
\newtheorem{proposition}[theorem]{Proposition}
\newtheorem{remark}[theorem]{Remark}

\newcommand{\resmeas}{\mathbin{\vrule height 1.6ex depth 0pt width
0.13ex\vrule height 0.13ex depth 0pt width 1.3ex}}

\begin{document}

\title{
Free boundary regularity for a spectral optimal 
partition problem with volume and inclusion constraints
}

\author{Dario Mazzoleni, Makson S. Santos and Hugo Tavares}

\maketitle

\begin{abstract}
This paper is devoted to a complete characterization of the free boundary of all solutions to the following spectral $k$-partition problem with measure and inclusion constraints:
    \begin{equation*}
\inf \left\{\sum_{i=1}^k \lambda_1(\omega_i)\; : \;
\begin{array}{c}
\omega_i \subset \Omega \mbox{ are nonempty open sets for all } i=1,\ldots, k, \vspace{0.05cm}\\
 \omega_i \cap \omega_j = \emptyset  \: \text{for all}\: i \not=j \mbox{ and } \sum_{i=1}^{k}|\omega_i| =  a\\
\end{array}
\right\},
\end{equation*}
where $\Omega$ is a bounded domain of $\R^N$, $a\in (0,|\Omega|)$. In particular, we prove free boundary conditions, classify contact points, characterize the regular and singular part of the free boundary (including branching points), and describe the interaction of the partition with the fixed boundary $\partial \Omega$.

The proof is based on a perturbed version of the problem, combined with monotonicity formulas, blowup analysis and classification of blowups, suitable deformations of optimal sets and eigenfunctions, as well as the improvement of flatness of [Russ-Trey-Velichkov, CVPDE 58, 2019] for the one-phase points, and of [De Philippis-Spolaor-Velichkov, Invent. Math. 225, 2021] at two-phase points.
\end{abstract}

\smallskip


\smallskip


\tableofcontents

\section{Introduction}
This paper is devoted to a complete characterization of a solution to a spectral partition problem settled in a box and with measure constraints. More precisely, given $k\in \mathbb{N}$, a bounded Lipschitz domain $\Omega\subset\R^N$ and $a\in(0,|\Omega|)$, we study the problem:
\begin{equation}\label{eigenvalue_problem}
\Lambda_{1,k}(a):=\inf \left\{\sum_{i=1}^k \lambda_1(\omega_i)\; : \;
\begin{array}{c}
\omega_i \subset \Omega \mbox{ are nonempty open sets for all } i=1,\ldots, k, \vspace{0.05cm}\\
 \omega_i \cap \omega_j = \emptyset  \: \text{for all}\: i \not=j \mbox{ and } \sum_{i=1}^{k}|\omega_i| =  a\\
\end{array}
\right\}.
\end{equation}

In \cite{ASST}, it is shown that an optimizer always exists and that, moreover, if $(\Omega_1,\ldots, \Omega_k)$ is \emph{any} optimizer and we take, for each $i=1,\ldots, k$,  $u_i\in H^1_0(\Omega_i)$ to be an associated first eigenfunction for the set $\Omega_i$, then $u_i$ is locally Lipschitz continuous in $\Omega$. The aim of the present paper is to pursue a deeper study of the problem: we will provide a complete description of the free boundary $\cup_{i=1}^k \partial \Omega_i$  (regularity, characterization and properties of contact points between different elements of the partition, interaction with the fixed boundary $\partial \Omega$,\ldots) for all solutions of \eqref{eigenvalue_problem}.

\smallbreak

In order to state our results, we introduce some notation. Let $(\Omega_1,\dots, \Omega_k)$ be a partition of $\Omega$, i.e., $\Omega_1,\ldots, \Omega_k\subset \Omega$ are open sets, and $\Omega_i\cap \Omega_j=\emptyset$ for every $i\neq j$. For each $i=1,\ldots, k$, we define the following mutually disjoint sets:
\begin{itemize}
\item $\Gamma_{OP}(\partial\Omega_{i})$ denotes the set of \emph{interior one-phase points}, namely\[
\Gamma_{OP}(\partial\Omega_{i})=(\partial \Omega_{i}\cap \Omega)\setminus \cup_{j\not =i}\partial\Omega_{j};
\]
\item  $\Gamma_B(\partial\Omega_{i})$ denotes the sets of \textit{one-phase points on the boundary} of $\partial \Omega$, namely \[
\Gamma_B(\partial\Omega_{i})=(\partial \Omega_{i}\cap \partial\Omega)\setminus \cup_{j\not=i}\partial \Omega_{j};
\]
\item $\Gamma_{TP}(\partial\Omega_{i})$ denotes the sets of \textit{interior} two-phase points:
\[
\Gamma_{TP}(\partial\Omega_{i})=\Big\{x\in (\partial \Omega_{i}\cap \partial \Omega_{j}\cap \Omega)\setminus \cup_{k\neq j,i}\partial \Omega_k \text{ for some $j\not=i$ }\Big\}.
\]
\end{itemize}

Our main result is the following.

\begin{theorem}\label{thm:main}
Let $\Omega\subset \R^N$ be a bounded Lipschitz domain. Problem \eqref{eigenvalue_problem} admits a solution, and every optimal partition $(\Omega_1,\ldots, \Omega_k)$ of \eqref{eigenvalue_problem} satisfies the following properties.
For each $i=1,\ldots,k$, each set $\Omega_i$ is connected, and the boundary $\partial \Omega_{i}$ of  $\Omega_{i}$ can be decomposed as the following disjoint union \[
\partial \Omega_{i}=\Gamma_{OP}(\partial\Omega_{i})\cup \Gamma_{TP}( \partial\Omega_{i})\cup  \Gamma_{B} (\partial\Omega_{i}),
\]
Moreover: 
\begin{enumerate}
    \item $\partial \Omega_i\cap \Omega$ can be decomposed as a disjoint union of a regular and a (possibly empty) singular part, 
$$
\partial \Omega_i \cap \Omega=\operatorname{Reg}\left(\partial \Omega_i\cap \Omega\right) \cup \operatorname{Sing}\left(\partial \Omega_i\cap \Omega\right),
$$
where:
\begin{enumerate}
\item the regular part  $\operatorname{Reg}\left(\partial \Omega_i\cap \Omega\right)$ is a relatively open subset of $\partial \Omega_i \cap \Omega$ and it is of class $C^{1, \alpha}$ for all $\alpha\in (0,1/2]$. Moreover, the two-phase free boundary is regular, that is,
$$
\Gamma_{TP}(\Omega_i) \subset \operatorname{Reg}\left(\partial \Omega_i\cap \Omega\right);
$$
\item the singular set $\operatorname{Sing}\left(\partial \Omega_i\cap \Omega\right)$ is a closed subset of $\partial \Omega_i \cap \Omega$ of Hausdorff dimension at most $N-5$. Precisely, there is a critical dimension $N^* \in\{5,6,7\}$ such that
\begin{itemize}
    \item[-] if $N<N^*$, then $\operatorname{Sing}\left(\partial \Omega_i\cap \Omega\right)=\emptyset$;
\item[-] if $N=N^*$, then $\operatorname{Sing}\left(\partial \Omega_i\cap \Omega\right)$is locally finite in $\Omega$;
\item[-] if $N>N^*$, then $\operatorname{Sing}\left(\partial \Omega_i\cap \Omega\right)$is a closed $\left(N-N^*\right)$-rectifiable subset of $\partial \Omega_i \cap \Omega$ with locally finite $\mathcal{H}^{N-N^*}$ measure.
\end{itemize}
\end{enumerate}
\item If, in addition, $\Omega$ is of class $C^{1,1}$, then the full $\partial \Omega_i$ can be decomposed as the disjoint union of a regular part $\Reg\left(\partial \Omega_i\right)$ and a singular part $\Sing\left(\partial \Omega_i\right)$, where:
\begin{enumerate}
\item[(a)]$\Reg\left(\partial \Omega_i\right)$ is an open subset of $\partial \Omega_i$ and locally the graph of a $C^{1,1 / 2}$ function; moreover, $\Reg\left(\partial \Omega_i\right)$ contains both $\Reg\left(\partial \Omega_i \cap \Omega\right)$ and $\partial \Omega_i \cap \partial \Omega$; 
\item[(b)] $\Sing\left(\partial \Omega_i\right)=\Sing\left(\partial \Omega_i \cap \Omega\right)$.  
\end{enumerate}
\end{enumerate}
\end{theorem}

\begin{remark}
We highlight that, in Theorem \ref{thm:main}, the assumption that $\Omega$ has Lipschitz boundary is needed only to show that there are no two-phase points at the fixed boundary $\partial \Omega$ (see Section \ref{eq:no2-phase-bdary} below). For the \textit{inner} regularity results (namely, of $\partial\Omega_{u_i}\cap \Omega$) it is enough to consider a bounded domain $\Omega$.
\end{remark}

Observe that, for the one-phase problem $k=1$, a complete characterization of solutions has been obtained in a collection of works \cite{
BrianconHayouniPierre,BrianconLamboley, BucurMazzoleniPratelliVelichkov, RussTreyVelichkov}, and we  will comment this below in this introduction.

\smallbreak

Before we explain the strategy of the proof of Theorem \ref{thm:main}, we provide the state of the art of this topic, pointing out what are the similarities and differences with related open partition problems in the literature.

\subsection{State of the art and related problems}
The study of shape optimization for spectral functionals has been a very active topic in the last few years, see~\cite{HenrotPierre} for a general overview. The interest is motivated both by the several applications in physics and engineering, where these quantities naturally arise, and by the mathematical beauty and difficulty of the problems in this field. We will focus the discussion here on spectral functionals with eigenvalues of the Dirichlet Laplacian (and, above all, the first Dirichlet eigenvalue) but other boundary conditions or other operators could also be considered.
Concerning optimal partitions for the (sum of the) first eigenvalues of the Dirichlet Laplacian, in the last 20 years there has been a flourishing of works, mostly for two main classes of problems, and nowadays several properties of optimizers are well-understood.
The first class of spectral optimal partitions for the principal Dirichlet eigenvalue that has been studied starting from~\cite{CaffarelliLin2} and~\cite{ContiTerraciniVerziniOPP,ContiTerraciniVerziniVariation} is the following: 
\begin{equation}\label{eigenvalue_problem_biblio}
\inf \left\{\sum_{i=1}^k \lambda_1(\omega_i)\; : \;
\begin{array}{c}
\omega_i \subset \Omega \mbox{ are nonempty open sets for all } i=1,\ldots, k,\\ \omega_i \cap \omega_j = \emptyset  \: \text{for all}\: i \not=j 
\end{array}
\right\},
\end{equation}
which corresponds to $a=|\Omega|$ in \eqref{eigenvalue_problem}, that is, it coincides with the level $\Lambda_{1,k}(|\Omega|)$ (since, by minimality, the solution to \eqref{eigenvalue_problem_biblio} exhausts the whole domain $\Omega$).
In this setting, it is known~\cite{CaffarelliLin2,ContiTerraciniVerziniOPP, TavaresTerracini1} that optimal partitions exists,
and that the interior free boundary, i.e. $\cup_i\partial\omega_i\cap \Omega$,  is regular, up to a singular set of lower dimension. A deeper analysis of the singular set has been carried out much more recently in~\cite{Alper}, where the author proves both that the $(N-2)$-Hausdorff dimension of the singular set is finite, as well as a stratification result. Let us notice that, in this case, also the problem for higher eigenvalues has been
studied in~\cite{RamosTavaresTerracini}. We also point out to the reader the very recent work~\cite{OgnibeneVelichkov}, where the boundary regularity (namely, at the intersection with the boundary box $\Omega$) is considered for the sum of first Dirichlet eigenvalues.
Let us stress that the shape and properties of the optimal sets in this case are rather different from the ones of our problem~\eqref{eigenvalue_problem}: in our case, as stated in Theorem \ref{thm:main}, the free boundary $\cup_{i=1}^k \partial \Omega_i$ can not contain triple or higher multiplicity points, and it intersects the boundary of $\Omega$ at one-phase points tangentially. This is in sharp contrast with spectral partition problems without measure constraints~\eqref{eigenvalue_problem_biblio},
where triple points may appear and the sets of the optimal partition reach the boundary of the box orthogonally. 

\smallbreak

Another optimal partition problem which has been considered in the literature~\cite{BogoselVelichkov,BucurVelichkov,DePhilippisSpolaorVelichkov, SpolaorTreyVelichkov} and whose behavior is more similar to that of our problem~\eqref{eigenvalue_problem} is:
\begin{equation}\label{eigenvalue_problem_plusmeasure}
\inf \left\{\sum_{i=1}^k \lambda_1(\omega_i)+m \sum_{i=1}^k|\omega_i|\; : \;
\begin{array}{c}
\omega_i \subset \Omega \mbox{ are nonempty open sets for all } i=1,\ldots, k,\ \\
\omega_i \cap \omega_j = \emptyset  \: \text{for all}\: i \not=j 
\end{array}
\right\},
\end{equation}
for $m>0$. The properties of the solutions of~\eqref{eigenvalue_problem} and of~\eqref{eigenvalue_problem_plusmeasure} are the same, see \cite[Corollary~1.3]{DePhilippisSpolaorVelichkov}. The main results of existence and regularity of the solutions for this problem have been proved first by Bucur and Velichkov~\cite{BucurVelichkov}, and then improved in the works~\cite{BogoselVelichkov,SpolaorTreyVelichkov}, while the complete regularity picture has been proved in~\cite{DePhilippisSpolaorVelichkov}.

In this setting, it is also interesting to study the \emph{one-phase} case, namely $k=1$. Extensive work has been conducted in this topic, starting from the existence result by Buttazzo and Dal Maso~\cite{ButtazzoDalMaso} in the framework of capacitary measures and quasi-open sets. In this case, the regularity of the optimal sets comes directly from the associated one-phase free boundary problems~\cite{AltCaffarelli}, see also~\cite{BrianconHayouniPierre,BrianconLamboley}. Several works concerning the regularity of optimal shapes have been done, on the other hand, in the case of higher eigenvalues~\cite{BucurMazzoleniPratelliVelichkov,KriventsovLin,MazzoleniTerraciniVelichkov} in $\R^N$ (where it is easy to prove, by scaling, the equivalence with a constraint problem), and \cite{MazzoleniTreyVelichkov} for the second eigenvalue in a bounded $\Omega$.

We highlight that problems \eqref{eigenvalue_problem} and \eqref{eigenvalue_problem_plusmeasure} are not equivalent. Instead, it has been shown in \cite{ASST} that \eqref{eigenvalue_problem} is equivalent to
\begin{equation}\label{eigenvalue_problem_plusmeasure_positivepart}
\inf \left\{\sum_{i=1}^k \lambda_1(\omega_i)+\mu \left(\sum_{i=1}^k|\omega_i|-a\right)^+\; : \;
\begin{array}{c}
\omega_i \subset \Omega \mbox{ are nonempty open sets for all } i=1,\ldots, k,\ \\
\omega_i \cap \omega_j = \emptyset  \: \text{for all}\: i \not=j 
\end{array}
\right\}
\end{equation}
for sufficiently large $\mu>0$. Even though this problem may look similar to \eqref{eigenvalue_problem_plusmeasure_positivepart}, the term with the positive part lacks good Lipschitz-type estimates from below, which creates several difficulties in terms of the study of free boundaries (for example, in the proofs of nondegeneracy properties or of finite perimeter).  In conclusion, the technical issues due to the measure constraint do not allow to apply directly the known results for \eqref{eigenvalue_problem_plusmeasure}, but rather  require the development of several new ideas and tools.
This is not surprising, as also in the case of the one-phase problem, namely $k=1$, the study of the problem with the measure constraint  requires substantial additional work~\cite{BrianconLamboley,RussTreyVelichkov}. The arguments therein use strongly the deduction of a Euler-Lagrange equation, and in particular a local minimality property for $\lambda_1(\cdot)+\Lambda_\pm |\cdot|$ for inner and outer perturbations respectively. However, this is also a property  which seems specific of one-phase problems, and in order to deal with multiple phases we have to follows a different direction, see Subsection \ref{sec:strategy}.

\smallbreak

Although we have seen that the features of problems~\eqref{eigenvalue_problem_biblio} and~\eqref{eigenvalue_problem_plusmeasure} are rather different, the main techniques to study the regularity of the sets of an optimal partition are based on the tools developed in the framework of free boundary problems. These tools and ideas will be crucial also in our paper as well, as it can be guessed already from the statement of our main result. Therefore, we try to summarize (without any claim of completeness) also the most important results on the regularity theory for free boundary problems.

\smallbreak

\paragraph{\bf{One-phase free boundary problems.}}
The first and seminal paper on the regularity theory for one-phase free boundary problems was developed by Alt and Caffarelli~\cite{AltCaffarelli}, for functionals of type
\begin{equation}\label{eq:AC}
J(u)=\int_\Omega (|\nabla u|^2+Q(x)\chi_{\{u>0\}})\, dx,\qquad u=u_0 \text{ on } \partial \Omega,
\end{equation}
where $Q$ is a H\"older continuous given function and $u_0\in H^{1/2}(\partial \Omega,\R^+)$.
They prove existence of minimizers, nondegeneracy, Lipschitz continuity of optimizers, finite perimeter of the free boundary, and that the reduced boundary is $C^{1,\alpha}$ (while the remaining singular set has $\mathcal{H}^{N-1}$ measure zero). Several important tools have been developed to show those properties, such as  harmonic extensions to build suitable competitors, and blowup analysis. Higher regularity results for the reduced boundary, namely the passage from $C^{1,\alpha}$ to analytic (if $Q$ is regular enough), employ the renowned techniques of Kinderlehrer-Nirenberg \cite{KN77}.

Concerning the singular set, in \cite{AltCaffarelli} it is shown that it is empty in dimension $2$. Then a stratification result was first done by Weiss~\cite{Weiss99}, introducing a monotonicity formula: namely, calling $N^*$ the lowest dimension for which the Alt-Caffarelli functional~\eqref{eq:AC} admits a one-homogeneous singular local minimizer in the whole space $\Omega=\R^N$, he proved that the singular set has at most dimension $N-N^*$ and $N^*\geq 3$.
Then several works~\cite{CFK04,JerisonSavin15,DeSilvaJerison22,EdelenEngelstein19} improved the result by showing that $N^*\in\{5,6,7\}$, that the singular set has $\mathcal{H}^{N-N^*}$-finite measure, and it is $(N-N^*)$-rectifiable.

\smallbreak

\paragraph{\bf{Two-phase free boundary problems.}} The first paper on two-phase free boundary problems in this setting is the celebrated work by Alt, Caffarelli, and Friedman~\cite{AltCaffarelliFriedman}, concerning the functional
\begin{equation}\label{eq:ACF}
J_{tp}(u)=\int_\Omega (|\nabla u|^2+\lambda_+^2\chi_{\{u>0\}}+\lambda_-^2\chi_{\{u<0\}}+\lambda_0^2\chi_{\{u=0\}})\, dx,\qquad u=v\text{ on } \partial \Omega,
\end{equation}
where $v\in H^{1/2}(\partial\Omega)$ and $\lambda_\pm\geq \lambda_0>0$ (actually, they considered the simplified version where $\lambda_+\not=\lambda_-$ and $\lambda_+=\lambda_0$ or $\lambda_-=\lambda_0$ so that the zero region has zero measure).
In this work, the fundamental ACF-monotonicity formula was first introduced. Then, several regularity results in the two-phase setting have been deduced, see for instance~\cite{CaffarelliSalsa,Caffarelli88, DeSilvaFerrariSalsa} and references (we refer the reader to the introduction of \cite{DePhilippisSpolaorVelichkov} for a complete list and description). Summarizing, it was proved that the optimal sets (namely $\{u>0\}$ and $\{u<0\}$) are regular up to a singular one-phase set, which is small, but there could also be branching points which were not completely understood until recently.
Finally (following by the work \cite{SpolaorVelichkov} in the particular dimension 2) the work by De Philippis, Spolaor and Velichkov~\cite{DePhilippisSpolaorVelichkov}  led to a substantial advancement in the understanding of the problem; the authors were able also to classify the branching points  (the key technical issue being a suitable \emph{improvement of flatness} at those points). Let us stress that branching points do not appear in the ACF case when $\lambda_+\not=\lambda_-$ and $\lambda_+=\lambda_0$ or $\lambda_-=\lambda_0$, but actually are a key feature of the problem as soon as the zero phase exists. This difference is analogous to the difference between problem~\eqref{eigenvalue_problem_biblio} (where the partition fully occupies the box) and~\eqref{eigenvalue_problem_plusmeasure} (where the partition does not fully cover the box and thus there is a zero phase).

\subsection{Strategy of the proof}\label{sec:strategy}
The proof of the main Theorem~\ref{thm:main} is rather involved and requires several technical steps. Therefore, for the reader's sake, we summarize here the main ideas. We start with the information that, if $(\Omega_1,\ldots, \Omega_k)$ is a solution to \eqref{eigenvalue_problem}, then
\[
\lambda_1(\Omega_i)=\min \left\{\lambda_1(\omega) : \omega\subset \Omega\setminus \cup_{j\not=i}\overline \Omega_j,\; |\omega|=a-\sum_{j\not=i}|\Omega_j|\right\}
\]
for each $i$, and so, at one-phase points, one can use the existing results in the literature. This is not however the case at two-phase points, whose study is the core of this work.

Having this in mind, first of all, in Section~\ref{sec:penalized}, we introduce a penalized version of problem~\eqref{eigenvalue_problem} in order to get rid of the measure constraint  in \eqref{eigenvalue_problem}, inspired by what is done in~\cite{AguileraAltCaffarelli} in the one-phase setting. We use a penalization with good upper and lower Lipschitz-type estimates (see Lemma \ref{lemma:feta} for a precise statement), which also allows to avoid the positive part as in \eqref{eigenvalue_problem_plusmeasure_positivepart}. This penalized problem can be naturally formulated in a free boundary environment, as it is equivalent to consider a partition or the vector of first eigenfunctions (suitably normalized) associated to each set of the partition. It is through this penalization that we are able to treat two-phase points; on the other hand, this penalization is not $C^1$, which creates many difficulties and requires several new ideas. For this new problem without measure constraint, in this chapter we prove some of the basic results concerning existence of a solution (which in this case is at the same time both a partition and a vector of $H^1$ functions), the nondegeneracy and Lipschitz continuity of the optimal eigenfunctions and suitable density estimates. The techniques used here are essentially the nowadays standard strategies for proving these basic regularity results in free boundary problems, and we adapt them to our situation.

Then, in Section~\ref{sec:notriple}, using a three phase monotonicity formula, the nondegeneracy property and ideas from~\cite{BucurVelichkov}, we prove that optimal partitions in our setting do not admit neither triple nor higher multiplicity points (in the interior of the box),  neither two-phase nor higher multiplicity points at the boundary of the box $\Omega$. 

Section~\ref{sec:equivalence} is aimed at proving the equivalence between the original problem and the penalized one (for a suitable choice of the parameter in the penalization), in the sense that the energy levels coincide and both problems have the same optimizers. To do this, we first need to show the existence of a \emph{one-phase} point in the reduced boundary of at least one set of the an optimal partition, which allows us to use shape variations to deduce an optimality condition, which is the key to show the desired equivalence as it was done in~\cite{AguileraAltCaffarelli}.

Then, in Section~\ref{sec:butwophase}, we start a detailed study of two-phase points, precisely classifying \emph{all} blowups centered there. To do this, we need to use ACF-type monotonicity formulas and techniques inspired from the works of Conti, Terracini and Verzini~\cite{ContiTerraciniVerziniVariation,ContiTerraciniVerziniOPP}. We stress that, since the penalization is not $C^1$, then blowup limits are \emph{not} local minimizers of some functional, and this fact  is the main difficulty of this part (and a key difference between our penalized problem and \eqref{eigenvalue_problem_plusmeasure}).

Thanks to the analysis of blowups at two-phase points, we can use, in Section~\ref{sec:viscosity}, the notion of viscosity solutions to interpret the optimality condition at all points of the free boundary of an optimal partition and then employ the improvement of flatness of~\cite{RussTreyVelichkov} for the one-phase points and the one of~\cite{DePhilippisSpolaorVelichkov} at two-phase points (including the branching points) to conclude the proof of our main result.
We stress that we are not able to use from the beginning the approach of the viscosity solution on the penalized problem, once again because the penalization is not $C^1$ (and it does not seem possible to find a regular penalization with the good properties that we need).

Finally, in the appendix, we include some useful results and remarks, some basic facts about Geometric Measure Theory and quasi-open sets.

\vspace{0.2cm}
\paragraph{\bf Notation} 
We denote the ball centered at $x_0\in \R^N$ and radius $r>0$ by $B_r(x_0)$; whenever $x_0=0$, we simply write $B_r$. Given a measurable function $v:\Omega \to \R$, we denote its nodal set by $\Omega_v:=\{v\neq 0\}$. For measurable sets $A_n,A\subset \Omega$, we say that $A_n\to A$ in $L^1(\Omega)$ if $\chi_{A_n}\to \chi_A$ in $L^1(\Omega)$, where $\chi$ denotes the characteristic function. The notation for $L^p(\Omega)$ norms is $\|\cdot \|_p$, for $p\in [1,\infty]$, while $\|u\|_{H^1_0(\Omega)}^2:=\int_\Omega |\nabla u|^2$ is the standard $H^1_0$-norm, where $|\cdot|$ denotes, as usual, the Euclidean norm. The $N$-dimensional Lebesgue measure of a measurable set $E\subset \R^N$ is denoted by $|E|$. The average of a function $f$ over a set $E$ is given by $\dashint_E f=\frac{1}{|E|}\int_E f$. The positive and negative parts of $f$ are denoted, respectively, by $f_+=\max\{f,0\}$, $f_-=\max\{-f,0\}$. 
We denote by $\partial_{rel}$ the relative boundary of a set.

\section{A penalized problem. Existence, nondegeneracy, Lipschitz continuity, finite perimeter, density estimates.}\label{sec:penalized}

In this section, inspired by \cite{AguileraAltCaffarelli}, we introduce a penalized function to treat the regularity of the free boundary for the solutions of \eqref{eigenvalue_problem}. We proceed by establishing various properties pertinent to the solutions of this functional. We start with the following piecewise linear auxiliary function: for $\eta\in (0,1]$,  define

\begin{equation}\label{eq:propfeps}
f_\eta(s):=
\begin{cases}
\frac{1}{\eta}(s-a) & \text{ if } s\geq a\\
{\eta}(s-a) & \text{ if } s\leq a.\\
\end{cases}
\end{equation}

\begin{lemma}\label{lemma:feta} The function $f_\eta$ satisfies the following properties.
\begin{enumerate}
\item  $f_\eta(s)\geq -\eta a$ for every  $s\geq 0$,
\item  $\eta (s_2-s_1)\leq f_\eta(s_2)-f_\eta(s_1)\leq \frac{1}{\eta} (s_2-s_1)$ whenever $0\leq s_1\leq s_2$.
\end{enumerate}
\end{lemma}
The proof of this lemma is straightforward, and we omit it. Now, we  work with the penalized functional
\begin{equation*}
J_{\eta} (u_1,\ldots, u_k):= \sum_{i=1}^k \frac{\displaystyle \int_{\Omega}|\nabla u_i|^2}{\displaystyle \int_{\Omega} u_i^2}  + f_\eta\left(\sum_{i=1}^k|\Omega_{u_i}|\right),\qquad (u_1\ldots, u_k) \in \overline{H},
\end{equation*}
where 
\begin{equation}\label{eq:barH}
\overline{H} : = \Big\{(u_1,\ldots,  u_k) \in H^1_0(\Omega;\R^k): \  u_i \neq 0 \text{ for every } i,\   u_i \cdot u_j \equiv 0 \text{ for } i\neq j \Big \}.
\end{equation}
We now focus our attention on the infimum:
\begin{equation}\label{eq:ceta}
c_\eta:=\inf_{\overline H} J_\eta.
\end{equation}
For future reference, observe that, by Lemma \ref{lemma:feta}-(1), 
\begin{equation}\label{eq:useful_inequality}
\sum_{i=1}^k\frac{\int_\Omega |\nabla u_i|^2}{ \int_\Omega u_i^2}\leq J_\eta(u_1,\ldots, u_k)+a, \qquad \text{for every }(u_1,\ldots, u_k)\in \overline H,\ \eta\in(0,1].
\end{equation}

\begin{remark}\label{rem:Jc}
We point out that in \cite[Theorem 1.2]{ASST} it is shown that
\[
\Lambda_{1,k}(a)=\min \left\{J(u_1,\ldots, u_k):\ (u_1,\ldots, u_k)\in H_a\right\},\quad \text{ where }\quad
J(u_1,\ldots, u_k)=\sum_{i=1}^k \frac{\displaystyle \int_{\Omega}|\nabla u_i|^2}{\displaystyle \int_{\Omega} u_i^2} 
\]
and $H_a=\{(u_1,\ldots, u_k)\in H^1_0(\Omega;\R^k): \ u_i\neq 0 \text{ for every $i$},\ u_i\cdot u_j\equiv 0 \ \text{ for } i\neq j,\ \sum_{i=1}^k |\Omega_{u_i}|\leq a\}$.
\end{remark}

\subsection{Uniform bounds, variational inequalities and regularity of minimizers}

\begin{lemma}\label{lemma:bounds_ceta} There exists $C=C(N,k,a,\Omega)>0$ such that  
\[
\Lambda_{1,k}(|\Omega|)-a\leq c_\eta\leq \Lambda_{1,k}(a)\quad \text{ for every $\eta\in (0,1]$},
\]
where we recall that $\Lambda_{1,k}(|\Omega|)$ coincides with the level \eqref{eigenvalue_problem_biblio}.
\end{lemma}
\begin{proof}
Take $(u_1,\ldots, u_k)\in H_a$ (recall Remark \ref{rem:Jc}) and note that $\sum_{i=1}^k|\Omega_{u_i}|\leq a$, thus using the definition of $f_\eta$ it holds $f_\eta(\sum_{i=1}^k|\Omega_{u_i}|)\leq 0$. Hence, we obtain
$$
J(u_1,\ldots,u_k)\geq J_\eta(u_1,\ldots,u_k)\geq c_\eta.
$$
Then, passing to the infimum in $H_a$, we deduce$$
\Lambda_{1,k}(a)\geq c_\eta.
$$
The lower bound, on the other hand, is a consequence of \eqref{eq:useful_inequality} and the fact that
\[
\Lambda_{1,k}(|\Omega|)=\inf \left \{\sum_{i=1}^k \frac{\displaystyle \int_{\Omega}|\nabla u_i|^2}{\displaystyle \int_{\Omega} u_i^2} : \ (u_1,\ldots, u_k) \in \overline H   \right\}
\]
(for this identity, see, for instance, the proof of Theorem 1.1 in \cite{ContiTerraciniVerziniOPP}). 
\end{proof}

By using the previous result, we can prove the existence of minimizers of $J_\eta$ over $\overline H$. This is the content of the next lemma. 

\begin{lemma}\label{lemma:unifbound_ceta}
The level $c_\eta$ is achieved, that is, 
\[
\text{there exists $u_\eta=(u_{1,\eta},\ldots, u_{k,\eta})\in \overline H$ such that $J_\eta(u_{1,\eta},\ldots, u_{k,\eta})=c_\eta$.}
\]
Moreover, there exists a constant $C=C(N,k,a,\Omega)>0$ such that
\begin{equation}\label{eq:H10bounds}
\|u_\eta\|_{H^1_0(\Omega)} \leq C, \qquad \text{ for every } L^2\text{--normalized minimizer }  u_\eta \text{ of } c_\eta, \text{ and } \eta\in (0,1].
\end{equation}
\end{lemma}
\begin{proof}
The existence of a minimizer is a straightforward consequence of the direct method of the Calculus of Variations. Indeed, Lemma \ref{lemma:bounds_ceta} implies that   $c_\eta>-\infty$. Take a minimizing sequence $(u_{1,n},\ldots, u_{k,n})\in \overline H$ such that
\[
J_\eta(u_{1,n},\ldots, u_{k,n})\to c_\eta,\qquad \int_\Omega u_{i,n}^2=1 \text{ for every $i$}.
\]
Then, by \eqref{eq:useful_inequality} and Lemma \ref{lemma:bounds_ceta},
\begin{equation}\label{bds_aux1}
\sum_{i=1}^k \int_\Omega |\nabla u_{i,n}|^2 \leq  J_\eta(u_{1,n},\ldots, u_{k,n})+ a \leq c_\eta+1+a \leq C+1+a
\end{equation}
for sufficiently large $n$. In particular, for each $i$, the sequence $(u_{i,n})_{n \in \mathbb{N}}$ is uniformly bounded in $H^1_0(\Omega)$, and then there exists $u_i\in H^1_0(\Omega)$ such that, up to a subsequence, as $n\to \infty$,
\[
u_{i,n} \to u_i \;\; \text{ weakly in } H^1_0(\Omega), \mbox{ strongly in } L^2(\Omega),\ \text{ and pointwise a.e. in } \Omega.
\]
In particular,
\begin{equation}\label{bds_aux2}
\int_\Omega u_i^2=1 \quad \text{ and } \quad u_i\cdot u_j\equiv 0 \text{ a.e. in $\Omega$,}
\end{equation}  
so that, in particular,  $(u_{1},\ldots,u_k)\in \overline H$. Moreover, 
\[
\int_\Omega |\nabla u_{i}|^2\leq \liminf_n \int_\Omega |\nabla u_{i,n}|^2
\]
and, since $\chi_{u_i}\leq \liminf_n \chi_{u_{i,n}}$ a.e. in $\Omega$, by Fatou's lemma
\[
\sum_{i=1}^k |\Omega_{u_i}|\leq \liminf_n \sum_{i=1}^k|\Omega_{u_{i,n}}|.
\]
In conclusion, since $f_\eta$ is a continuous and increasing function:
\[
c_\eta\leq J_\eta(u_1,\ldots, u_n)\leq \liminf_n J_\eta(u_{1,n},\ldots u_{k,n})=c_\eta,
\]
and $c_\eta$ is achieved by $(u_1,\ldots, u_n)$. 

The uniform bounds in $H^1_0(\Omega)$ norm are a consequence of \eqref{eq:useful_inequality}, which implies
\begin{equation*}
\sum_{i=1}^k \int_\Omega |\nabla u_i|^2\leq J_\eta(u_1,\ldots, u_k)+ a=c_\eta+ a\leq C+a.  \qedhere
\end{equation*}
\end{proof}

Now, we are ready to prove some regularity results for the minimizers of $J_\eta$. We start by showing, in particular, the $L^\infty$ boundedness of the solutions.

\begin{lemma}\label{lemma:otherproperties}
Let $\eta\in (0,1]$ and let $(u_{1},\ldots, u_{k})\in \overline H$ be a nonnegative $L^2$-normalized minimizer of $J_\eta$, attaining the level $c_\eta$. Then
\[
- \Delta u_i=\lambda_1(\Omega_{u_i}) u_i \text{ in } \Omega_{u_i} \quad \text{and}\quad  - \Delta u_i\leq \lambda_1(\Omega_{u_i})u_i \text{ in } \Omega.
\]
As a consequence, there exists a constant $C=C(N,k,a,\Omega)>0$ (in particular, independent of $\eta)$ such that
\[
\|u_i\|_{L^\infty(\Omega)}\leq C.
\]
\end{lemma}
\begin{proof} 
Let us note that $\Omega_{u_i}=\{u_i>0\}$ is a quasi-open set (see Section~\ref{ssec:quasiopen}). For all nonzero $v\in H^1_0(\Omega_{u_i})$, testing the minimality of $(u_1,\dots,u_k)$ for $J_\eta$ with $(u_1,\dots, u_{i-1},v,u_{i+1},u_k)\in \overline{H}$ and noting that $\{v>0\}\subseteq \{u_i>0\}$ and $f_\eta$ is increasing shows immediately that $u_i$ is the first Dirichlet eigenfunction on $\Omega_{u_i}$.

\smallbreak

Let $\varphi \in C^{\infty}_c(\Omega)$ be a nonnegative function and, for $t> 0$ small and $i= 1, \ldots, k$, consider the perturbation
\[
\left(u_1, \ldots, u_{i-1}, (u_i-t\varphi)^+, u_{i+1}, \ldots, u_k\right)\in \overline H.
\]
Observe that $\Omega_{(u_i-t\varphi)^+} \subseteq \Omega_{u_i}$, thus we have $|\Omega_{(u_i-t\varphi)^+}| \leq |\Omega_{u_i}|$. Since $(u_1, \ldots, u_k)$ is a minimizer of $J_\eta|_{\overline{H}}$, we have
\[
\sum_{j=1}^k\int_{\Omega}|\nabla u_j|^2 + f_\eta\left(\sum^k_{j=1}|\Omega_{u_j}|\right) \leq \mathop{\sum_{j=1}^k}_{j\neq i} \int_\Omega |\nabla u_j|^2 + \dfrac{\int_\Omega |\nabla(u_i-t\varphi)^+|^2}{\|(u_i-t\varphi)^+\|_2^2} + f_\eta\left(\mathop{\sum_{j=1}^k}_{j\neq i}|\Omega_{u_j}| + |\Omega_{(u_i-t\varphi)^+}|\right),
\]
which implies, by Lemma \ref{lemma:expansion_L^2} and the fact that $\|u_i\|_2=1$,
\begin{align*}
\int_{\Omega}|\nabla u_i|^2 + f_\eta\left(\sum^k_{j=1}|\Omega_{u_j}|\right) - f_\eta &\left(\mathop{\sum_{j=1}^k}_{j\neq i}|\Omega_{u_j}| + |\Omega_{(u_i-t\varphi)^+}|\right) \leq \int_{\Omega}|\nabla(u_i-t\varphi)|^2\left( 1 + 2t\int_{\Omega}u_i\varphi + o(t) \right)\\
        &=\left(\int_\Omega |\nabla u_i|^2-2t \int_\Omega \nabla u_i \cdot\nabla \varphi + o(t)\right) \left( 1 + 2t\int_{\Omega}u_i\varphi + o(t) \right)\\
        &=\int_\Omega |\nabla u_i|^2 + 2t \int_\Omega |\nabla u_i|^2\int_\Omega u_i \varphi - 2t \int_\Omega \nabla u_i \cdot \nabla \varphi + o(t),
\end{align*}
as $t\to 0^+$. Hence, by using the monotonicity of $f_\eta$, we obtain
\[
0 \leq f_\eta\left(\sum^k_{j=1}|\Omega_{u_j}|\right) - f_\eta \left(\mathop{\sum_{j=1}^k}_{j\neq i}|\Omega_{u_j}| + |\Omega_{(u_i-t\varphi)^+}|\right)  \leq -2t\int_{\Omega}\nabla u_i\cdot\nabla\varphi + 2t\int_{\Omega}|\nabla u_i|^2\int_{\Omega}u_i\varphi + o(t).
\]
By dividing the inequality above by $t>0$, and passing to the limit as $t \to 0^+$, we conclude that
\[
\int_{\Omega}\nabla u_i\cdot\nabla\varphi \leq \int_{\Omega}|\nabla u_i|^2\int_{\Omega}u_i\varphi,\quad \text{i.e.,} \quad
-\Delta u_i \leq \lambda_{u_i}u_i \;\; \mbox{ in } \; \Omega.
\]

Finally, the uniform $L^\infty$--bounds are now a consequence of \eqref{eq:H10bounds} and of a standard Brezis-Kato type argument (or, alternatively, a consequence of the result for eigenfunctions by Davies~\cite[Example~2.1.8]{Davies})
\end{proof}

Following some ideas from \cite{ASST}, we  proceed with some qualitative properties for minimizers, which shall lead to their Lipschitz regularity (see Theorem \ref{thm:general_Lipschitz} and Corollary \ref{coro_Lispchitz} below). This proof uses some perturbations which will also be employed later on in the classification result for blowups at two-phase points (Theorem \ref{thm:blowuplimits}), in particular in the proof of Lemma \ref{lemma:inequalities_classS_blowup}.

\begin{proposition}\label{prop:classS}
Let $\eta\in (0,1]$  and let $(u_{1},\ldots, u_{k})\in \overline H$ be an $L^2$-normalized minimizer of $J_\eta$ achieving $c_\eta$. Define  
\[
\hat{u}_i=u_i-\sum_{j\neq i} u_j\quad \text{ and } \quad \lambda_{u_i}:= \int_\Omega |\nabla u_i|^2\quad \text{ for }i=1,\ldots, k.
\]
\begin{enumerate}
\item For each $i=1, \ldots, k$, $x_0\in\Omega$, $r>0$, and for any nonnegative function $\varphi \in H_0^1(\Omega)$ such that $\supp(\varphi) \Subset B_{2r}(x_0) \Subset \Omega$, we have
\begin{equation}\label{eq:almss}
\left\langle -\Delta\hat{u}_i -\lambda_{u_i}u_i+ \sum_{j\neq i}\lambda_{u_j}u_j, \varphi \right\rangle \geq -C\left(r^{N-1} + \|\varphi\|_{1}+ r\|\varphi\|_{2}^2 + r\|\nabla\varphi\|^2_2 + r\|\varphi\|_{1}\|\nabla\varphi\|^2_2 + r\|\varphi\|_{2}^2\|\nabla\varphi\|^2_2\right),
\end{equation}
where $C=C(N,k,a, \Omega,\eta)$. 

\item If $B$ is a ball contained in the domain $\Omega$ and $|B \cap \{u_j=0\}|=0$ for every $j=1,\ldots, k$, then 
\[
-\Delta\hat{u}_i -\lambda_{u_i}u_i+ \sum_{j\neq i}\lambda_{u_j}u_j\geq 0  \quad \text{ in } B,\qquad \text{for all }i=1,\dots, k.
\]
\end{enumerate}
\end{proposition}

\begin{proof}
Without loss of generality, we prove the claims for $i=1$.\\

\textbf{(1)} Consider the perturbation 
\begin{equation}\label{eq:deformation}
u_t=\left(u_{1,t}, \ldots, u_{k,t}\right):=\left(\left(\hat{u}_1+t \varphi\right)^{+}, \left(\hat{u}_2-t \varphi \right)^{+}, \ldots, \left(\hat{u}_k-t \varphi\right)^{+}\right) ,
\end{equation}
with $t \in (0,1)$ sufficiently small. It is straightforward to check (see \cite[Lemma A.3]{ASST} for the details) that
\[
u_t\in \overline{H},\qquad \Omega_{u_{1,t}}\subseteq \Omega_{u_1}\cup \Omega_\varphi,\qquad \Omega_{u_{j,t}}\subseteq \Omega_{u_j} \text{ for $j>1$}.
\]
By using the minimality of $u=(u_1, \ldots, u_k)$, we have $J_\eta(u) \leq J_\eta(u_t)$, that is
\begin{align*}
\sum_{i=1}^k\int_{\Omega}|\nabla u_i|^2 + f_\eta\left(\sum_{i=1}^k|\Omega_{u_i}|\right)  \leq & \;\displaystyle \dfrac{\int_{\Omega}|\nabla (\hat{u}_1 + t\varphi)^+|^2}{\|(\hat{u}_1 + t\varphi)^+\|^2_2} + \sum_{i=2}^k\dfrac{\int_{\Omega}|\nabla (\hat{u}_i - t\varphi)^+|^2}{\|(\hat{u}_i - t\varphi)^+\|^2_2}  + f_\eta\left(|\Omega_{u_{1,t}}|+\sum_{i=2}^k|\Omega_{u_{i,t}}|\right).
\end{align*}
Hence, since $f_\eta$ is an increasing function and $|\Omega_{u_{1,t}}|+\sum_{i=2}^k|\Omega_{u_{i,t}}| \leq \sum_{i=1}^k|\Omega_{{u}_i}| + |\Omega_\varphi|$, we estimate
\begin{align*}\label{eq:01}
\sum_{i=1}^k\int_{\Omega}|\nabla u_i|^2  \leq & \;\displaystyle \dfrac{\int_{\Omega}|\nabla (\hat{u}_1 + t\varphi)^+|^2}{\|(\hat{u}_1 + t\varphi)^+\|^2_2} + \sum_{i=2}^k\dfrac{\int_{\Omega}|\nabla (\hat{u}_i - t\varphi)^+|^2}{\|(\hat{u}_i - t\varphi)^+\|^2_2} \\
 & + f_\eta\left(|\Omega_{u_{1,t}}|+\sum_{i=2}^k|\Omega_{u_{i,t}}|\right) - f_\eta\left(\sum_{i=1}^k|\Omega_{u_i}|\right) \\
\leq & \;\displaystyle \dfrac{\int_{\Omega}|\nabla (\hat{u}_1 + t\varphi)^+|^2}{\|(\hat{u}_1 + t\varphi)^+\|^2_2} + \sum_{i=2}^k\dfrac{\int_{\Omega}|\nabla (\hat{u}_i - t\varphi)^+|^2}{\|(\hat{u}_i - t\varphi)^+\|^2_2} \\
 & + f_\eta\left(\sum_{i=1}^k|\Omega_{u_i}| + |\Omega_{\varphi}|\right) - f_\eta\left(\sum_{i=1}^k|\Omega_{u_i}|\right)
 \end{align*}
 Using Lemma \ref{lemma:feta}-(2) to estimate the right hand side, we obtain
 \begin{align*}
\sum_{i=1}^k\int_{\Omega}|\nabla u_i|^2 \leq  \;\displaystyle \dfrac{\int_{\Omega}|\nabla (\hat{u}_1 + t\varphi)^+|^2}{\|(\hat{u}_1 + t\varphi)^+\|^2_2} + \sum_{i=2}^k\dfrac{\int_{\Omega}|\nabla (\hat{u}_i - t\varphi)^+|^2}{\|(\hat{u}_i - t\varphi)^+\|^2_2} + \frac{1}{\eta}|\Omega_{\varphi}|.
\end{align*}
This inequality corresponds precisely to \cite[eq. (3.4)]{ASST}, with the constant $\mu$ therein being replaced by $1/\eta$ in our case. From here, we can repeat word by word the proof of \cite[Proposition 3.2]{ASST} to complete the proof of (1).

\smallbreak

\noindent \textbf{(2)} This proof is adapted from \cite[Proposition 3.4]{ASST}  to our situation. By taking $\varphi\in C^\infty_c(B)$ and considering once again the deformation \eqref{eq:deformation}, we have
\begin{align*}
\sum_{i=1}^{k}|\Omega_{u_{i,t}}|
&= \sum_{i=1}^{k}|\Omega_{u_{i,t}}\cap B | + {\sum}_{i=1}^{k}|\Omega_{u_{i,t}}\cap (\Omega \setminus B) | \leq |B| + \sum_{i=1}^{k}|\Omega_{ {u}_i}\cap (\Omega\setminus B)|\\
&= \sum_{i=1}^{k}|\Omega_{u_i} \cap B|+  \sum_{i=1}^{k}|\Omega_{u_i}\cap (\Omega\setminus B)| = \sum_{i=1}^{k}|\Omega_{u_i}|.
\end{align*}
Then, by minimality, $J_\eta(u_t)\geq J_\eta(u)$, which implies, using Lemma \ref{lemma:expansion_L^2} and since $f_\eta$ is increasing, 
\begin{align*}
\sum_{i=1}^k\int_{\Omega}|\nabla u_i|^2  &\leq  \;\displaystyle \dfrac{\int_{\Omega}|\nabla (\hat{u}_1 + t\varphi)^+|^2}{\|(\hat{u}_1 + t\varphi)^+\|^2_2} + \sum_{i=2}^k\dfrac{\int_{\Omega}|\nabla (\hat{u}_i - t\varphi)^+|^2}{\|(\hat{u}_i - t\varphi)^+\|^2_2} + f_\eta\left(\sum_{i=1}^{k}|\Omega_{u_{i,t}}|\right)-f_\eta\left(\sum_{i=1}^k|\Omega_{u_i}|\right)\\
 \leq&  \displaystyle \dfrac{\int_{\Omega}|\nabla (\hat{u}_1 + t\varphi)^+|^2}{\|(\hat{u}_1 + t\varphi)^+\|^2_2} + \sum_{i=2}^k\dfrac{\int_{\Omega}|\nabla (\hat{u}_i - t\varphi)^+|^2}{\|(\hat{u}_i - t\varphi)^+\|^2_2}\\
 \leq& \int_{\Omega}|\nabla (\hat{u}_1 + t\varphi)^+|^2 \left( 1 - 2t\int_{\Omega}u_i\varphi + o(t) \right) + \sum_{i=2}^k\int_{\Omega}|\nabla (\hat{u}_i - t\varphi)^+|^2\left(1 + 2t\int_{\Omega}u_i\varphi + o(t)\right)\\
 =&\int_{\Omega}\left(|\nabla (\hat{u}_1 + t\varphi)^+|^2+\sum_{j=2}^k |\nabla (\hat{u}_i - t\varphi)^+|^2\right)-2t \int_{\Omega}|\nabla (\hat{u}_1 + t\varphi)^+|^2 \int_\Omega u_1 \varphi \\
        &+ 2t \sum_{i=2}^k\int_{\Omega}|\nabla (\hat{u}_i - t\varphi)^+|^2 \int_\Omega u_i \varphi+ o(t).
\end{align*}
Now, combining this with
\[
\int_\Omega \Big(|\nabla (\hat{u}_1 + t\varphi)^+|^2+\sum_{j=2}^k |\nabla (\hat{u}_i - t\varphi)^+|^2\Big)=\int_\Omega |\nabla (\hat{u}_1+t\varphi)|^2=\sum_{i=1}^k \int_\Omega |\nabla u_i|^2+2t \int_\Omega \nabla \hat{u}_1\cdot \nabla \varphi + o(t),
\]
and
\begin{align*}
\int_{\Omega}|\nabla (\hat{u}_1 + t\varphi)^+|^2 = \int_\Omega |\nabla u_1|^2+o(1)=\lambda_{u_1}+ o(1),\  \int_{\Omega}|\nabla (\hat{u}_i - t\varphi)^+|^2 =  \int_\Omega |\nabla u_i|^2+o(1)=\lambda_{u_i}+ o(1)
\end{align*}
as $t\to 0^+$, we conclude
\[
0\leq 2t\int_\Omega \nabla \hat{u}_1\cdot \nabla \varphi -2t \lambda_{u_1}\int_\Omega u_1\varphi + 2 t \sum_{i=2}^k \lambda_{u_i}\int_\Omega u_i \varphi + o(t).
\]
Dividing the above inequality by $t>0$ and letting $t\to 0^+$ entails the result.
\end{proof}

The following is a general criterion to guarantee that a $k$-uple of functions is Lipschitz continuous.
\begin{theorem}\label{thm:general_Lipschitz}
Consider $(u_1, \ldots, u_k) \in H_0^1(\Omega\,;\,\R^k)$
such that, for each $i = 1, \ldots, k$, $u_i$ is a nonnegative $L^2$-normalized function satisfying
\begin{enumerate}
\item
$
-\Delta u_i=\lambda_1(\Omega_{u_i}) u_i \text{ in } \Omega_{u_i} \quad \text{and}\quad  - \Delta u_i\leq \lambda_1(\Omega_{u_i})u_i \text{ in } \Omega.
$
\item Whenever we take $x_0 \in\Omega$, $r>0$ and any nonnegative function $\varphi \in C^\infty_c(\Omega)$, such that $supp(\varphi) \Subset B_{2r}(x_0) \Subset \Omega$, and $0 \leq \varphi \leq 1$, we have 
\begin{align}
\left\langle -\Delta\hat{u}_i -\lambda_{u_i}u_i+ \sum_{j\neq i}\lambda_{u_j}u_j, \varphi \right\rangle \geq & -C(r^{N-1} + r\|\nabla\varphi\|^2_2), \label{eq:almss20}
\end{align}
where $C$ depends only on $\|u_1\|_{H_0^1(\Omega)}, \ldots, \|u_k\|_{H_0^1(\Omega)}$, $\|u_1\|_\infty, \ldots, \|u_k\|_\infty$ and $N$.
\item If $B \subset \Omega$ is a ball such that $|B \cap \{u_j=0\}|=0$ for every $j=1,\ldots, k$, then 
\[
-\Delta\hat{u}_i -\lambda_{u_i}u_i+ \sum_{j\neq i}\lambda_{u_j}u_j\geq 0  \text{ in } B,\qquad \text{for all }i=1,\dots, k.
\]
Then $u_i$ is locally Lipschitz continuous in $\Omega$ for each $i=1, \ldots, k$, and for all $\Omega'\Subset \Omega$ there exists $C > 0$, that depends on $N, \|u_1\|_{H_0^1(\Omega)}, \ldots, \|u_k\|_{H_0^1(\Omega)},\|u_1\|_\infty, \ldots, \|u_k\|_\infty$ and $\mbox{dist}(\Omega', \Omega)$, such that 
\[
\sup_{x \neq y \in \Omega'}\dfrac{|u_i(x) - u_i(y)|}{|x-y|} \leq C. 
\]
\end{enumerate}    
\end{theorem}

\begin{proof}
The proof follows the same ideas as in \cite{ASST}, in fact, these properties are all it takes to prove Lipschitz continuity just like in \cite[Theorem 1.2]{ASST}. Let us be more precise.
\begin{itemize}
\item Property ${\it (1)}$ corresponds to \cite[Proposition 3.1]{ASST}, ${\it (3)}$ to \cite[Proposition 3.4]{ASST}. On the other hand, (2) is more general than \eqref{eq:almss} (which corresponds to \cite[Proposition 3.2]{ASST}).
\item  We now perform a simple scaling of \eqref{eq:almss20}. Consider $R>0, x_0 \in \Omega$ and a sequence $\left(r_n\right)_{n \in \mathbb{N}}$ such that $B_{r_n R}\left(x_0\right) \Subset \Omega$. For each $i=1, \ldots, k$, define
$$
u_{i, n}(x)={u}_i\left(x_0+r_n x\right).
$$
Then, given a nonnegative $\varphi \in H_0^1(\Omega)$ with $\operatorname{supp}(\varphi) \subseteq B_R$, we have
\begin{equation}\label{eq:scalingineq}
\left\langle-\Delta \hat{u}_{i, n}-r_n^2 \lambda_{\bar{u}_i} \bar{u}_{i, n}+r_n^2 \sum_{j \neq i} \lambda_{\bar{u}_j} \bar{u}_{j, n}, \varphi\right\rangle  \geq-C r_n R\left(R^{N-2}+\|\nabla \varphi\|_2^2\right).
\end{equation}
This plays the role of \cite[Proposition 3.3]{ASST} under the more general assumption \eqref{eq:almss20}.

\item From this, we can prove the continuity of minimizers $(u_1,\ldots, u_k)$ repeating almost word by word the proof of \cite[Proposition 3.8]{ASST}. Indeed, let $x_0 \in \cap_j\{u_j=0\}$; one takes $(x_n)_{n\in \mathbb{N}}\subset \Omega$ such that $x_n \to x_0$, sets $r_n := |x_0 - x_n| \to 0$ and aims to prove that $u_{i,n}(x)=u_i(x_0+r_n x)\to 0=u_i(x_0)$. Property (3) is used in \cite{ASST} to prove this claim under the condition $\left|B_{r_n}\left(x_0\right) \cap\{u_j=0\}\right|=0$ for every $j$. In the complementary situation, one uses Property ${\it (1)}$ and \eqref{eq:scalingineq} to show that the limit of the scaling $u_{i,n}$ is superharmonic, and then uses \eqref{eq:almss20} for test functions satisfying $\|\nabla\varphi\|_{\infty} \leq C_1/r$ ($C_1$ does not depend on $r$) which implies
\[
\left\langle -\Delta\hat{u}_i -\lambda_{u_i}u_i+ \sum_{j\neq i}\lambda_{u_j}u_j, \varphi \right\rangle \geq  -Cr^{N-1}
\]
with this new constant $C$ also depending on $C_1$. These are the necessary ingredients for the proof of \cite[Proposition 3.8]{ASST} to work.
\item To prove the Lipschitz continuity of each $u_i$, it is enough to consider $\Sigma: \bar{\Omega} \times (0, \infty) \to \mathbb{R}$ defined as
\[
\Sigma(x,r) := \dfrac{1}{r^N}\int_{B_r(x)}|\nabla U|^2, \quad \text{for} \quad (x,r) \in \bar{\Omega}\times (0,\infty).
\]
and show that $\Sigma$ is bounded over $\Omega'\times(0,\infty)$, for every $\Omega'$ compactly contained in $\Omega$. This is proved by a contradiction argument; with the continuity of each $u_i$ and the properties previously recalled, we can now follow word by word the arguments from \cite[Section 4]{ASST}.\qedhere
\end{itemize}
\end{proof}

\begin{corollary}\label{coro_Lispchitz}
Let $(u_1,\dots,u_k)$ be a minimizer for the functional $J_\eta$. Then each $u_i$ is locally Lipschitz continuous in $\Omega$.
\end{corollary}
\begin{proof} From Lemma \ref{lemma:otherproperties} and Proposition \ref{prop:classS}, the minimizers of $J_\eta$ satisfy the conditions of Theorem \ref{thm:general_Lipschitz}, hence we have in particular that any minimizer of the functional $J_\eta$ is locally Lipschitz continuous.
\end{proof}

Next, we show that the nodal sets of a vector  $(u_1,\ldots, u_k)\in \overline H$ that achieves \eqref{eq:ceta} have measure smaller than or equal to the constraint $a$ of the original problem, so that $f_\eta\left(\sum_{i=1}^k|\Omega_{u_i}|\right)=\eta \sum_{i=1}^k|\Omega_{u_i}|$.

\begin{lemma}[Optimal partitions have measure smaller than or equal to $a$]\label{lem:measmalla}
There exists $\bar \eta=\bar \eta(N,k,a,\Omega)\in (0,1]$ such that, given $\eta\in (0,\bar \eta)$ and  $(u_{1}, \ldots, u_{k}) \in \overline{H}$  a nonnegative $L^2$-normalized minimizer achieving $c_\eta$, then it holds
\[
\sum_{i=1}^k |{\Omega}_{u_{i}}| \leq a.
\]
\end{lemma}

\begin{proof}
We argue by contradiction. Suppose that $\sum_{i=1}^k |{\Omega}_{u_{i}}| > a$ for some $L^2$-normalized minimizer $(u_1,\ldots, u_k)$, and consider the perturbation 
\[
(u_1^t, \ldots, u_k^t) := ((u_{1} - t)^{+}, \ldots, (u_{k} - t)^{+}), 
\]
for $t>0$ small. Since $(u_{1},\ldots, u_{k}) \in \overline{H}$, it is clear that also $(u_1^t,\ldots, u_k^t) \in \overline{H}$ for $t>0$. Moreover, for $t$ sufficiently small, we have $\sum_{i=1}^k |{\Omega}_{u_i^t}| > a$. 

Since $(u_{1},\ldots, u_{k})$ is a minimizer of $J_{\eta}$ on $\overline H$, we have 
\[
J_{\eta}(u_{1}, \ldots, u_{k}) \leq J_{\eta}(u_1^t,\ldots, u_k^t).
\]
Hence,
\begin{align*}
\sum_{i=1}^k\int_{\Omega}|\nabla u_{i}|^2 + f_\eta\left(\sum_{i=1}^k|{\Omega}_{u_{i}}|\right) \leq  \sum_{i=1}^k\frac{\int_{\Omega} |\nabla (u_{i} - t)^{+}|^2}{\int_{\Omega} |(u_{i} - t)^{+}|^2} + f_\eta\left(\sum_{i=1}^k|{\Omega}_{u_i^t}|\right),
\end{align*}                                                                              
and, by using the definition of $f_\eta$ and the contradiction assumption, 
\begin{align*}
\sum_{i=1}^k\int_{\Omega}|\nabla u_{i}|^2 + \frac{1}{\eta}\left(\sum_{i=1}^k|{\Omega}_{u_{i}}| - a\right) \leq  \sum_{i=1}^k\frac{\int_{\Omega} |\nabla (u_{i} - t)^{+}|^2}{\int_{\Omega} |(u_{i} - t)^{+}|^2} + \frac{1}{\eta}\left(\sum_{i=1}^k|{\Omega}_{u_i^t}| - a\right).
\end{align*}                                                                              
From this, we can follow the proof of \cite[Proposition 2.3]{ASST} (where, instead of $\frac{1}{\eta}$, one has $\mu$), and conclude 
\[
\frac{1}{\sqrt{\eta}} \leq \dfrac{2^{(k-1)/2}\,c_\eta}{N | B_1|^{1/N}}\left(\sum_{i=1}^k|\Omega_{u_{i}}|\right)^{\frac{2-N}{2N}} \leq \dfrac{2^{(k-1)/2}\, C}{N | B_1|^{1/N}}a^{\frac{2-N}{2N}},
\]  
where $C$ is the constant of the upper bound for $c_\eta$ from Lemma \ref{lemma:bounds_ceta}. We have found a contradiction up to  take $\eta$ sufficiently small, so the claim is proved.
\end{proof}

We conclude with the following observation.

\begin{remark}\label{rem:connected} 
Let $(u_{1},\ldots, u_{k})\in \overline H$ be an $L^2$-normalized minimizer for $J_\eta$, attaining $c_\eta$. Then each $\Omega_{u_i}$ is connected (by minimality and from the fact that all sets are open). 
Moreover, for each $i=1,\ldots, k$, defining
\[
D_i:= \Omega\setminus \bigcup\limits_{j\neq i} \overline \Omega_{u_j}.
\]
the function $u_i$ is a solution to the (one-component) minimization problem
\begin{equation}\label{eq:pbuiconstr}
\min\left\{ \frac{\int_{\Omega}|\nabla v|^2}{\int_\Omega v^2}+f_\eta\Big(|\cup_{j\not =i}\Omega_{u_j}\cup \Omega_v|\Big) : v\in H^1_0(D_i)\setminus \{0\}\right\}.   
\end{equation}

Observe that
\[
f_\eta\Big(|\cup_{j\not =i}\Omega_{u_j}\cup \Omega_v|\Big)=g_\eta(|\Omega_v|),
\]
where
\[
g_\eta(s)=\begin{cases}
    \eta(s+|\cup_{j\neq i} \Omega_j|-a) &\text{ if } s\geq a-|\cup_{j\neq i} \Omega_j|\\
    \frac{1}{\eta}(s+|\cup_{j\neq i} \Omega_j|-a) &\text{ if } s\leq a-|\cup_{j\neq i} \Omega_j|,\\
\end{cases}
\]
 and we note that $g_\eta$ satisfies the same properties of $f_\eta$, see Lemma~\ref{lemma:feta}. Then \eqref{eq:pbuiconstr} is equivalent to
\begin{equation}\label{eq:pbuiconstr2}
\min\left\{\frac{\int_{D_i}|\nabla v|^2}{\int_{D_i} v^2}+g_\eta(| \Omega_v|) : v\in H^1_0(D_i)\right\}.
\end{equation}
Summarizing, $u_i$ is a minimizer for~\eqref{eq:pbuiconstr2}, and the interior one-phase points of $\partial \{u_i>0\}$ for problem \eqref{eq:ceta} are the interior (one-phase) free boundary points when looking at $u_i$  as a minimizer of \eqref{eq:pbuiconstr2}. Therefore, the interior one-phase points of $\partial\{u_i>0\}$ can be treated with tools from the regularity theory for one-phase free boundaries~\cite{AguileraAltCaffarelli,AltCaffarelli,BrianconLamboley,RussTreyVelichkov,Wagner}.
\end{remark}

\subsection{Nondegeneracy of eigenfunctions and density estimate from below}

In this subsection, we establish more qualitative properties for the minimizers $(u_1,\dots, u_k)$ of $J_\eta$ in ${\overline H}$, i.e., elements in $\overline{H}$ that achieve $c_\eta$. We prove a nondegeneracy condition which implies, in particular, the fact that $\Omega_{u_i}$ has the finite perimeter. We begin with an estimate from below, which was first proposed in~\cite{AltCaffarelli} for the harmonic case.

\begin{lemma}\label{le:nondegeneracy}
Given $\eta\in (0,1]$, let $(u_1,\dots,u_k)$ be an $L^2$-normalized minimizer for the problem $c_\eta$. 
For every $\theta\in (0,1)$, there are positive constants $K_0,\rho_0$ depending only on $\theta$,  $N$, $k$, $a$, $\Omega$ and $\eta$, such that the following assertion holds: if $\rho\leq \rho_0$ and $x_0\in \overline\Omega$, then for all $i=1,\dots,k$ we have
 \begin{equation}\label{eq:nondegmean}
 \dashint_{\partial B_\rho(x_0)\cap \Omega}{u_i\,d\mathcal H^{N-1}}\leq K_0 \rho\,\,\,\,
 \Longrightarrow \,\,\,\,u_i\equiv0 \text{ in }\,\,B_{\theta\rho}(x_0)\cap \Omega.
 \end{equation}
 \end{lemma}
\begin{proof} 
Without loss of generality, we fix $x_0=0$; we also consider $u_i$ to be defined in the whole $\R^N$, being extended to 0 in the complementary of $\Omega$. 
\smallbreak

\noindent \textbf{Step 1.} (building a competitor) By Lemmas \ref{lemma:unifbound_ceta} and \ref{lemma:otherproperties}, there exist universal constants $\gamma,M$ (depending only on $N$, $k$, $a$, $\Omega$)  such that 
\begin{equation}\label{eq:universalestimates_nondeg}
-\Delta u_i\leq \gamma\quad \text{ in } \R^N \text{ in the distributional sense,} \quad \|u_i\|_{L^\infty(\Omega)}\leq \gamma,\quad  \|u_i\|^2_{H^1_0(\Omega)}\leq M
\end{equation}
for every $i=1,\ldots, k$. Then, the function 
\[
 x\mapsto u_i(x)+\gamma\frac{|x|^2-\rho^2}{2N}
 \]
 is subharmonic in $B_\rho$.

 By Lemma \ref{lemma:kind_of_meanvalueth} and using the assumption in \eqref{eq:nondegmean}, for every $\theta\in(0,1)$ there exists $c=c(\theta,N)$ such that 
 \begin{align}
 \delta_\rho&:=\sup_{B_{\sqrt{\theta}\rho\cap \Omega}} u_i \leq \sup_{B_{\sqrt{\theta}\rho}}(u_i+\gamma\frac{|x|^2-\rho^2}{2N})+\gamma\frac{\rho^2(1-\theta)}{2N} \nonumber \\ 
            &\leq c\left(\dashint_{\partial B_\rho\cap \Omega}{u_i\,d\mathcal H^{N-1}}+\gamma \rho^2\right)\leq c(K_0\rho+\gamma\rho^2). \label{eq:4.20}
 \end{align}
Let us consider the nonnegative solution $w$ of \begin{equation}\label{www}
 \begin{cases}
 -\Delta w=M\gamma,\qquad &\text{in }B_{\sqrt{\theta}\rho}\setminus B_{\theta \rho},\\
 w=\delta_\rho,\qquad &\text{on }\partial B_{\sqrt{\theta}\rho},\\
 w=0,\qquad &\text{on }B_{\theta\rho}.
 \end{cases}
 \end{equation}
Observe that the function $w$ is explicit in the annulus $B_{\sqrt{\theta}\rho}\setminus B_{\theta \rho}$:
\[
w(x)=\frac{M\gamma}{2N}((\theta\rho)^2-|x|^2)+
\left(\delta_\rho + \frac{M\gamma ((\sqrt{\theta}\rho )^2-(\theta\rho)^2)}{2N}\right)\frac{(\theta^{3/2}\rho^2)^{N-2}}{(\sqrt{\theta}\rho)^{N-2}-(\theta\rho)^{N-2}}\left(\frac{1}{(\theta \rho)^{N-2}}-\frac{1}{|x|^{N-2}}\right),
\]
if $N\geq 3$, while
\[
w(x)=\frac{\gamma}{2N}((\theta \rho)^2-|x|^2)+\frac{\delta_\rho+\frac{\gamma }{2N}((\sqrt{\theta}\rho)^2-(\theta \rho)^2)}{\log (1/\sqrt{\theta})}(\log |x|-\log (\theta \rho))
\]
 if $N=2$. By definition, $w\geq u_i$ on $\partial B_{\sqrt{\theta}\rho}$, therefore the function 
 \[
 v:=
 \begin{cases}
 u_i,\qquad &\text{in }{\R^N}\setminus B_{\sqrt{\theta}\rho},\\
 \min\{u_i,w\},\qquad &\text{in }B_{\sqrt{\theta}\rho},
 \end{cases}
 \]
 satisfies $v=w=0$ in $B_{\theta \rho}$, {$v=u_i=0$ in $\R^N\setminus \Omega$ } and observe that
 \[
 v\leq u_i,\qquad \Omega_{v}\subset \Omega_{u_i},\qquad \Omega_{v}\setminus B_{\sqrt{\theta}\rho}=\Omega_{u_i}\setminus B_{\sqrt{\theta}\rho}.
 \]
 \smallbreak

 \noindent \textbf{Step 2.} (Comparing energies)
Since $v\in H^1_0(\Omega)$ and from the properties listed above, observe that $(u_1,\dots, u_{i-1},{v|_\Omega},u_{i+1},\dots, u_k)\in \overline{H}$, which we can test against the minimality of $(u_1,\dots, u_k)$. At the level of the corresponding nodal sets, this can be thought as an inner variation of the optimizer, giving
\begin{equation}\label{eq:e1}
\begin{split}
\int_{\Omega} |\nabla u_i|^2+f_\eta(\sum_j|\Omega_{u_j}|)\leq \frac{\displaystyle \int_{\Omega} |\nabla v|^2}{\displaystyle \int_\Omega v^2}+f_\eta(\sum_{j\not=i}|\Omega_{u_j}|+ |\Omega_{v}|)
\end{split}
\end{equation}
or, equivalently,
\begin{equation}\label{eq:e1.co}
\int_\Omega v^2 \left[\int_{\Omega} |\nabla u_i|^2+f_\eta(\sum_j|\Omega_{u_j}|)-f_\eta(\sum_{j\not=i}|\Omega_{u_j}| +|\Omega_{v}|)\right]\leq \int_{\Omega} |\nabla v|^2.    
\end{equation}
As $v=0$ in $B_{\theta\rho}$, using also the Lipschitz-like property~ Lemma \ref{lemma:feta}-(2) of $f_\eta$, we obtain 
 \[
 \begin{split}
 \eta|\Omega_{u_i}\cap B_{\theta\rho}|&=\eta|(\Omega_{u_i}\setminus \Omega_v)\cap B_{\theta \rho}|\leq \eta|(\Omega_{u_i}\setminus \Omega_{v})\cap B_{\sqrt{\theta}\rho}|=\eta \left(|\Omega_{u_i}|-|\Omega_{v}|  \right)\\
 & \leq f_\eta(\sum_j|\Omega_{u_j}|)-f_\eta(\sum_{j\not= i}|\Omega_{u_j}|+|\Omega_{v}|).
 \end{split}
 \]
On the other hand, notice that, since $v=u_i$ in $\Omega\setminus B_{\sqrt{\theta}\rho}$ and $v=u_i=0$ in $\R^N\setminus \Omega$,
\[
\begin{split}
\int_\Omega v^2&= \int_\Omega u_i^2+ \int_\Omega (v^2-u_i^2){=1-\int_{B_{\sqrt{\theta}\rho}\cap \Omega}(u_i^2-v^2)}=1-\int_{B_{\sqrt{\theta}\rho}}(u_i^2-v^2).
\end{split}
\]
Putting the information above together into~\eqref{eq:e1.co}, we have 
 \[
\left(1-\int_{B_{\sqrt{\theta}\rho}}(u_i^2-v^2)\right)\left(\int_{\Omega} |\nabla u_i|^2+\eta|\Omega_{u_i}\cap B_{\theta\rho}|\right)\leq \int_{\Omega} |\nabla v|^2.
 \]
Moreover, since $v=0$ in $B_{\theta \rho}$,
\begin{align*}
 \int_{B_{\sqrt{\theta}\rho}} |\nabla u_i|^2+\eta\left(1-\int_{B_{\sqrt{\theta}\rho}}(u_i^2-v^2)\right)|\Omega_{u_i}\cap B_{\theta\rho}| & \leq \int_{B_{\sqrt{\theta}\rho}} |\nabla v|^2+ \int_{B_{\sqrt{\theta}\rho}}(u_i^2-v^2)\int_\Omega |\nabla u_i|^2\\
 &= \int_{B_{\sqrt{\theta}\rho}\setminus B_{\theta \rho}} |\nabla v|^2+ \int_{B_{\sqrt{\theta}\rho}}(u_i^2-v^2)\int_\Omega |\nabla u_i|^2.
 \end{align*}
Clearly, by \eqref{eq:universalestimates_nondeg} and by taking $\rho_0$ sufficiently small, 
\[
0\leq \int_{B_{\sqrt{\theta}\rho}}(u_i^2-v^2)\leq  \|u_i\|^2_{L^\infty(\Omega)}|B_{\rho_0}|\leq \gamma^2 \rho_0^N |B_1| \leq\frac12.
\]
Recalling also the definition of $M>0$ from \eqref{eq:universalestimates_nondeg},
\[
 \begin{split}
 \int_{B_{\sqrt{\theta}\rho}} |\nabla u_i|^2+\frac{\eta}{2}|\Omega_{u_i}\cap B_{\theta\rho}|\leq \int_{B_{\sqrt{\theta}\rho}\setminus B_{\theta\rho}} |\nabla v|^2+ M\int_{B_{\sqrt{\theta}\rho}}(u_i^2-v^2).
 \end{split}
\]

 Thanks to the inequalities above, we infer 
 \begin{align}
 \int_{B_{\theta\rho}}|\nabla u_i|^2+\frac{\eta}{2}|\Omega_{u_i}\cap B_{\theta\rho}|&=\int_{B_{\sqrt{\theta}\rho}}|\nabla u_i|^2+\frac{\eta}{2}|\Omega_{u_i}\cap B_{\theta\rho}|-\int_{B_{\sqrt{\theta}\rho}\setminus B_{\theta \rho}}|\nabla u_i|^2 \nonumber\\
 &\leq \int_{B_{\sqrt{\theta}\rho}\setminus B_{\theta\rho}} (|\nabla v|^2-|\nabla u_i|^2)+ M\int_{B_{\sqrt{\theta}\rho}}(u_i^2-v^2).\label{eq:4.21}
 \end{align} 

 \noindent \textbf{Step 3.} (estimating the right-hand-side of \eqref{eq:4.21}) Now we estimate both terms of the right-hand-side of \eqref{eq:4.21}. From the definition of $v$, we obtain
\begin{align}
\int_{B_{\sqrt{\theta}\rho}\setminus B_{\theta \rho}} (|\nabla v|^2-|\nabla u_i|^2)&=
\int_{(B_{\sqrt{\theta}\rho}\setminus B_{\theta \rho})\cap\{u_i\leq w\}} (|\nabla u_i|^2-|\nabla u_i|^2)+\int_{(B_{\sqrt{\theta}\rho}\setminus B_{\theta \rho})\cap \{u_i>w\}} (|\nabla w|^2-|\nabla u_i|^2)\nonumber\\
&=\int_{(B_{\sqrt{\theta}\rho}\setminus B_{\theta \rho})\cap \{u_i>w\}} (|\nabla w|^2-|\nabla u_i|^2) \nonumber\\
&\leq 2\int_{(B_{\sqrt{\theta}\rho}\setminus B_{\theta\rho})\cap\{u_i>w\}}(|\nabla w|^2-\nabla u_i\cdot\nabla w). \label{eq:nondeg_aux1}
\end{align}
Moreover, since $0\leq u_i+w\leq 2\|u_i\|_{L^\infty(\Omega)}\leq 2\gamma$ in $(B_{\sqrt{\theta}\rho}\setminus B_{\theta \rho})\cap \{u_i>w\}$,
\begin{align}
\int_{B_{\sqrt{\theta}\rho}} (u_i^2-v^2)&=\int_{B_{\sqrt{\theta}\rho}\setminus B_{\theta \rho}} (u_i^2-v^2)+ \int_{B_{\theta \rho}} (u_i^2-v^2) =\int_{(B_{\sqrt{\theta}\rho}\setminus B_{\theta \rho})\cap \{u_i>w\} } (u_i^2-w^2) + \int_{B_{\theta \rho}} u_i^2\nonumber \\
            &\leq 2\gamma \int_{(B_{\sqrt{\theta}\rho}\setminus B_{\theta \rho})\cap \{u_i>w\} } (u_i-w) + \int_{B_{\theta \rho}} u_i^2. \label{eq:nondeg_aux2}
\end{align}
 By combining \eqref{eq:nondeg_aux1} and \eqref{eq:nondeg_aux2} with \eqref{eq:4.21}, we conclude that
\begin{align}
\int_{B_{\theta\rho}}|\nabla u_i|^2+\frac{\eta}{2}|\Omega_{u_i}\cap B_{\theta\rho}|\leq & 2\int_{(B_{\sqrt{\theta}\rho}\setminus B_{\theta\rho})\cap\{u_i>w\}}(|\nabla w|^2-\nabla u_i\cdot\nabla w)\nonumber \\
        &+2 M \gamma\int_{(B_{\sqrt{\theta}\rho}\setminus B_{\theta \rho})\cap \{u_i>w\}}(u_i-w)+M\int_{B_{\theta \rho}}u_i^2. \label{eq:4.22.2}
\end{align}

On the other hand, testing \eqref{www} with $(u_i-w)_+$ and integrating over ${B_{\sqrt{\theta}\rho}\setminus B_{\theta\rho}}$, we obtain 
\[
\int_{B_{\sqrt{\theta}\rho}\setminus B_{\theta\rho}} \nabla w \cdot \nabla (u_i-w)_+ {+} \int_{\partial B_{\sqrt{\theta}\rho}} \frac{\partial w}{\partial \nu} (u_i-w)_+ \,d\mathcal H^{N-1}{-} \int_{\partial B_{{\theta}\rho}} \frac{\partial w}{\partial \nu} (u_i-w)_+\,d\mathcal H^{N-1} =M\gamma \int_{B_{\sqrt{\theta}\rho}\setminus B_{\theta\rho}} (u_i-w)_+,
\]
where $\nu$ stands for the inner unit normal with respect to $B_{\sqrt{\theta}\rho}$ and $B_{{\theta}\rho}$, respectively. After multiplying this identity by a factor 2, and since $u_i\leq \delta_\rho=w$ on $\partial B_{\sqrt{\theta}\rho}$ and $w=0$ on $\partial B_{\theta \rho}$, we have 
 \begin{equation}\label{eq:4.22}
2\int_{(B_{\sqrt{\theta}\rho}\setminus B_{\theta\rho})\cap \{u_i>w\}}(|\nabla w|^2-\nabla u_i\cdot\nabla w)+2M\gamma\int_{(B_{\sqrt{\theta}\rho}\setminus B_{\theta \rho})\cap \{u_i>w\}}(u_i-w)=-2\int_{\partial B_{\theta\rho}}\frac{\partial w}{\partial \nu}u_i\,d\mathcal H^{N-1},
 \end{equation}
 where, in the last integral, $\nu$ denotes the inner unit normal to $B_{\theta\rho}$.

Recalling the explicit expression of the torsion function on an annulus (see also ~\cite[Equations (4.22)-(4.23)]{BrascoDePhilippisVelichkov}), with a direct computation one obtains, for $x\in \partial B_{\theta \rho}$,
\[
\left|\frac{\partial w}{\partial \nu}(x)\right|=|w'(\theta \rho)|\leq \begin{cases}
    \frac{M\gamma\theta \rho}{N}+\frac{(2N\delta_\rho+ M\gamma\theta \rho^2)(N-2)\theta^\frac{N-4}{2}}{2N\rho (\theta^\frac{N-2}{2}-\theta^{N-2})} & \text{ if } N\geq 3,\\
    \frac{\gamma\theta \rho}{N}+\frac{2N\delta_\rho+\gamma \theta \rho^2}{2N \log(1/\sqrt{\theta})\rho} & \text{ if } N\geq 2.
\end{cases}
\]
In either case, recalling that $\gamma$ and $M$ are universal constants, there exists $\beta=\beta(N,\theta)$ such that
\begin{equation}\label{eq:4.22.3}
 2\left|\frac{\partial  w}{\partial \nu}\right|\leq \beta\frac{\delta_\rho+\rho^2}{\rho}\qquad \text{on }\partial B_{\theta\rho}.
 \end{equation}

 We can now combine~\eqref{eq:4.22.2},  \eqref{eq:4.22} and \eqref{eq:4.22.3} to obtain
 \begin{equation}\label{eq:4.23}
 \int_{B_{\theta\rho}}|\nabla u_i|^2+\frac{\eta}{2}|\Omega_{u_i}\cap B_{\theta\rho}|\leq \beta\frac{\delta_\rho+\rho^2}{\rho}\int_{\partial B_{\theta\rho}}u_i\,d\mathcal H^{N-1}+M\int_{B_{\theta\rho}}u_i^2.
 \end{equation}

By definition of $\delta_p$, we have the estimate
\begin{equation}\label{eq:estimate_aux_square}
 M\int_{B_{\theta\rho}}u_i^2=M\int_{B_{\theta\rho}\cap\Omega_{u_i}}u_i^2\leq M\delta_\rho^2|\Omega_{u_i}\cap B_{\theta \rho}|.
 \end{equation}
Moreover, using the definition of $\delta_\rho$, the trace inequality in $W^{1,1}$ and the Young inequality we obtain, 
 \begin{align}
 \int_{\partial B_{\theta\rho}}u_i\,d\mathcal H^{N-1}&\leq \alpha(N,\theta)\left(\frac{1}{\rho}\int_{B_{\theta\rho}}u_i+\int_{B_{\theta\rho}}|\nabla u_i|\right) \nonumber \\
        &\leq  \alpha(N,\theta)\left( \frac{1}{\rho}\int_{B_{\theta\rho}\cap \{u_i>0\}}u_i+\int_{B_{\theta\rho}\cap \{u_i>0\}}\left(\frac{|\nabla u_i|^2}{2}+\frac{1}{2} \right)\right) \nonumber \\
 &\leq \alpha(N,\theta)\left(\left(\frac{\delta_\rho}{\rho}+\frac12\right)|\Omega_{u_i}\cap B_{\theta\rho}|+\frac12\int_{B_{\theta\rho}}|\nabla u_i|^2\right),\label{eq:estimate_aux_square2}
 \end{align}
 for some $\alpha=\alpha(N,\theta)>0$. 
 By using \eqref{eq:estimate_aux_square} and \eqref{eq:estimate_aux_square2} in \eqref{eq:4.23}, and recalling again~\eqref{eq:4.20}, we have, for all $\rho\leq \rho_0$ \[
 \begin{split}
 &\int_{B_{\theta\rho}}|\nabla u_i|^2+\frac{\eta}{2}|\Omega_{u_i}\cap B_{\theta\rho}|\\
 &\leq \beta\frac{\delta_\rho+\rho^2}{\rho}\int_{\partial B_{\theta\rho}}u_i\,d\mathcal H^{N-1}+\delta_\rho^2 M|\Omega_{u_i}\cap B_{\theta\rho}|\\
 &\leq \beta(c(K_0+{\gamma} \rho)+\rho)\int_{\partial B_{\theta\rho}}u_i\,d\mathcal H^{N-1}+c^2(K_0\rho+{\gamma}\rho^2)^2M|\Omega_{u_i}\cap B_{\theta\rho}|\\
 &\leq \beta\alpha(c(K_0+{\gamma}\rho)+\rho)\left[ \left(\frac{\delta_\rho}{\rho}+\frac12\right)|\Omega_{u_i}\cap B_{\theta\rho}|+\frac12\int_{B_{\theta\rho}}|\nabla u_i|^2\right]\\
 &\hspace{50pt}+c^2(K_0\rho+{\gamma}\rho^2)^2M|\Omega_{u_i}\cap B_{\theta\rho}|\\
 &\leq \beta\alpha(c(K_0+{\gamma}\rho)+\rho)\left[ \left(c(K_0+\gamma \rho)+\frac12\right)|\Omega_{u_i}\cap B_{\theta\rho}|+\frac12\int_{B_{\theta\rho}}|\nabla u_i|^2\right]\\
 &\hspace{50pt}+c^2(K_0\rho+{\gamma}\rho^2)^2M|\{u_i>0\}\cap B_{\theta\rho}|\\
 &\leq \frac{\beta \alpha}{2}(c(K_0+\gamma \rho_0)+\rho_0)\int_{B_{\theta \rho}}|\nabla u_i|^2\\
 &\hspace{50pt}+\left(\beta \alpha (c(K_0+\gamma \rho_0+\rho_0))(c(K_0+\gamma \rho_0)+\frac{1}{2})+c^2(K_0\rho_0+\gamma \rho_0^2)^2M\right) |\Omega_{u_i}\cap B_{\theta \rho}|.
 \end{split}
 \]
Then, by choosing $K_0,\rho_0$ sufficiently small (and depending on $\eta$) so that
 \[
 \frac{\beta \alpha}{2}(c(K_0+\gamma \rho_0)+\rho_0)<1,\quad \left(\beta \alpha (c(K_0+\gamma \rho_0+\rho_0))(c(K_0+\gamma \rho_0)+\frac{1}{2})+c^2(K_0\rho+\gamma \rho_0^2)^2M\right)<\frac{\eta}{2},
 \]
 we conclude that $u_i\equiv 0$ in $B_{\theta\rho}$, for all $\rho\leq\rho_0$.
 \end{proof}

 \begin{remark}\label{rmk:nondeg}
The results of Lemma~\ref{le:nondegeneracy} and, in particular, equation~\eqref{eq:nondegmean}, can also be written  in the following equivalent ways:\[
 \dashint_{ B_\rho(x_0)\cap \Omega}{u_i\,d\mathcal H^{N-1}}\leq K_0 \rho\,\,\,\,
 \Longrightarrow \,\,\,\,u_i\equiv0 \text{ in }\,\,B_{\theta\rho}(x_0)\cap \Omega,
 \]
 or \begin{equation}\label{eq:nondeg_sup}
 \sup_{B_{\rho}(x_0)\cap \Omega}u_i\leq K_0 \rho\,\,\,\,
 \Longrightarrow \,\,\,\,u_i\equiv0 \text{ in }\,\,B_{\theta\rho}(x_0)\cap \Omega,
 \end{equation}
 see for example~\cite[Remark~2.8]{MazzoleniTerraciniVelichkov}.
 They imply there is a universal constant $C>0$ such that, if $x_0\in \overline \Omega_{u_i}$,  $\sup_{B_\rho(x_0)\cap \Omega}u_i\geq C\rho$.
 Moreover, equation~\eqref{eq:nondegmean} can be also written in the following counterpositive way: 
 \begin{equation}\label{eq:nondegmean2}
B_{\theta \rho}(x_0)\cap \Omega_{u_i}\not=\emptyset \,\,\,\,
 \Longrightarrow \,\,\,\,\dashint_{ B_\rho(x_0)\cap \Omega}{u_i\,d\mathcal H^{N-1}}\geq K_0 \rho.
 \end{equation}
 \end{remark}

 An immediate and fundamental consequence of Lemma~\ref{le:nondegeneracy} and of the Lipschitz continuity of the minimizers $(u_1,\ldots, u_k)$ of $J_\eta$ on ${\overline{H}}$ is the following two-sided density estimate on each $\Omega_{u_i}$.

\begin{lemma}\label{le:densityest}
Given $\eta\in (0,1]$, let $(u_1,\dots,u_k)$ be an $L^2$--normalized minimizer for the problem $c_\eta$, and take $\rho_0$ and $K_0$ from Lemma~\ref{le:nondegeneracy}.
There exists $\rho_1=\rho_1(N,k,a,\Omega,\eta)\in(0,\rho_0)$,  $\xi>0$ (depending on $N,k,a,\Omega$, on the Lipschitz constant of $(u_1,\dots, u_k)$, and on $K_0$ - thus, in particular, on $\eta$) such that, for every $i=1,\dots, k$, 
\begin{equation}\label{eq:densityest}
\xi\leq \frac{|\Omega_{u_i}\cap B_\rho(x_0)|}{|B_\rho|}\leq 1-\xi \qquad \text{whenever  $x_0\in \partial \Omega_{u_i}\cap {\Omega}$, \;$\rho\leq \rho_1$}.
\end{equation}
In particular, 
\begin{equation}\label{eq:measHN-1}
\mathcal H^{N-1}(\partial \Omega_{u_i}\setminus\partial^*\Omega_{u_i})=0.
\end{equation}
\end{lemma}
\begin{proof}
{\bf Step 1.}
Let us start from the bound from below in~\eqref{eq:densityest}. The proof is very similar to the case of harmonic functions treated in~\cite[Lemma~5.1]{Velichkov_onephasebook}.
 The nondegeneracy
condition \eqref{eq:nondeg_sup} from Remark~\ref{rmk:nondeg} implies that 
\[
\|u_i\|_{L^\infty(B_{\rho/2}(x_0))}\geq K_0 \tfrac{\rho}{2}\qquad \text{ for every } \rho\leq 2\rho_0.
\] 
Thus, there is a point $y\in \overline B_{\rho/2}(x_0)$ such that $u_i(y)\geq K_0\tfrac{\rho}{2}$.
On the other hand, by letting $L$ be a Lipschitz constant for $u_i$, then $u_i>0$ in the ball $B_{\tfrac{\rho}{2}\min\{1,\tfrac{K_0}{L}\}}(y)\subset B_\rho(x_0)$. In conclusion,
\[
\frac{|\Omega_{u_i}\cap B_\rho(x_0)|}{|B_\rho(x_0)|}\geq \frac{|B_{\tfrac{\rho}{2}\min\{1,\tfrac{K_0}{L}\}}(y)|}{|B_\rho(x_0)|}=\left(\tfrac{1}{2}\min\{1,\tfrac{K_0}{L}\}\right)^N:=\xi.
\]

\noindent {\bf Step 2.} Upper bound of~\eqref{eq:densityest} at one-phase points. This can be obtained as in \cite{AltCaffarelli}, see also~\cite[Lemma~5.1]{Velichkov_onephasebook}, with a few modifications. 
We adapt these ideas to our situation, where the functions are not harmonic but, instead, are eigenfunctions (see also the related \cite[Theorem 5.4]{Wagner}). Precisely, let $x_0=0\in \Big(\partial\Omega_{u_i}\setminus \cup_{j\not=i}\partial \Omega_{u_j}\Big)\cap \Omega$ and consider the function $h$ solution of
\[
\begin{cases}
-\Delta h =\gamma \quad&\text{in }B_\rho,\\
h = u_i\quad&\text{in~} \Omega\setminus B_\rho,
\end{cases}
\]
where $\gamma$ is the constant of~\eqref{eq:universalestimates_nondeg}.
As a consequence, we obtain that $-\Delta (h-u_i)\geq 0$ in $B_\rho$. In particular, we
have that $u_i \leq h$ and $\{u_i>0\}\subset \{h>0\}$ in $B_\rho$.
Moreover, since $h$ is the torsion
function multiplied by $\gamma$ with boundary datum $u_i$,
recalling that $u_i(0)=0$ and that $u_i$ is Lipschitz continuous with constant $L=L(N,k,a,\Omega, \eta)$, see Theorem~\ref{thm:general_Lipschitz}, we deduce by a comparison argument that
\begin{equation}\label{eq:Linftyh}
\|h\|_{L^\infty(B_\rho)}\leq C_h \rho, 
\end{equation}
with a constant $C_h=C_h(N,k,a,\Omega,\eta)$.
Testing the optimality of $(u_1,\dots, u_k)$ with $(u_1,\dots, u_{i-1},h,u_{i+1},\dots, u_k)$, using also~\eqref{eq:propfeps} and an integration by parts, we have\begin{align*}
\frac1\eta|B_\rho \cap \{u_i=0\}| &\geq
\int_{B_\rho}|\nabla u_i|^2\,dx-\int_{B_\rho}|\nabla h|^2\,dx =
\int_{B_\rho}|\nabla (u_i-h)|^2\,dx +2 \int_{B_\rho} \Big(\nabla h\cdot \nabla (u_i-h)\Big)\, dx\\
&=\int_{B_\rho}|\nabla (u_i-h)|^2\,dx +2 \int_{B_\rho} (-\Delta h) (u_i-h)\, dx+2\int_{\partial B_\rho}(u_i-h)\frac{\partial h}{\partial \nu}\,d\mathcal H^{N-1}\\
&=\int_{B_\rho}|\nabla (u_i-h)|^2\, dx-2\gamma\int_{B_\rho}(h-u_i)\,dx.
\end{align*}
Let us first treat the terms not involving the gradient: using the bound $L^\infty$ on $h$ and the positivity of  $u_i$, we have 
\begin{equation}\label{eq:estnormL1}
\int_{B_\rho}(h-u)\,dx\leq \int_{B_\rho}h\,dx\leq C_h  \rho |B_\rho|.
\end{equation}	
Let us now focus on the gradient term.
By the Poincar\'e and Cauchy-Schwarz inequalities, we have
\[
\int_{B_\rho}|\nabla (u_i-h)|^2\, dx \geq \frac{C_d}{|B_\rho|}\left(\frac1\rho \int_{B_\rho}(h-u_i)\,dx\right)^2,
\]
where $C_d$ is a universal constant depending only on the dimension.
Thus, in order to prove the upper bound in the claim, we first need to show that
$\frac1{\rho|B_\rho|} \int_{B_{\rho}}(h-u_i)\,dx$ is bounded from below by a positive constant. Notice that, by the non-degeneracy of $u_i$ (see Remark~\ref{rmk:nondeg}), we have
\[
C \rho \leq \sup_{B_{\rho/2}} u_i \leq \sup_{B_{\rho/2}}h\,,
\]
for a constant $C=C(N,k,a,\Omega,\eta)$.
On the other hand, since $h(x)+\gamma \frac{|x|^2}{2N}$ is harmonic in $B_\rho$, the Harnack inequality in $B_\rho$ implies
\[
C \rho \le \sup_{B_{\rho/2}}h\le C_N\big(h(x)+\gamma\rho^2\big)\quad\text{for every}\quad x\in B_{\frac{\rho}2}\,,
\]
where $C_N$ is a dimensional constant.
Thus, by taking $\rho_1$ such that $2C_N\rho_1\gamma\le C$, we obtain that $h \geq C_NC \rho = \overline C\rho$ in $B_{\frac{\rho}2}.$
On the other hand, if $L=L(N,k,a,\Omega,\eta)$ is the Lipschitz constant of $u_i$ (by Theorem~\ref{thm:general_Lipschitz}), then for any $\eps\in(0,1)$, $u_i\leq L \eps \rho$ in $B_{\eps \rho}$. Then
$$\int_{B_\rho}(h-u_i)\,dx\geq\int_{B_{\eps \rho}} (h-u_i)\,dx\geq (\overline C \rho-L\eps \rho)|B_{\eps \rho}|,$$
which, after choosing $\eps\le \frac12$ small enough, shows that \[
\frac1{\rho} \int_{B_{\rho}}(h-u_i)\,dx\geq C_0|B_\rho|,
\]
for some constant $C_0=C_0(N,k,a,\Omega,\eta)>0$ and thus \[
\int_{B_\rho}|\nabla (u_i-h)|^2\,dx\geq C_b|B_\rho|,
\]
for some $C_b=C_b(N,k,a,\Omega,\eta)>0$.
			
At this point, using also~\eqref{eq:estnormL1}, we have \[
C_b|B_\rho|\leq \frac1\eta|B_\rho \cap \{u=0\}|+C_h\rho|B_\rho|.
\]
It is then enough to take\[
\rho_1\leq \frac{C_b}{2C_h},
\]
and we obtain that \[
\frac{C_b}{2}|B_\rho|\leq \frac1\eta|B_\rho \cap \{u_i=0\}|,
\]
which entails the density estimate from above, so the upper bound is proved for one-phase points.

\noindent \textbf{Step 3.} Upper bound of~\eqref{eq:densityest} at two-phase points. Suppose now that there exists $i\neq j$ such that $x_0\in \partial \Omega_{u_i}\cap \partial \Omega_{u_j}\cap {\Omega}$, so that $x_0$ is a two-phase point (thus, in a neighborhood of $x_0$ there are no other components). Then
\[
1 = \frac{|B_\rho(x_0)|}{|B_\rho(x_0)|} = \frac{|\Omega_{u_i}\cap B_\rho(x_0)|+ |\Omega_{u_j}\cap B_\rho(x_0)|+|\{u_i=u_j=0\}\cap B_\rho(x_0)|}{|B_\rho(x_0)|}
\]
Hence, by using the bound from below proved in Step 1 we obtain
\begin{align*}
\frac{|\Omega_{u_i}\cap B_\rho(x_0)|}{|B_\rho(x_0)|} & = 1-\frac{|\Omega_{u_j}\cap B_\rho(x_0)|}{|B_\rho(x_0)|}-\frac{|\{u_i=u_j=0\}\cap B_\rho(x_0)|}{|B_\rho(x_0)|} \leq 1-\xi,
\end{align*}
as wanted.

\noindent {\bf Step 4.} We now focus on the final claim, estimate \eqref{eq:measHN-1}. Notice that from Federer's Structure Theorem (see~Theorem \ref{Federer_thm}), we have
\[
\mathcal H^{N-1}(\partial^e \Omega_{u_i}\setminus\partial^*\Omega_{u_i})=0.
\]
On the other hand, the definition of essential boundary and the results of Steps $1,2,3$ imply
\[
\partial^e \Omega_{u_i}=\partial\Omega_{u_i}\setminus (\Omega_{u_i}^{(0)}\cap \Omega_{u_i}^{(1)})=\partial \Omega_{u_i}.\qedhere
\]
\end{proof}

One can also deduce the following nondegeneracy property of the gradient, see for instance~\cite[Corollary~3.4]{BogoselVelichkov}.
\begin{corollary}\label{cor:nondeggrad} 
Given $\eta\in(0,1]$, let $(u_1,\dots,u_k)\in \overline{H}$ be an $L^2$ normalized minimizer for $J_\eta$ achieving  $c_\eta$. Then there exist $\rho_0,K_0>0$ (the constants from Lemma~\ref{le:nondegeneracy}) and $C_N>0$, a constant depending only on the dimension $N$, such that
\begin{equation}
\dashint_{B_\rho(x_0)}{|\nabla u_i|^2}\geq C_N K_0^2,\qquad \text{for all }\rho\leq \rho_0,\ x_0\in \partial\Omega_{u_i}.
\end{equation}
\end{corollary}
\begin{proof}
It is clear that, for all $\rho>0$,
\[
B_\rho(x_0)\cap \Omega_{u_i} \not =\emptyset.
\]
Moreover, we have that for all $\rho\leq \rho_0$\[
\partial B_\rho(x_0)\cap \Omega_{u_i}\not =\emptyset.
\]
In fact, if for the sake of contradiction this does not hold, and  for some $\rho\in (0,\rho_0]>0$ we have  $\Omega_{u_i}\cap \partial B_\rho(x_0) =\emptyset$, then, recalling also that $\Omega_{u_i}$ is connected (Remark~\ref{rem:connected}), then $\Omega_{u_i}\subset B_{\rho}(x_0)$ and thus $\lambda_1(\Omega_{u_i})\geq \lambda_1(B_\rho)$, which is in contradiction with the minimality of $(u_1,\dots,u_k)$ up to taking $\rho_0$ sufficiently small.
Then, the conclusion follows by applying    Lemma~\ref{le:potentialestimate}, \eqref{eq:nondegmean2} and \eqref{eq:densityest}.
\end{proof}

\subsection{The optimal sets have finite perimeter}

We close this section with another important property for minimizers $(u_{1,\eta},\ldots, u_{k,\eta})$ of $c_\eta$: the sets $\Omega_{u_{i, \eta}}$ have finite perimeter.

\begin{lemma}\label{le:finiteper}
Let $(u_{1,\eta},\dots,u_{k,\eta})\in \overline{H}$ be an $L^2$--normalized minimizer of $J_\eta$ attaining $c_\eta$.
There exists a positive constant $C=C(N,k,a,\Omega)$ such that, for every $i=1,\dots, k$,
\[
\Per(\Omega_{u_{i,\eta}}) \leq C\eta^{-1/2}.
\]
\end{lemma}

\begin{proof}
We drop the dependence of the minimizers on $\eta$, for simplicity. For $t > 0$ and $i = 1, \ldots, k$, we consider the perturbations 
\[
(u_{1}, \ldots, u_{i-1}, (u_{i}-t)^+, u_{i+1}, \ldots, u_{k}).
\]
The minimality of $(u_1, \ldots, u_k)$ implies that 
\[
J_{\eta}(u_{1}, \ldots, u_{k}) \leq J_{\eta}(u_1,\ldots, u_{i-1}, (u_{i}-t)^+, u_{i+1}, \ldots, u_k).
\]
Hence,
\begin{align*}
\sum_{j=1}^k\int_{\Omega}|\nabla u_{j}|^2 + f_\eta\left(\sum_{j=1}^k|{\Omega}_{u_{i}}|\right) \leq  \sum_{j \neq i}^k\int_{\Omega} |\nabla u_{j}|^2 + \frac{\int_{\Omega} |\nabla (u_{i} - t)^{+}|^2}{\int_{\Omega} |(u_{i} - t)^{+}|^2} + f_\eta\left(\sum_{j\neq i}^k|{\Omega}_{u_j}| + |{\Omega}_{(u_i-t)^+}|\right),
\end{align*}                                                                                
and, by using the fact that $|\Omega_{(u_i-t)^+}|\leq |\Omega_{u_j}|$ and Lemma \ref{lemma:feta}-(2),
\begin{align*}
\sum_{j=1}^k\int_{\Omega}|\nabla u_{i}|^2 + \eta\left(|\Omega_{u_i}|-|\Omega_{(u_i-t)^+}|\right) \leq  \sum_{j \neq i}^k\int_{\Omega} |\nabla u_{j}|^2 + \frac{\int_{\Omega} |\nabla (u_{i} - t)^{+}|^2}{\int_{\Omega} |(u_{i} - t)^{+}|^2} .
\end{align*}     
This implies (by using \cite[Lemma A.1]{ASST})
\begin{align*}
\int_{\Omega}|\nabla u_{i}|^2 +  \eta|\{0 < u_{i} \leq t \}|  \leq & \int_{\{ u_{i}> t \}} |\nabla u_{i}|^2 + 2t \int_{\Omega}u_{i} \int_{\Omega} |\nabla (u_{i} - t)^{+}|^2\\
& + Ct^2\int_{\Omega} |\nabla (u_{i} - t)^{+}|^2,
\end{align*}
for $t$ sufficiently small.

Hence, by using $\|u_{i}\|_2 = 1$, $t\leq 1$ and the H\"{o}lder inequality, we infer that
\begin{align*}
\int_{\{0<u_{i}\leq t\}}|\nabla u_{i}|^2 +  \eta|\{0 < u_{i} \leq t \}| & \leq  \Big(2t |\Omega_{u_{i}}|^{1/2}+Ct^2\Big)\int_{\Omega} |\nabla (u_{i} - t)^{+}|^2 ,\\  
& \leq (2a^{1/2}+C)c_\eta t.
\end{align*}
By the Coarea formula and the Young inequality, we deduce 
\[
2\sqrt{\eta}\int_0^t{\mathcal H}^{N-1}(\{u_{i}=s\})\, ds = 2\sqrt{\eta}\int_{\{0\leq u_i\leq t\}}{|\nabla u_{i}|\, dx }\leq \int_{\{0<u_{i}\leq t\}}|\nabla u_{i}|^2 +  \eta|\{0 < u_{i} \leq t \}|\leq (C+2a^{1/2})c_\eta t,
\]
or, equivalently,
\[
\int_0^t{\mathcal H}^{N-1}(\{u_{i}=s\})\, ds \leq \frac{(C+2a^{1/2})c_\eta t}{2\sqrt{\eta}}.
\]
Now, we take $t = 1/n$ to obtain
\[
n\int_0^{1/n} \Per(\{u_{i}>s\})\, ds= n\int^{1/n}_0{\mathcal H}^{N-1}(\{u_{i}=s\}) \, ds \leq \frac{(C+2a^{1/2})c_\eta}{2\sqrt{\eta}}.
\]
Hence, there exists $\delta_n \in [0, 1/n]$ such that
\[
\Per(\{u_{i}> \delta_n\}) \leq n\int^{1/n}_0\Per(\{u_{i}>s\})\, ds\leq \frac{(C+2a^{1/2})c_\eta}{2\sqrt{\eta}}
\]
and, by taking the limit as $n \to 0$ and using the lower semicontinuity of the perimeter with respect to $L^1$-convergence, we conclude
\[
\Per(\{u_{i} > 0\}) \leq \frac{(C+2a^{1/2})c_\eta}{2\sqrt{\eta}}.\qedhere
\]
\end{proof}

\section{Two-phase points in the interior and one-phase point on the boundary}\label{sec:notriple}

\subsection{Absence of triple points in $\Omega$}

First, we rule out triple points in the interior of $\Omega$, using a three phase monotonicity formula, as in~\cite{BogoselVelichkov,BucurVelichkov}.

\begin{theorem}\label{thm:notriplepoints}
Given $\eta\in (0,1]$, let $(u_1,\dots,u_k)$ be a minimizer for the functional $J_\eta$ over $\overline{H}$, achieving $c_\eta$. 
Then, if $1\leq i,j,l\leq k$ are three different indexes, we have that 
\[
\partial \Omega_{u_i}\cap \partial \Omega_{u_j} \cap \partial \Omega_{u_l}\cap \Omega=\emptyset.
\]
\end{theorem}

\begin{proof}
We argue by contradiction. Suppose that we can find $x_0\in \partial \Omega_{u_i}\cap\partial \Omega_{u_j}\cap \partial \Omega_{u_l}\cap \Omega$. Then, by Corollary~\ref{cor:nondeggrad}, we have 
\begin{equation}\label{eq:contdpb0.5}
\dashint_{B_r(x_0)}|\nabla u_i|^2\geq 4K_0^2,\qquad \dashint_{B_r(x_0)}|\nabla u_j|^2\geq 4K_0^2, \qquad \dashint_{B_r(x_0)}|\nabla u_l|^2\geq 4K_0^2,
\end{equation}
for $r$ sufficiently small. By applying Lemma \ref{tpml} (in the form of inequality \eqref{CJK_remark}), we obtain
\[
\prod_{m=i, j, l}\left( \dfrac{1}{r^{2+\varepsilon}}\int_{B_r}\dfrac{|\nabla u_m|^2}{|x|^{N-2}}dx\right) \leq C\left(1+ \sum_{m = i, j, l}\int_{B_2}u_m^2\right)^3,
\]
which implies
\begin{align*}
\prod_{m=i, j, l}\left(\dfrac{1}{r^{N}}\int_{B_r}|\nabla u_m|^2dx\right) & \leq r^{3\varepsilon} C\left(1+ \sum_{m = i, j, l}\int_{B_2}u_m^2\right)^3 \leq r^{{3\varepsilon}} C.
\end{align*}
Therefore, for $r$ sufficiently small 
\[
\prod_{m=i, j, l}\left(\dashint_{B_r}|\nabla u_m|^2dx\right)  < (4K_0^2)^3,
\]
which is a contradiction with \eqref{eq:contdpb0.5}.
\end{proof}

\subsection{Absence of two-phase point on the boundary of $\Omega$} \label{eq:no2-phase-bdary}
We can also show that there are no two-phase points on the boundary of $\Omega$, inspired by, for example,~\cite[Section~4]{MazzoleniTreyVelichkov}.

\begin{theorem}\label{thm:notpbdry}
Let $\Omega\subset \R^N$ be an open bounded set with Lipschitz boundary, $\eta\in (0,1]$ and $(u_1,\dots,u_k)$ be a minimizer for $J_\eta$ on $\overline H$. Then
\[
\partial \Omega_{u_i}\cap \partial \Omega_{u_j}\cap \partial \Omega=\emptyset,\qquad\text{ for all $i\not=j$.}
\]
\end{theorem}

\begin{proof}
For the sake of contradiction, let $x_0\in \partial \Omega\cap\partial \Omega_{u_i}\cap \partial \Omega_{u_j}$. Then by Corollary~\ref{cor:nondeggrad} we have 
\begin{equation}\label{eq:contdpb}
\dashint_{B_r(x_0)}|\nabla u_i|^2\geq 4K_0^2,\qquad \dashint_{B_r(x_0)}|\nabla u_j|^2\geq 4K_0^2,\qquad\text{for $r\leq \bar r$ sufficiently small.} 
\end{equation}

Since $\partial \Omega$ is Lipschitz, then $ B_{\bar r} (x_0)\setminus(\overline{\Omega_{u_i}\cup\Omega_{u_j}})$ is a nonempty domain with Lipschitz boundary; for simplicity, from now on we assume without loss of generality that $x_0=0$. We need to build an auxiliary third phase there to employ the three phase monotonicity formula. We follow the approach of~\cite[Proof of Proposition~4.2]{MazzoleniTreyVelichkov}.
Let $v\in H^1(\R^N)$ be the $(1+\gamma)$-homogeneous, non-negative harmonic function on the cone $\mathcal C_\delta =\big\{x\in\R^N\ :\ x_N>\delta |x|\big\}$, which vanishes on $\partial \mathcal C_\delta$. We note that, for $\delta$ small enough, it is clear that, up to a rotation,  $\mathcal C_\delta\subset \R^N\setminus \overline{\Omega_{u_i}\cup\Omega_{u_j}}$.
In polar coordinates, 
\[v=r^{1+\gamma}\phi(\theta),\]
where $\phi$ is the first eigenfunction of the spherical Laplacian on $\mathcal C_\delta\cap \mathbb S^{N-1}$, that is,
$$-\Delta_{\mathbb S^{N-1}}\phi=(1+\gamma)(N-1+\gamma)\phi\quad\text{in}\quad \mathcal C_\delta\cap \mathbb S^{N-1},\qquad \phi=0\quad\text{on}\quad \partial \mathcal C_\delta\cap\mathbb S^{N-1},\qquad \int_{\mathbb S^{N-1}}\phi^2(\theta)\,d\theta=1,$$
where we notice that $\gamma$ is uniquely determined by $\delta$ (and the dimension $N$) and 
\[\lim_{\delta\to0}\gamma(\delta)=0.\]
Moreover, we have that 
\[\Delta v\ge 0\quad\text{in sense of distributions in } \R^N.\]
By the three-phase monotonicity formula, Lemma~\ref{tpml}, which we can apply thanks to the subharmonicity of $v$, there are constants $C>0$ and $\eps>0$ such that 
\[Cr^{3\eps}\geq \left(\mean{B_r}{|\nabla u_i|^2\,dx}\right)\left(\mean{B_r}{|\nabla u_j|^2\,dx}\right)\left(\mean{B_r}{|\nabla v|^2\,dx}\right).
\]
Now, from \eqref{eq:contdpb},
\begin{align*}
r^{3\eps}&\ge C_NK_0^4\frac{1}{|B_r|}\int_{B_r}|\nabla v|^2\,dx\\
&= C_NK_0^4 \frac{1}{|B_r|}\int_0^r\int_{\mathbb S^{N-1}}\Big((1+\gamma)^2\phi^2(\theta)+|\nabla_\theta\phi(\theta)|^2\Big)\rho^{N-1+2\gamma}\,d\theta\,d\rho= C_NK_0^4(1+\gamma)r^{2\gamma}, 
\end{align*}
which is impossible when $\delta$ (and thus $\gamma$) is small enough ($\eps$ being a fixed constant, depending on $N$, $\lambda_1(\Omega_i)$ and $\lambda_1(\Omega_j)$, but not on $\delta$).
\end{proof}

\section{Existence of a one-phase point, shape variations at the free boundary, and equivalence with the constrained problem}\label{sec:equivalence}

Let us first show that at least a regular one-phase point exists in the reduced boundary of one component of the partition. This is something that is natural to expect, but its proof is not completely trivial.

\begin{lemma}\label{le:onephase}
Given $\eta\in (0,1]$, let $(u_1,\dots,u_k)$ be optimal for problem~\eqref{eq:ceta}.
    There exists $i\in\{1,\ldots, k\}$ such that $\Gamma_{OP}(\partial\Omega_{u_i})\cap \partial^* \Omega_{u_i}\not=\emptyset$.
\end{lemma}
\begin{proof}
If, for some $i\in\{1,\dots,k\}$, we have that $\partial \Omega_{u_i}\cap \partial \Omega\not=\emptyset$, then thanks to Theorem~\ref{thm:notpbdry}, we infer that all points $p$ in the relative boundary of $\partial \Omega_{u_i}$ with respect to $\partial \Omega$, i.e.   $p\in \partial_{rel} \Big(\partial \Omega_{u_i}\cap \partial \Omega\Big)\not=\emptyset$,  are one-phase points. 
As a consequence, also using the relative isoperimetric inequality (see for example~\cite[Remark 12.38]{Maggi}), we deduce that  $\mathcal H^{N-1}(\partial \Omega_{u_i}\cap B_r(p))>0$ for all $r>0$.
Since the set of  one-phase points $\Gamma_{OP}(\partial\Omega_{u_i})$ is a relatively open set of $\partial\Omega_{u_i}$, and $\mathcal H^{N-1}(\partial\Omega_{u_i}\setminus \partial^*\Omega_{u_i})=0$, we can find a point $p'\in \partial^*\Omega_{u_i}\cap \Gamma_{OP}(\partial\Omega_{u_i}).$

If this does not happen, namely for all $i=1,\dots,k$ it holds  $\partial\Omega_{u_i}\cap\partial\Omega=\emptyset$, then we can find another open, bounded set with Lipschitz boundary $\Omega'\subset \Omega$ such that the optimal vector $(u_1,\dots,u_k)$ still solves problem~\eqref{eq:ceta} in $\Omega'$ and there is at least a contact point $p\in \partial\Omega'\cap\partial\Omega_{u_i}$ for some $i\in \{1,\dots,k\}$, so we are reduced to the situation of the previous paragraph.
To show the existence of such an $\Omega'$, it is sufficient to choose a direction, for example $e_1$, and consider the intersection $\Omega_t=\Omega\cap \{x\in \R^N : x\cdot e_1=x_1>t\}$, for $t\in \R$.
Then we can define $\Omega'=\Omega_{\overline t}$ (which is clearly an open, bounded set with Lipschitz boundary), where $\overline t$ is the largest  $t$ such that $\Omega_t$ contains all $\cup_l\Omega_{u_l}$.
This choice of $\overline t$ also ensures that $\partial\Omega'\cap \partial \Omega_{u_i}\not=\emptyset$ at least for some $i$, so we have concluded.
\end{proof}

We now prove (see Propositions  \ref{prop:blowupconv} and \ref{prop:optimalitycond} and Remark~\ref{rem:reduced_2phase} below) that, in the sense of  measures as in~\cite{AguileraAltCaffarelli,AltCaffarelli}, there exists $m_\eta$ such that an optimal vector $(u_1,\dots,u_k)$ for~\eqref{eq:ceta} is a solution to\[
\begin{cases}
-\Delta u_i=\lambda_1(\Omega_{u_i})u_i-|\nabla u_i|\mathcal H^{N-1}\resmeas{\partial^*\Omega_{u_i}}\qquad \text{in }\quad\Omega,\\
|\nabla u_i|=m_\eta\qquad \text{on }\Big(\partial^* \Omega_{u_i}\setminus \cup_{j\not =i}\partial \Omega_{u_j}\Big)\cap \Omega,\\
|\nabla u_i|\ge m_\eta,\quad |\nabla u_j|\geq m_\eta \quad\text{and}\quad |\nabla u_i|-|\nabla u_j|=0\quad\text{on}\quad \partial^*\Omega_{u_i}\cap \partial^*\Omega_{u_j}\cap \Omega.
\end{cases}
\]

At this point the notation $|\nabla u_i|$ is purely formal (see Propositions \ref{prop:blowupconv} and \ref{prop:optimalitycond} and Remark \ref{rem:reduced_2phase} for the actual statements), but it will be precise once we prove that $\partial^* \Omega_{u_i}$ is regular.
We note that we do not have different constants for each component because the measure-penalization term computes $f_\eta$ on the measure of the union of \emph{all} $\Omega_{u_i}$, $i=1,\ldots, k$, and does not penalize each set with a different weight.

\begin{proposition}\label{prop:blowupconv}
Let $(u_1,\dots,u_k)$ be optimal for problem~\eqref{eq:ceta}.
For all $i=1,\dots, k$, there is a nonnegative Borel function $q_{u_i}:\partial^* \Omega_{u_i}\to \R$ such that, in the sense of distributions, one has
\begin{equation}\label{eq:radonmeasure}
-\Delta u_i = \lambda_1(\Omega_{u_i})u_i -q_{u_i}(x)\mathcal H^{N-1}\resmeas\partial^*\Omega_{u_i}\qquad\text{ in }\Omega.
\end{equation}
Moreover, for all points $\overline x\in \partial^*\Om_{u_i}$, the measure theoretic inner unit normal $\nu_{u_i}(\overline x)$ is well defined and, as $\varepsilon\to0$, 
\begin{equation}\label{eq:bdp4.32}
\frac{\Om_{u_i}-\overline x}{\varepsilon}\rightarrow \{x : x\cdot \nu_{u_i}(\overline x)\geq 0\}\qquad \text{ in }L^1(\Omega).
\end{equation}
 Finally, for every $\overline x\in \partial^*\Omega_{u_i}$, we have 
\begin{equation}\label{eq:bdp4.33}
\frac{u_i(\overline x+\varepsilon x)}{\varepsilon}\longrightarrow q_{u_i}(\overline x)(x\cdot \nu_{u_i}(\overline x))_+\qquad \text{in } L^1(\Omega),  \text{ as } \varepsilon \to 0.
\end{equation}
\end{proposition}

\begin{proof} 
Checking that equation~\eqref{eq:radonmeasure} holds is standard and boils down to showing that $-\Delta u_i-\lambda_1(\Omega_{u_i})u_i$ is locally a Radon measure concentrated on $\partial^*\Omega_{u_i}$ and can be done for example as in~\cite[Proposition~2.3]{BucurMazzoleniPratelliVelichkov}, see also Lemma~\ref{lemma:otherproperties}.
Claim \eqref{eq:bdp4.32} is a direct consequence of Theorem \ref{thm:deGiorgi}, while \eqref{eq:bdp4.33} follows as in ~\cite[Theorem 4.8]{AltCaffarelli}. 
\end{proof}
We are now in position to prove an optimality condition at one-phase points of the reduced boundary of each component of an optimal partition, following the approach of~\cite{AguileraAltCaffarelli}.

\begin{proposition}\label{prop:optimalitycond}
Under the assumptions and notations of Proposition~\ref{prop:blowupconv}, there exists $m_\eta\geq 0$ such that the functions $(q_{u_1}, \dots, q_{u_k})$ satisfy the following:
\begin{equation}\label{eq:m_eta}
q_{u_i}=m_\eta\qquad \text{on }\Big(\partial^* \Omega_{u_i}\setminus \cup_{j\not =i}\partial \Omega_{u_j}\Big)\cap \Omega,\qquad \text{for all $i=1,\dots, k$}.
\end{equation}
\end{proposition}

\begin{proof} 
Let  $i,j\in \{1,\ldots, k\}$, $i\leq j$, and
\[
x_0\in\Big(\partial^*\Omega_{u_i}\setminus \cup_{\ell\not=i}\partial \Omega_{u_\ell}\Big)\cap \Omega, \qquad x_1\in\Big(\partial^*\Omega_{u_j}\setminus \cup_{\ell\not=j}\partial \Omega_{u_\ell}\Big)\cap \Omega
\]
such that $x_0\neq x_1$. Then we construct a family of volume preserving diffeomorphisms  as follows: let $\rho>0$ be such that
\begin{equation}\label{eq:ball_rho_def}
B_\rho(x_0)\subset \Omega\setminus \cup_{\ell\not=i}\partial \Omega_{u_\ell},\quad  B_\rho(x_1)\subset \Omega\setminus \cup_{\ell\not=j}\partial \Omega_{u_\ell},\quad B_\rho(x_0)\cap B_\rho(x_1)=\emptyset,
\end{equation}
 and let $\varphi\in C^1([0,1])$ be a nontrivial function, such that $\varphi=0$ in a neighborhood of $1$. We define, for $\kappa>0$,
\[
\tau_{\rho,\kappa}(x)=\tau(x):=x+\kappa\rho\varphi\left(\frac{|x-x_0|}{\rho}\right)\nu_{x_0} \chi_{B_\rho(x_0)}-\kappa\rho\varphi\left(\frac{|x-x_1|}{\rho}\right)\nu_{x_1} \chi_{B_\rho(x_1)},
\]
where $\nu_{x_m}$, with $m\in\{0,1\}$, is the measure theoretic inner normals to $\partial^*\Omega_{u_\ell}$ at $x_m$, for $\ell=i,j$ and $m=0,1$.
We have:
\begin{equation}\label{eq:dettau}
D\tau(x)=Id+\kappa\varphi'\left(\frac{|x-x_0|}{\rho}\right)\frac{x-x_0}{|x-x_0|}\otimes\nu_{x_0} \chi_{B_\rho(x_0)}-\kappa\varphi'\left(\frac{|x-x_1|}{\rho}\right)\frac{x-x_1}{|x-x_1|}\otimes\nu_{x_1} \chi_{B_\rho(x_1)}. 
\end{equation}
We have that $\tau(x)-x$ vanishes outside $B_\rho(x_0)\cup B_\rho(x_1)$ while, for $\kappa$ small enough, $\tau$ is a diffeomorphism. Using the Jacobi's formula $\det(Id+\xi A)=1+ {\rm trace}(A)\xi+o(\xi)$, valid for any matrix $A\in\R^{N\times N}$, we have
\begin{align}
\det(D\tau(x)) =& 1+ \kappa\varphi'\left(\frac{|x-x_0|}{\rho}\right)\frac{x-x_0}{|x-x_0|}\cdot\nu_{x_0} \chi_{B_\rho(x_0)}(x)\nonumber \\
&-\kappa\varphi'\left(\frac{|x-x_1|}{\rho}\right)\frac{x-x_1}{|x-x_1|}\cdot\nu_{x_1} \chi_{B_\rho(x_1)}(x)+o(\kappa)\chi_{B_\rho(x_0)\cup B_\rho(x_1)}(x)  \label{eq:espdet}
\end{align}
as $\kappa\to 0$, uniformly in $\rho$. We call $\Omega_{\rho,\ell} =\tau(\Omega_{u_\ell})$ and $\widetilde u_{\rho,\ell}(z)=u_\ell(\tau^{-1}(z))$, noting that $\widetilde u_{\rho,\ell}\in H^1_0(\Omega_{\rho,\ell})$, for $\ell=i,j$. 
Notice that, by \eqref{eq:ball_rho_def}, it is clear that $(u_1,\dots, \widetilde u_{\rho,i},\dots,\widetilde u_{\rho,j},\dots, u_k)\in \overline H$, so it is an admissible vector.
We now perform the first variation of each term of the sum defining $J_\eta$. Regarding the variation of the $L^2$-norm, 
\[
\begin{aligned}
\sum_{\ell\in \{i,j\}}\frac{1}{\rho^N}& \left(\int_{\Omega_{\rho,\ell}}\widetilde u^2_{\rho,\ell}(z)\,dz-\int_{\Omega_{u_\ell}}u_l^2(x)\,dx\right)  =\sum_{\ell\in \{i,j\}}\frac{1}{\rho^N}\int_{\Omega_{u_\ell}} \Big(u_{\ell}^2(x)\det(D\tau(x))-u^2_\ell(x)\Big)\,dx\\
=& \frac{1}{\rho^N}\int_{B_\rho(x_0)\cap\Omega_{u_i}} \Big(u_{i}^2(x)\det(D\tau(x))-u^2_i(x)\Big)\,dx+\frac{1}{\rho^N}\int_{B_\rho(x_1)\cap\Omega_{u_j}} \Big(u_{j}^2(x)\det(D\tau(x))-u^2_j(x)\Big)\,dx\\
=&\int_{B_1(0)\cap \frac{\Omega_{u_i}-x_0}{\rho}} \Big(u_i^2(x_0+\rho y)\det(D\tau(x_0+\rho y))- u_i^2(x_0+\rho y)\Big)\,dy\\
& +\int_{B_1(0)\cap \frac{\Omega_{u_j}-x_1}{\rho}} \Big(u_j^2(x_1+\rho y)\det(D\tau(x_1+\rho y))- u_j^2(x_1+\rho y)\Big)\,dy\\
=&\int_{B_1\cap\left(\frac{\Omega_{u_i}-x_0}{\rho}\right)} \frac{u_i^2(x_0+\rho y)}{\rho^2}\rho^2\,\kappa\,\varphi'(|y|)\frac{y}{|y|}\cdot\nu_{x_0}\,dy\\
&-\int_{B_1\cap\left(\frac{\Omega_{u_j}-x_1}{\rho}\right)} \frac{u_j^2(x_1+\rho y)}{\rho^2}\rho^2\,\kappa\,\varphi'(|y|)\frac{y}{|y|}\cdot\nu_{x_1}\,dy+o(\kappa)\\
=&o_\kappa(1)(\rho^2+\kappa)
\end{aligned}
\]
(where $o_\kappa(1)\to 0$ as $\kappa\to 0$), where we performed the change of variable $x=x_m+\rho y$, exploited~\eqref{eq:espdet} and used Proposition~\ref{prop:blowupconv}. We stress that the computations regarding the volume and the Dirichlet integral contributions are identical to those performed originally in \cite[Theorem 3]{AguileraAltCaffarelli}. 
For the volume term, one has
\begin{equation*}
\sum_{\ell\in\{i,j\}}(|\Omega_{\rho,\ell}|-|\Omega_{u_\ell}|)= O(\kappa)\rho^N+o(\kappa)\rho^N\qquad \text{as }\kappa, \rho\rightarrow 0;
	\end{equation*}
 from the properties of $f_\eta$ (Lemma \ref{lemma:feta}), one immediately infers that 
		\begin{equation*}
	f_\eta(|\Omega_{\rho,i}\cup \Omega_{\rho,j}\cup (\cup_{\ell\not\in\{i,j\}} \Omega_{u_j})|)-f_\eta(|\cup_\ell\Omega_{u_\ell}|)= o_\kappa(1)\rho^N+o(\kappa)\rho^N,\qquad \text{as }\kappa,\rho\rightarrow 0.
	\end{equation*}
 For the Dirichlet energy term we rewrite all the details from \cite{AguileraAltCaffarelli}, since we are going to need some intermediate steps in the proof of Proposition \ref{prop:m(x0)>=m_eta} below:
\begin{align*}
\sum_{\ell=1}^k &\frac{1}{\rho^N} \Big(\int_{\Omega_{\rho,\ell}}|\nabla \widetilde u_{\rho,\ell}|^2-\int_\Omega |\nabla u_\ell|^2\Big)\\
 =&\sum_{\ell\in \{i,j\}} \frac{1}{\rho^N}\int_{\Omega_{\rho,\ell}} |Du_\ell(x) [D\tau(x)]^{-1}|^2\det(D\tau(x))-|\nabla u_\ell|^2\, dx   \\
 =&\int_{ B_1\cap \frac{\Omega_{u_i}-x_0}{\rho}}  |Du_i(x_0+\rho y) [D\tau(x_0+\rho y)]^{-1}|^2\det(D\tau(x_0+\rho y))-|\nabla u_i|^2\, dx\\
 &+\int_{ B_1\cap \frac{\Omega_{u_j}-x_1}{\rho}}  |Du_j(x_1+\rho y) [D\tau(x_1+\rho y)]^{-1}|^2\det(D\tau(x_1+\rho y))-|\nabla u_j|^2\, dx\\
 =& -\kappa \int_{B_1\cap \frac{\Omega_{u_i}-x_0}{\rho}} 2\varphi'(|y|)\left( \nabla \widetilde u_{\rho,i}\cdot \frac{y}{|y|}\right) \left( \nabla \widetilde u_{\rho,i}\cdot \nu_{x_0}\right) -\varphi'(|y|)|\nabla \widetilde u_{\rho,i}|^2 \left( \nu_{x_0}\cdot \frac{y}{|y|}\right) \\
 &+ \kappa \int_{B_1\cap \frac{\Omega_{u_j}-x_1}{\rho}} 2\varphi'(|y|)\left( \nabla \widetilde u_{\rho,j}\cdot\frac{y}{|y|}\right) \left( \nabla \widetilde u_{\rho,j}\cdot \nu_{x_1}\right) -\varphi'(|y|)|\nabla \widetilde u_{\rho,j}|^2 \left( \nu_{x_1}\cdot\frac{y}{|y|}\right)+o(\kappa)\\
 =&\kappa{ (C_0(\varphi)q_{u_i}^2(x_0)-C_1(\varphi)q^2_{u_j}(x_1))} + o_\rho(1)+o(\kappa),
	\end{align*}
 
where we have used Proposition \ref{prop:blowupconv} (which implies that $\nabla \widetilde u_{\rho,i}\to q_{u_i}(x_0) \nu_{x_0}\chi_{\{x\cdot \nu_{x_0}\geq 0\}}$ a.e., and the same for $u_j$ at $x_1$) and for $m\in\{0,1\}$
\begin{equation}\label{rfk}
C_m(\varphi):=-\int_{B_1\cap \{y\cdot \nu_{x_m}>0\}}\varphi'(|y|)\frac{y\cdot \nu_{x_m}}{|y|}\,dy=\int_{B_1\cap\{y\cdot\nu_{x_m}=0\}}\varphi(|y|)\,d\mathcal H^{N-1}(y).
\end{equation}
The last equality follows from the Divergence Theorem, recalling that $\nu_{x_m}$ is an inner normal and 
$$
{\mathrm{div}}(\varphi(|y|)\nu_{x_m})=\varphi'(|y|)\frac{y\cdot \nu_{x_m}}{|y|}.
$$
Notice also that, by the radial symmetry of $\varphi$, $C_0(\varphi)=C_1(\varphi)= C(\varphi)$, namely this quantity is not affected by the choice of the normal $\nu$.

Summarizing, we have that (recalling also that $\int_\Omega u_i^2=1$), and using the minimality of $(u_1,\ldots, u_k)$
\begin{align*}\label{eq:conclvariation}
J_\eta(u_1,\ldots, u_k)&\leq J_\eta(u_1,\dots, \widetilde u_{\rho,i},\dots,\widetilde u_{\rho,j},\dots, u_k)\\
&\leq J_\eta(u_1,\ldots, u_k)+\kappa \rho^N C(\varphi)((q_{u_i}(x_0))^2-(q_{u_j}(x_1))^2)+o(\rho^N)+\rho^No(\kappa).
\end{align*}
Therefore
\[
0\leq \kappa \rho^N C(\varphi)((q_{u_i}(x_0))^2-(q_{u_j}(x_1))^2)+o(\rho^N)+\rho^No(\kappa).
\] Dividing by $\kappa \rho^N$, letting $\rho\to 0$, and then $\kappa\to 0$, we deduce that 
\[
q_{u_j}(x_1)\leq q_{u_i}(x_0),
\]
using also the fact that $q_{u_i},q_{u_j}$ are nonnegative. Since $x_0,x_1$ are arbitrary (and taking $i=j$ in the above arguments), then there exists a constant $m_\eta$ such that
\[
q_{u_i}(x_0)\equiv m_\eta \qquad \text{ for all } x_0\in \left(\partial^* \Omega_{u_i}\setminus \cup_{j\neq i} \partial \Omega_{u_j}\right)\cap \Omega,\ i=1,\ldots, k.
\]
\end{proof}

\begin{remark}\label{rem:reduced_2phase}
    Using a proof similar to the one of Proposition \ref{prop:optimalitycond}, it is possible to show the following result at two-phase points of the reduced boundaries:
    \[
    q_{u_i}\ge m_\eta,\quad q_{u_j}\geq m_\eta \quad\text{and}\quad q_{u_i}-q_{u_j}=0\quad\text{on}\quad \partial^*\Omega_{u_i}\cap \partial^*\Omega_{u_j}\cap \Omega,\qquad \text{for all $i\not=j$.}
    \]
    However, we will prove this kind of result at \emph{all} two-phase points, see Proposition \ref{prop:m(x0)>=m_eta} below.
\end{remark}

The density estimates from above and from below of Lemma~\ref{le:densityest} then entail the following result.
\begin{lemma}\label{le:boundmeta}
There are universal constants $0<c\leq C$ such that, for all $\eta\in (0,1]$ and for all $i$, 
\[
0<c\leq q_{u_i}(x)\leq C \text{ for }x\in \partial^* \Omega_{u_i}\cap \Gamma_{OP}(\partial \Omega_{u_i}),\qquad \text{ which implies}\qquad 0<c\leq m_\eta\leq C.
\]
\end{lemma}
\begin{proof}
The first statement follows as in \cite[Theorem~4.3]{AltCaffarelli}. Then, since $q_{u_i}$ is equal to $m_\eta$ at one-phase points, we deduce the same bounds also on $m_\eta$.
\end{proof}

We are now in position to prove the existence of a constant $\eta$ for which there is equivalence between the constrained and the unconstrained problem, showing in particular that $\Lambda_{1,k}(a)$ and $c_\eta$ coincide and have the same solutions.
\begin{lemma}\label{le:measOK}
There exists $\eta_0>0$ such that, for $0<\eta<\eta_0$, if $(u_{1,\eta},\dots, u_{k,\eta})$ is optimal for problem~\eqref{eq:ceta}, then we have
\[
\sum_{i=1}^k |\Omega_{u_{i,\eta}}|=a.
\]
In particular $c_\eta=\Lambda_{1,k}(a)$, and the partition $(\Omega_{u_{1,\eta}},\ldots, \Omega_{u_{k,\eta}})$ solves our original problem \eqref{eigenvalue_problem}. Conversely, if $(\Omega_1,\ldots, \Omega_k)$ is a solution to \eqref{eigenvalue_problem}, then the corresponding eigenfunctions  $(u_1,\ldots, u_k)$ optimize~\eqref{eq:ceta}.
\end{lemma}

\begin{proof}
\textbf{Step 1.} (Saturation of the constraint) We already know that $\sum_{i=1}^k |\Omega_{u_{i,\eta}}|\leq a$, thanks to Lemma~\ref{lem:measmalla}. 
For the sake of contradiction, let us assume that \[
\sum_{i=1}^k |\Omega_{u_{i,\eta}}|<a.
\]
Then, following the same idea of~\cite{AguileraAltCaffarelli}, it is enough to do an outward perturbation at a one-phase (regular) point $x_0\in \partial^*\Omega_{u_{i,\eta}}\cap\Omega$ for some $i=1,\dots,k$. Such a point exists thanks to Lemma~\ref{le:onephase}. 
We use the optimality condition at the one-phase free boundary, $|\nabla u_{i,\eta}|^2(x_0)=m_\eta$, with $m_\eta$ uniformly bounded from below (and from above) by Lemma~\ref{le:boundmeta}, and a standard first variation argument on the first eigenvalue of the Dirichlet Laplacian.
More precisely, we consider a small ball  $B_r(x_0)\subset \Omega\setminus \cup_{j\not= i}\Omega_{u_{j,\eta}}$  such that $|\cup_{j}\Omega_{u_{j,\eta}}\cup B_r(x_0)|<a$, and take a smooth vector field compactly supported in this ball, $\xi\in C^\infty_c(B_r(x_0),\R^N)$. For $t>0$ sufficiently small, we define 
\[
 \Omega_{u_{i,\eta},t}:=(Id+t\xi)(\Omega_{u_{i,\eta}})\subset \Omega_{u_{i,\eta}}\cup B_r(x_0).
\]
We note that we do not require that $\xi$ only  pushes outward the set.
A shape variation formula (see~\cite[Section~5.9.3]{HenrotPierre}) leads to 
\begin{align*}
\lambda_1(\Omega_{u_{i,\eta}})-\lambda_1(\Omega_{u_{i,\eta},t})&=t\int_{\partial^*\Omega_{u_{i,\eta}}\cap B_r(x_0)}q_{u_{i,\eta}}\xi\cdot\nu\,d\mathcal H^{N-1}+o(t)\\
&=tm_\eta \int_{\partial^*\Omega_{u_{i,\eta}}\cap B_r(x_0)}\xi\cdot\nu\,d\mathcal H^{N-1}+o(t)\quad\text{ as }t\to 0,
\end{align*}
where, in the last identity, we used Proposition~\ref{prop:optimalitycond}.
Then we use the Divergence Theorem (on the set $\Omega_{u_{i,\eta}}\cap B_r(x_0)$), and we can choose a vector field as a suitable multiple of the one in~\cite[Lemma~11.3]{Velichkov_onephasebook} and $t$ sufficiently small so that \[
\lambda_1(\Omega_{u_{i,\eta}})-\lambda_1(\Omega_{u_{i,\eta},t})=t m_\eta\int_{\Omega_{u_{i,\eta}}\cap B_r(x_0)}\div \xi\,dx+o(t)=m_\eta|\Omega_{u_{i,\eta,t}}\setminus \Omega_{u_{i,\eta}}|t+o(t).
\]
Using the minimality of $(u_{1,\eta},\dots, u_{k,\eta})$, the information above and Lemma~\ref{le:boundmeta}, we finally obtain
\[
\begin{split}
c|\Omega_{u_{i,\eta},t}\setminus \Omega_{u_{i,\eta}}|&\leq m_\eta |\Omega_{u_{i,\eta},t}\setminus \Omega_{u_{i,\eta}}|\leq \sum_{j=1}^k\lambda_1(\Omega_{u_{j,\eta}})-\sum_{j=1}^k
\lambda_1(\Omega_{u_{j,\eta},t})\\
&\leq f_\eta(|\cup_{j\not=i}\Omega_{u_{j,\eta}}\cup \Omega_{u_{j,\eta},t}|)-f_\eta(|\cup_j\Omega_{u_{j,\eta}}|)\\
&=\eta(|\Omega_{u_{i,\eta},t}\setminus \Omega_{u_{i,\eta}}|-|\Omega_{u_{i,\eta}}\setminus \Omega_{u_{i,\eta},t}|)\leq\eta|\Omega_{u_{i,\eta},t}\setminus \Omega_{u_{i,\eta}}|.
\end{split}
\]
For $\eta\leq \eta_0$ small enough we obtain a contradiction.

\smallbreak
\noindent \textbf{Step 2.} (Equivalence between \eqref{eigenvalue_problem} and \eqref{eq:ceta}) From Lemma \ref{lemma:bounds_ceta}, $c_\eta\leq \Lambda_{1,k}(a)$. On the other hand, if  $(u_{1,\eta},\dots, u_{k,\eta})$ is optimal for problem~\eqref{eq:ceta}, then by Step 1, it belong to $H_a$ and $f_\eta(\sum_{i=1}^k|\Omega_{u_i}|)=a$, so
\[
c_\eta=J_\eta(u_{1,\eta},\dots, u_{k,\eta})=J(u_{1,\eta},\dots, u_{k,\eta})\geq \Lambda_{1,k}(a).
\]
Therefore $\Lambda_{1,k}(a)=c_\eta$, and the rest of the claims follow directly.
\end{proof}

\begin{remark}\label{rmk:etafixed}
Thanks to Lemma~\ref{le:measOK}, we can now fix $\eta>0$ small enough so that, taking an optimal vector $(u_1,\dots,u_k)\in \overline{H}$ of problem~\eqref{eq:ceta}, the nodal sets $\Omega_{u_1},\dots, \Omega_{u_k}$ are mutually disjoint and satisfy the measure constraint $|\cup_{i=1}^k\Omega_{u_i}|=a$.
Thus $f_\eta(\cup_{i=1}^k\Omega_{u_i})=0$, $J_\eta(u_1,\ldots, u_k)=\sum_{i=1}^k \lambda_1(\Omega_{u_i})$, and 
\[
(\Omega_{u_1},\ldots, \Omega_{u_k})\text{ is an optimal partition of the original constrained optimal partition problem } \eqref{eigenvalue_problem}.
\] 

In order to conclude the proof of Theorem \ref{thm:main}, we will use, in Section \ref{sec:viscosity}, the fact that we have found \emph{both} a solution of \eqref{eq:ceta} and \eqref{eigenvalue_problem}. In the following section, in order to classify blowups at two-phase points, we exploit \eqref{eq:ceta}.
\end{remark}

\section{Classification of blowups at two-phase points}\label{sec:butwophase}
Up to now, we have only partially classified blowups at points of the reduced boundary (recall Propositions \ref{prop:blowupconv} and \ref{prop:optimalitycond}, and Remark \ref{rem:reduced_2phase}). In this section, using other arguments, we completely classify blowup limits at \emph{all} two-phase points, combining then all these information together to obtain a common lower bounds of some quantities.  More precisely, we prove the following result.
\begin{theorem}\label{thm:blowuplimits}
Take $\eta\in (0,1]$ and let $(u_1,\dots, u_k)$ be an $L^2$-normalized solution for problem~\eqref{eq:ceta}. Let $i\neq j$ and  $x_0\in \Gamma_{TP}(\partial \Omega_i)\cap \Gamma_{TP}(\partial \Omega_j)$, that is, $x_0\in   (\partial \Omega_{u_i}\cap \partial\Omega_{u_j}\cap \Omega)\setminus \cup_{\ell\not=i,j}\overline \Omega_{u_\ell}$ is a two-phase point at the boundary of $\Omega_{u_i}$ and $\Omega_{u_j}$.

Take the blowup sequences centered at $x_0$: for small $\varepsilon>0$,
\[
u_{i,\eps}(x)=\frac{u_i(x_0+\varepsilon x)}{\varepsilon},\qquad u_{j,\eps}(x)=\frac{u_j(x_0+\varepsilon x)}{\varepsilon}, \quad \text{ for } x\in \Omega_\varepsilon:=\frac{\Omega-x_0}{\varepsilon},
\]
extended by zero outside $\Omega_\varepsilon$. Then there exists $m(x_0)$ such that, up to pass to a subsequence, as $\eps\to0$,  
\begin{equation}\label{blow_up_limits_tp}
u_{i,\eps}\rightarrow m(x_0) (x\cdot \nu)_+,\qquad u_{j,\eps}\rightarrow m(x_0)(x\cdot \nu)_- \quad \text{ strongly in } L^\infty_{loc}(\R^N) \text{ and } H^1_{loc}(\R^N),
\end{equation}
for some $\nu\in \partial B_1$. Moreover, 
\[
\frac{\Omega_{u_{i}}-x_0}{\eps}\to \{x\cdot \nu>0\},\quad \frac{\Omega_{u_{j}}-x_0}{\eps}\to \{x\cdot \nu<0\} \quad \text{ in } L^1_{loc}(\R^N).
\]
\end{theorem}

The remainder of this section is dedicated to the proof of this theorem. To simplify notation, we take from now on $i=1$, $j=2$ and assume without loss of generality that $x_0=0$. We also consider $u_1,u_2$ extended by zero to $\R^N\setminus \Omega$.

We start with a first characterization of the blowup limits.

\begin{lemma}\label{lemma:blowuplimits_def}
Under the assumptions of Theorem~\ref{thm:blowuplimits}, there exists $\bar u_1,\bar u_2\in C^{0,1}_\text{loc}(\R^N)\cap H^1_\text{loc}(\R^N)$ such that, up to pass to a subsequence,
\[
u_{1,\varepsilon}\to \bar u_1,\qquad u_{2,\varepsilon}\to \bar u_2 \qquad \text{strongly in } L^\infty_\text{loc}(\R^N),\text{ and  } H^1_\text{loc}(\R^N).
\]
Moreover,
\begin{equation}\label{eq:subharmonic_blowuplimits}
-\Delta \bar u_i\leq 0\ \text{ in } \R^N,\quad -\Delta \bar u_i=0\ \text{ in } \Omega_{\bar u_i}\quad \text{ and } \quad\bar u_1\cdot \bar u_2\equiv 0.
\end{equation}

\end{lemma}
\begin{proof}
Observe that, since $u_1$ is Lipschitz continuous and $0\in \partial \Omega_{u_1}$ then, in each compact $K\Subset \R^N$, 
\[
|u_{1,\varepsilon}(x)|=|u_{1,\varepsilon}(x)-u_{1,\varepsilon}(0)|\leq L |x|\leq C,\qquad  |\nabla u_{1,\varepsilon}(x)| = |\nabla u_{1}(\varepsilon x)|\leq C,
\]
and the same for $u_{2,\varepsilon}$. Hence, the sequences are bounded in $H^1_{loc}(\R^N)$ and, by Ascoli-Arzel\`a's theorem, the blowup sequences converge locally uniformly, up to pass to a subsequence, and also weakly in $H^1_{loc}(\R^N)$. On the other hand,
\[
-\Delta u_{i,\varepsilon}\leq \varepsilon^2 \lambda_{1}(\Omega_{u_i})u_{i,\varepsilon} \text{ in } \R^N,\quad -\Delta u_{i,\varepsilon}=\varepsilon^2 \lambda_{1}(\Omega_{u_i})u_{i,\varepsilon} \text{ in } \Omega_{u_{i,\varepsilon}}\quad \text{ and }\quad  u_{1,\varepsilon}\cdot u_{2,\varepsilon}\equiv 0 \text{ in } \R^N, 
\]
which shows \eqref{eq:subharmonic_blowuplimits}. The only thing left to prove is the strong $H^1_{loc}$-limit of blowup sequences. For that, we follow \cite[Lemmas 3.7 and 3.11]{TavaresTerracini1} (see also \cite[Step 5 of the proof of Theorem 4.1]{RamosTavaresTerracini} or \cite[Lemma 7.4]{RussTreyVelichkov}). We know that
\[
\mu_{i,\varepsilon}:=\Delta u_{i,\varepsilon}+\varepsilon^2\lambda_1(\Omega_{u_i}) u_{i,\varepsilon},\quad \bar \mu_i:=\Delta \bar u_i
\]
are positive Radon measures, concentrated on $\partial \Omega_{u_{i,\varepsilon}}$ and $\partial \Omega_{\bar u_i}$, respectively. Let $R>0$ and take $\varphi$ to be a smooth cutoff function such that $0\leq \varphi \leq 1$, $\varphi=1$ in $B_R$, $\varphi=0$ in $\R^N\setminus B_{2R}$. Then
\[
\mu_{i,\varepsilon} (B_R) \leq \int_{B_{2R}} \varphi \, d\mu_{i,\varepsilon}=\int_{B_{2R}}-\nabla u_{i,\varepsilon}\cdot \nabla \varphi + \varepsilon^2 \lambda_1(\Omega_{u_i})u_{i,\varepsilon}\varphi\, dx \leq C
\]
and, by testing the equation $-\Delta (\bar u_{i,\varepsilon}-\bar u_i)=\varepsilon^2 \lambda_1(\Omega_{u_{i,\varepsilon}}) u_{i,\varepsilon}-\mu_{i,\varepsilon}+\bar \mu_i$ by $(u_{i,\varepsilon}-\bar u_i)$, we obtain
\begin{align*}
\int_{B_R}|\nabla (u_{i,\varepsilon}-\bar u_i)|^2\leq & \int_{B_{2R}}|\nabla (u_{i,\varepsilon}-\bar u_i)|^2\varphi=-\int_{B_{2R}}(\nabla (u_{i,\varepsilon}-\bar u_i)\cdot \nabla \varphi)(u_{i,\varepsilon}-\bar u_i)\\
        &+ \varepsilon^2 \lambda_1(\Omega_{i,\varepsilon})\int_{B_{2R}}u_{i,\varepsilon}(u_{i,\varepsilon}-\bar u_i)\varphi - \int_{B_{2R}}(u_{i,\varepsilon}-\bar u_i)\varphi \, d\mu_{i,\varepsilon}\\
        &+ \int_{B_{2R}}(u_{i,\varepsilon}-\bar u_i)\varphi \, d\bar \mu_{i} \to 0
\end{align*}
as $\varepsilon\to 0$.
\end{proof}

The rest of the proof of Theorem \ref{thm:blowuplimits} goes as follows:
\begin{enumerate}
\item First, we prove that there exists $a,b>0$ and $\nu\in \partial B_1$ such that
\[
\bar u_1(x)=a(x\cdot \nu)_+,\qquad \bar u_2(x)=b(x\cdot \nu)_-.
\]
(i.e., $\bar u_1,\bar u_2$ are two-plane functions). This is done by showing an Alt-Caffarelli-Friedman  type monotonicity formula for $(u_{1,\varepsilon},u_{2,\varepsilon})$ - see Proposition \ref{ACF} -, which then provides that the Alt-Caffarelli-Friedman functional for the blowup limit $(\bar u_1,\bar u_2)$ is constant.
\item Then, using a suitable variation argument and using the fact that $(u_1,u_2)$ minimizes problem~\eqref{eq:ceta}, we deduce that $a=b$. Observe that, due to the form of the perturbed problem with functional $J_\eta$, we do not have any local minimality condition for the blowup limit $(\bar u_1,\bar u_2)$, which makes this step particularly delicate.
\item Finally, we show that $a\geq m_\eta$, combining Proposition \ref{prop:optimalitycond} with a continuity in $x$ of an Alt-Caffarelli-Friedman type functional.
\end{enumerate}

We proceed with these two steps in the following two subsection.

\subsection{Alt-Caffarelli-Friedman type monotonicity formulas. Blowups are two-plane functions}

In this subsection we prove an Alt-Caffarelli-Friedman type monotonicity formula for subsolutions of divergence-type operators (see Proposition \ref{ACF} below), which we then apply to the blowup sequences $u_{1,\varepsilon}$, $u_{2,\varepsilon}$. Even though monotonicity formulas have been used in several contexts (see e.g. \cite{AltCaffarelliFriedman, Caffarelli88,CJK2002, ContiTerraciniVerziniOptimalNonlinear, ContiTerraciniVerziniVariation, DavidEngelsteinGarciaToro, DiasTavares, MatPet,  SoaveTerracini, Tortone}, for a nonexhaustive list), we could not find in the literature a version we could use directly to our purposes. For that reason, we present here a proof. To start with, we recall the Friedland-Hayman inequality \cite{FH_inequality} (see also \cite{AltCaffarelliFriedman} or \cite[Chapter 12]{CaffarelliSalsa}).

\begin{lemma}\label{lemma:OPPsphere}
Given $\omega\subset \partial B_1$ relatively open, let $\lambda_1(\omega)$ denote the first eigenvalue of the Dirichlet Laplace-Beltrami operator on $\omega$, and let 
\[
\gamma(t):=- \frac{N-2}{2}+\sqrt{\left(\frac{N-2}{2}\right)^2 + t} , \text{ for } t>0.
\]
Then
\[
\gamma(\lambda_1(\omega_1))+\gamma(\lambda_1(\omega_2)) \geq 2 
\]
for every $\omega_1,\omega_2\subset \partial B_1$, relatively open, with $\omega_1\cap \omega_2=\emptyset$.

Moreover, equality is achieved if and only if each $\omega_i$ is a half-sphere.
\end{lemma}

\begin{remark}
The function $\gamma$ arises from the equation $t=\gamma(N-2+\gamma)$, appearing when looking for harmonic functions with homogeneity $\gamma$ on cones.
\end{remark}

From now on we will mostly follow the notation and structure of \cite{stz}, where a one-phase monotonicity formula was shown (check Proposition 2.5 therein), adapting it for the two-phase case (see Proposition \ref{ACF} below).

By Lemmas \ref{lemma:unifbound_ceta} and \ref{lemma:otherproperties}, there exists $\bar \lambda>0$ (independent of $\eta\leq 1$) such that
\begin{equation}\label{eq:choice_barlambda}
\lambda_1(\Omega_{u_i})=\int_\Omega |\nabla u_i|^2\leq \bar \lambda \qquad \text{ for every $i=1,2$.}
\end{equation}
Next, let $\bar R = \bar R (\bar \lambda)> 0$ be such that the ball $B_{2\bar R}=B_{2\bar R}(0)$ has first Dirichlet eigenvalue equal to $\bar \lambda$. We denote by $\varphi$ the corresponding positive eigenfunction, normalized in $L^\infty$:
\begin{equation}\label{eq:eigenfunction_def}
\begin{cases}
	-\Delta \varphi = \bar \lambda \varphi & \text{in $B_{2 \bar R}$},\\
	\varphi = 0 &  \text{on $\partial B_{2 \bar R}$},\\
	\varphi(0) = 1.
\end{cases}
\end{equation}
We recall that $\varphi$ is radially symmetric and radially decreasing, attaining its only maximum at the origin.

\begin{lemma}\label{eq:subsolution}
We have
\begin{equation}\label{eq:divergencetype}
-\div \left(\varphi^2 \nabla \left(\frac{u_i}{\varphi}\right)\right)\leq 0\quad  \text{ in } B_{\bar R},\ \text{ for }i=1,2.
\end{equation}
\end{lemma}
\begin{proof}
This is a direct computation (cf. \cite[Lemma 9.1]{ContiTerraciniVerziniVariation}):
\[
-\div \left(\varphi^2\nabla \left(\frac{u_i}{\varphi}\right)\right)= -\div\left(\nabla u_i \varphi-u_i \nabla \varphi\right)=-\varphi \Delta u_i+u_i\Delta \varphi = u_i \varphi (\lambda_{1}(\Omega_{u_i})-\bar \lambda)\leq 0,
\]
where the last inequality follows from \eqref{eq:choice_barlambda}.
\end{proof}
Based on the fact that $u_1,u_2$ satisfy \eqref{eq:divergencetype}, we now prove an Alt-Caffarelli-Friedman-type monotonicity formula. We assume first that $N\geq 3$. 

With that in mind, let $\Gamma_\varphi \in C^2(B_{3\bar R/2}\setminus\{0\})$ be the Dirichlet Green function of the operator $-\div \left(\varphi^2 \nabla \cdot \right)$ in $B_{3\bar R/2}$, that is,  the solution to
\begin{equation}\label{eq:fundamental_eq}
\begin{cases}
	-\div\left(\varphi^2 \nabla \Gamma_\varphi\right) = \delta & \text{in $B_{3\bar R/2}$},\\
 \Gamma_\varphi=0 & \text{on $\partial B_{3\bar R/2}$,}
\end{cases}
\end{equation}
where $\delta$ is the Dirac delta centered at the origin. This Green function is given explicitly by
\[
\Gamma_\varphi(|x|)=\Gamma_\varphi(r) = (N-2) \int_r^{3 \bar R/2} \frac{s^{1-N}}{\varphi^2(s)} ds,\qquad r=|x|,
\]
and so, in particular, $\displaystyle \Gamma_\varphi'(r)=(2-N)\frac{r^{1-N}}{\varphi^2(r)}$. We also define, for $0<r\le \bar R$, 
\begin{equation}\label{eq:def_psi}
\psi(r):=r^{N-2}\varphi^2(r)\Gamma_\varphi(r),\quad \psi(0):=\lim_{r\to 0^+} \psi(r)=1.
\end{equation}
We recall the following result.
\begin{lemma}[{{\cite[Lemma 2.4]{stz}}}]\label{lem psi}
	 Let $N\geq 3$. We have:
  \begin{enumerate}
      \item $\psi(r) > 0$ for any $r \in [0,\bar R]$;
      \item $\psi(3\bar R/2) = 0$;
      \item there exists $C=C(N, \bar \lambda) > 0$ such that
	\begin{equation}\label{eq:psi}
		|\psi(r)-1| \le C r \qquad \text{for $r\in (0,3\bar R/2)$}, 
	\end{equation}
 and in particular $\psi$ is Lipschitz continuous in $B_{3\bar R/2}$.
  \end{enumerate} 
\end{lemma}

\begin{proposition}[Alt-Caffarelli-Friedman-type monotonicity formula]\label{ACF}
Let $N\geq 3$, and let $w_1,w_2\in C(\R^N)\cap H^1_{loc}(\R^N)$ be nonnegative functions such that $w_1\cdot w_2 \equiv 0$, and
\begin{equation}\label{eq:inequality_subsol}
-\div \left(\varphi^2 \nabla w_i\right)\leq 0 \quad \text{ in } B_{\bar R}\quad \text{ for $i=1,2$.}
\end{equation}
 Then there exist $C = C(N)> 0$ and $\tilde r = \tilde r(N)\in (0,\bar R)$ such that
the function
\begin{align}
	\Psi(w_1,w_2,r) &:= e^{ C r } \frac{1}{r^2} \int_{B_r} \frac{\psi(|x|)}{|x|^{N-2}} \left|\nabla w_1\right|^2\cdot \frac{1}{r^2} \int_{B_r} \frac{\psi(|x|)}{|x|^{N-2}} \left|\nabla w_2\right|^2\nonumber \\
 &=e^{ C r } \frac{1}{r^2} \int_{B_r} \varphi^2 \Gamma_\varphi\left|\nabla w_1\right|^2\cdot \frac{1}{r^2} \int_{B_r} \varphi^2 \Gamma_\varphi\left|\nabla w_2\right|^2.\label{eq:Psi}
\end{align}
	is nondecreasing in $r \in (0,\tilde r)$.
\end{proposition}
\begin{proof}
We follow closely the computations in the proof of \cite[Proposition 2.5]{stz}, to which we refer to for more details. 
	
 \smallbreak
 
\noindent  {\bf Step 1}. Let $i\in \{1,2\}$. By (formally) using $\Gamma_\varphi w_i\geq 0$ as test function in \eqref{eq:inequality_subsol}, and $w_i^2/2$ as test function in  \eqref{eq:fundamental_eq}, we see that
\begin{align}
	\int_{B_r} \frac{\psi(|x|)}{|x|^{N-2}} |\nabla w_i|^2&= \int_{B_r} \varphi^2\Gamma_\varphi |\nabla w_i|^2  
		\leq \int_{\partial B_r}\varphi^2 \Gamma_\varphi w_i (\partial_\nu w_i)- \int_{B_r} \langle \varphi^2 \nabla(\frac{w_i^2}{2}), \nabla \Gamma_\varphi \rangle \nonumber \\
         & =\int_{\partial B_r}\varphi^2 \Gamma_\varphi w_i (\partial_\nu w_i)-\int_{\partial B_r} \varphi^2 (\partial_\nu \Gamma_\varphi) \frac{w_i^2}{2} - \frac{w_i^2(0)}{2} \nonumber\\
         &\leq \int_{\partial B_r}\left( \varphi^2 \Gamma_\varphi w_i (\partial_\nu w_i)- \frac{1}{2} w_i^2 \varphi^2 (\partial_\nu \Gamma_\varphi) \right) = \int_{\partial B_r} \left(\frac{\psi(r)}{r^{N-2}} w_i (\partial_\nu w_i) + \frac{N-2}{2r^{N-1}} w^2\right)   \label{eq:aux_for_later}
\end{align}

\noindent {\bf Step 2}. Denote $\Psi(w_1,w_2,r)$ by $\Psi(r)$. We compute the logarithmic derivative of $\Psi$:
\begin{align*}
	\frac{d}{dr} \log \Psi(r) &= C - \frac{4}{r} + \sum_{i=1}^2\frac{\displaystyle \int_{\partial B_r} \frac{\psi(|x|)}{|x|^{N-2}} |\nabla w_i|^2}{\displaystyle\int_{B_r} \frac{\psi(|x|)}{|x|^{N-2}} |\nabla w_i|^2} \geq C-\frac{4}{r} + \sum_{i=1}^2\frac{\displaystyle  \frac{\psi(r)}{r^{N-2}}\int_{\partial B_r}  |\nabla w_i|^2}{\displaystyle \int_{\partial B_r}\left( \frac{\psi(r)}{r^{N-2}}w_i (\partial_\nu w_i) + \frac{N-2}{2r^{N-1}}  w_i^2 \right) }.
\end{align*}
For $i=1,2$, let $v_i=v_i^{(r)} := w_i(rx):\partial B_1\to \R$ and $\omega_{i,r}:=\{x\in \partial B_1:\ v_i^{(r)}(x)>0 \}$. Observe that $\omega_{1,r}\cap \omega_{2,r}=\emptyset$. Moreover, 
\begin{align*}
	\frac{d}{dr} \log \Psi(r) & \geq C - \frac{4}{r} + \frac{\psi(r)}{r} \sum_{i=1}^2\frac{\displaystyle \int_{\omega_{i,r}} |\nabla v_i|^2}{\displaystyle \int_{\omega_{i,r}} \left(\psi(r)v_i(\partial_\nu v_i)+\frac{N-2}{2}v_i^2\right)}\\
					&= C - \frac{4}{r} + \frac{1}{r\psi(r)} \sum_{i=1}^2\frac{\displaystyle \int_{\omega_{i,r}} \left(\psi^2(r)(\partial_\nu v_i)^2 + \psi^2(r) |\nabla_\theta v_i|^2\right)}{\displaystyle \int_{\omega_{i,r}}\left(\psi(r)v_i(\partial_\nu v_i)+\frac{N-2}{2}v_i^2\right)}\\
					&\geq C - \frac{4}{r} + \frac{1}{r\psi(r)} \sum_{i=1}^2 \frac{\displaystyle \int_{\omega_{i,r}} \left(\psi^2(r)(\partial_\nu v_i)^2 + \psi^2(r)\lambda_1(\omega_{i,r}) v_i^2 \right)}{\displaystyle \int_{\omega_{i,r}}\left(\psi(r)v_i(\partial_\nu v_i)+\frac{N-2}{2}v_i^2\right)},\\
\end{align*}
where $\lambda_1(\omega_{i,r})$ denotes the first eigenvalue of the Dirichlet Laplace-Beltrami operator on $\omega_{i,r}$. Since
\begin{align*}
\int_{\omega_{i,r}}\left(\psi(r)v_i(\partial_\nu v_i)+\frac{N-2}{2}v_i^2\right) \leq \int_{\omega_{i,r}} \left( \frac{\psi^2(r)}{2a(N-2)}(\partial_\nu v_i)^2+\frac{(N-2)(a+1)}{2}v_i^2 \right),
\end{align*}
by choosing $a>0$ such that 
\[
\frac{(N-2)(a+1)}{2}=\frac{\psi^2(r)\lambda_1(\omega_r)}{2a(N-2)} \iff a=-\frac{1}{2}+\frac{1}{N-2}\sqrt{\left(\frac{N-2}{2}\right)^2+\psi^2(r)\lambda_1(\omega_r) } = \frac{\gamma(\lambda_1(\omega_r)  \psi^2(r))}{N-2},
\] 
where the function $\gamma(\cdot)$ is as in Lemma \ref{lemma:OPPsphere},
\begin{align}
	\frac{d}{dr} \log \Psi(r) &\geq C - \frac{4}{r} + \frac{2}{r \psi(r)} \sum_{i=1}^2\gamma \left( \psi^2(r) \lambda_1(\omega_{i,r}) \right) \nonumber \\
		&= \frac{2}{r} \left(-2 + \frac{C}{2} r  + \frac{1}{\psi(r)} \sum_{i=1}^2\gamma \left( \psi^2(r) \lambda_1(\omega_{i,r}) \right)\right). \label{eq:lowerbound_ACF_aux}
\end{align}

\smallbreak

\noindent {\bf Step 3}. Observe that
\begin{align*}
|\gamma\left(\psi^2(r)\lambda_1({\omega_{i,r}})\right)-\gamma(\lambda_1(\omega_{i,r}))|&=\frac{|\psi^2(r)-1|\lambda_1({\omega_{i,r}})}{\sqrt{\left(\frac{N-2}{2}\right)^2 + \psi^2(r)\lambda_1({\omega_{i,r}})}+\sqrt{\left(\frac{N-2}{2}\right)^2 + \lambda_1({\omega_{i,r}})}}\\
&\leq \frac{|\psi^2(r)-1|}{\psi(r)+1}\frac{\lambda_1({\omega_{i,r}})}{\sqrt{\lambda_1({\omega_{i,r}})}}=|\psi(r)-1|\frac{\sqrt{\lambda_1({\omega_{i,r}})}}{\gamma(\lambda_1({\omega_{i,r}}))}\gamma(\lambda_1({\omega_{i,r}}))\\
&\leq Cr\gamma(\lambda_1({\omega_{i,r}})),
\end{align*}
using the fact that $t\in \R^+\mapsto \sqrt{t}/\gamma(t)$ is bounded, combined with \eqref{eq:psi}. 
Therefore there exists $C_1>0$ such that, for $r>0$ small,
\begin{align*}
\frac{\gamma(\psi^2(r)\lambda_1(\omega_{i,r}))}{\psi(r)}&\geq 
\frac{\gamma(\lambda_1(\omega_{1,r}))(1-Cr)}{1+Cr} \geq \gamma(\lambda_1(\omega_{1,r})) (1-C_1r).
\end{align*}
Going back to \eqref{eq:lowerbound_ACF_aux} and using Lemma \ref{lemma:OPPsphere} we have, for small $r$,
\begin{align*}
\frac{d}{dr} \log \Psi(r)&\geq \frac{2}{r} \left(-2 + \frac{C}{2} r  + \sum_{i=1}^2 \gamma(\lambda_1(\omega_{i,r}))(1-C_1r)\right)\\
&\geq \frac{2}{r} \left(-2 + \frac{C}{2} r  +  2(1-C_1r) \right)=2\left(\frac{C}{2}-2C_1\right),
\end{align*}
which is positive by choosing $C>4C_1$.

\end{proof}

\begin{lemma}\label{lemma:twoplanesolution}
Let $\bar u_1,\bar u_2$ be the blowup limits as in Lemma \ref{lemma:blowuplimits_def}. Then
\begin{equation}\label{eq:ACF_constant}
\frac{1}{ r^2} \int_{B_{ r}} \frac{\left|\nabla \bar u_{1}\right|^2}{|x|^{N-2}}  \cdot \frac{1}{r^2} \int_{B_{ r}} \frac{\left|\nabla \bar u_{2}\right|}{|x|^{N-2}}\equiv \gamma,
\end{equation}
for $\gamma:=\Psi(u_1/\varphi,u_2/\varphi,0^+):=\lim_{r\to 0^+}\Psi(u_1/\varphi,u_2/\varphi,r)>0$. In particular, $(\bar u_1,\bar u_2)$ is  a two plane configuration: 
\[
\bar u_1(x)=a (x\cdot \nu)_+,\qquad \bar u_2(x)=b (x\cdot \nu)_-,
\]
for some $a,b\geq 0$ with $ab=\gamma$, and $\nu\in \partial B_1$.
\end{lemma}
\begin{proof}
We aim to apply Proposition~\ref{ACF} to the functions $w_i=u_i/\varphi$. 
First, we make a change of variable so that 
\begin{align*}
&e^{ C \varepsilon r } \frac{1}{ r^2} \int_{B_{ r}}\frac{\psi(|\varepsilon x|)}{|x|^{N-2}} \left|\nabla \left(\frac{u_{1,\varepsilon}}{\varphi(\eps x)}\right)\right|^2\,dx \cdot \frac{1}{r^2} \int_{B_{ r}} \frac{\psi(|\varepsilon x|)}{|x|^{N-2}} \left|\nabla \left(\frac{u_{2,\varepsilon}}{\varphi(\eps x)}\right)\right|^2\,dx\\
&=e^{ C \varepsilon r } \frac{1}{ (\varepsilon r)^2} \int_{B_{\varepsilon r}} \frac{\psi(|x|)}{|x|^{N-2}} \left|\nabla \left(\frac{u_{1}}{\varphi}\right)\right|^2\,dx \cdot \frac{1}{(\varepsilon r)^2} \int_{B_{\varepsilon r}} \frac{\psi(| x|)}{|x|^{N-2}} \left|\nabla \left(\frac{u_{2}}{\varphi}\right)\right|^2\,dx=\Psi\left(\frac{u_1}{\varphi},\frac{u_2}{\varphi},\varepsilon r\right).
\end{align*}
By Proposition \ref{ACF}, there exists $\Psi(u_1/\varphi,u_2/\varphi,0^+):=\lim_{\varepsilon\to 0^+}\Psi(u_1/\varphi,u_2/\varphi,\varepsilon r)\equiv \gamma\in [0,\infty)$. On the other hand, since $\frac{\bar u_{i,\varepsilon}}{\varphi(\eps x)}\to \bar u_i$ strongly in $H^1_{loc}(\R^N)$ (recall Lemma \ref{lemma:blowuplimits_def} and the fact that $\varphi(0)=1$), then
\[
e^{ C \varepsilon r } \frac{1}{ r^2} \int_{B_{ r}} \frac{\psi(|\varepsilon x|)}{|x|^{N-2}} \left|\nabla \left(\frac{u_{1,\varepsilon}}{\varphi(\varepsilon x)}\right)\right|^2 \cdot \frac{1}{r^2} \int_{B_{ r}} \frac{\psi(|\varepsilon x|)}{|x|^{N-2}} \left|\nabla \left( \frac{u_{2,\varepsilon}}{\varphi(\varepsilon x)}\right)\right|^2\to \frac{1}{ r^2} \int_{B_{ r}} \frac{\left|\nabla \bar u_{1}\right|^2}{|x|^{N-2}}  \cdot \frac{1}{r^2} \int_{B_{ r}} \frac{\left|\nabla \bar u_{2}\right|}{|x|^{N-2}}
\]
as $\varepsilon\to 0$. In conclusion, the functional of the classical Alt-Caffarelli-Friedman monotonicity formula is constant in $r$, for the functions $\bar u_1, \bar u_2$. Then, by \cite{AltCaffarelliFriedman} (see also \cite[p. 224]{CaffarelliSalsa}), there exists $a,b\geq 0$ and $\nu\in \partial B_1$ such that $\bar u_1(x)=a(x\cdot \nu)_+$, $\bar u_2(x)=b(x\cdot \nu)_-$. By going back to \eqref{eq:ACF_constant}, we deduce that $ab=\gamma$.
\end{proof}

\begin{lemma}\label{lemma:L1_convergence_of_posit.}
Under the notations of Lemma \ref{lemma:blowuplimits_def}, for each $i=1,2$ we have
\[
\chi_{\Omega_{u_{i,\varepsilon}}}\to \chi_{\Omega_{\bar u_i}} \quad \text{ a.e. in $\R^N$, strongly in $L^1_{loc}(\R^N)$}.
\]
\end{lemma}

\begin{proof}
We adapt to our context the proof of the second part of \cite[Lemma 7.4]{RussTreyVelichkov}. For each $i=1,2$, by Lemma \ref{lemma:twoplanesolution}, either $\partial \Omega_{\bar u_i}=\emptyset$ (when $\bar u_i\equiv 0$) or $\partial \Omega_{\bar u_i}$ is the hyperplane $\{x\cdot \nu=0\}$. We prove that, in any case,
$\chi_{\Omega_{u_{i,\varepsilon}}}\to \chi_{\Omega_{\bar u_i}}$ pointwise  in $\R^N\setminus \{x\cdot \nu=0\}$.  Fix $i=1,2$. For $x_0\in \Omega_{\bar u_i}$, by the continuity of $\bar u_i$ there exists $\delta>0$ such that $\bar u_i>0$ in $B_\delta(x_0)$. Since $u_{i,\varepsilon}\to \bar u_i$ strongly in $L^\infty_{loc}(\R^N)$, then $u_{i,\varepsilon}(x_0)>0$ for $\varepsilon>0$ small, and
\[
\chi_{\Omega_{u_{i,\varepsilon}}}(x_0)=1\to 1=\chi_{\Omega_{\bar u_{i}}}(x_0).
\]
Now let $x_0\in \R^N\setminus \overline \Omega_{u_i}$. In particular, $\bar u_i=0$ in $\overline B_\delta(x_0)$, for some $\delta>0$. Assume, in view of a contradiction, that $\bar u_{i,\varepsilon}(x_0)>0$ for all $\varepsilon>0$ sufficiently small.

Then, using the nondegeneracy property of $u_i$ in the form of Remark \ref{rmk:nondeg}, there exists a universal constant $C>0$ and  $x_\varepsilon\in B_\delta (x_0)$ such that
\[
u_{i,\varepsilon}(x_\varepsilon)=\|u_{i,\varepsilon}\|_{L^\infty(B_\delta(x_0))}\geq C\delta.
\]
Then, as $\varepsilon\to 0$ and up to pass to a subsequence, $x_\varepsilon\to x_0\in \overline B_\delta (x_0)$, and by uniform convergence of blowup sequences (Lemma \ref{lemma:blowuplimits_def}),
\[
\bar u_{i}(x_0)\geq C\delta>0, 
\]
a contradiction. Therefore, the sequence of characteristic functions converges a.e. in $\R^N$, while the $L^1$ convergence is a direct consequence of Lebesgue's dominated convergence theorem.
\end{proof}

\begin{lemma}\label{lemma:nontrivial_blowups} Under the notations of Lemma \ref{lemma:blowuplimits_def}, we have $\bar u_1, \bar u_2\not\equiv 0$.
\end{lemma}
\begin{proof}
Let $i=1,2$. By Lemma \ref{le:densityest}, there exist a constant $\xi>0$ such that, for $\varepsilon>0$ small, we have
\[
\xi\leq \frac{|\Omega_{\bar u_i}\cap B_\varepsilon(0)|}{|B_\varepsilon(0)|}=\frac{|\Omega_{u_{i,\varepsilon}}\cap B_1|}{|B_1|}.
\]
On the other hand, by Lemma \ref{lemma:L1_convergence_of_posit.}, 
$|\Omega_{u_{i,\varepsilon}}\cap B_1|\to |\Omega_{\bar u_{i}}\cap B_1|$, and so $\bar u_i\not \equiv 0$.
\end{proof}

\subsection{Conclusion of the classification of blowup limits at two-phase points. Proof of Theorem \ref{thm:blowuplimits}}

So far we have show that $\bar u_1=a(x\cdot \nu)_+$, $\bar u_2=b(x\cdot \nu)_-$, for $a,b>0$ such that $ab=\gamma:=\Psi(u_1/\varphi,u_2/\varphi,0^+)$, together with strong convergences  of blowup sequences. In order to conclude the proof of Theorem \ref{thm:blowuplimits}, it remains to show  that $a=b$.

Let $R>0$. We show that $\bar u_1-\bar u_2$ is harmonic in $B_R$.  Take $0\leq \varphi \in C^\infty_c(B_R)$ and define the blowdown sequence
\[
\varphi_\eps(x)=\eps\varphi\left(\frac{x}{\eps}\right),\qquad \eps>0,
\]
satisfying $\supp(\varphi_\eps) \subset B_{\eps R}$ and the following scaling properties
\begin{equation}\label{eq:scaling}
\|\varphi_\eps\|_1 = \eps^{N+1}\|\varphi\|_1, \; \|\varphi_\eps\|_2^2 = \eps^{N+2}\|\varphi\|_2^2,\;\| \nabla \varphi_\eps\|_2^2 = \eps^{N}\|\nabla\varphi\|_2^2\,\, \mbox{ and }\,\, |\Omega_{\varphi_\eps}| = \eps^N|\Omega_{\varphi}|;
\end{equation}

Since $x_0=0\in   (\partial \Omega_{u_1}\cap \partial\Omega_{u_2}\cap \Omega)\setminus \cup_{\ell\not=1,2}\overline \Omega_{u_\ell}$, the for $\eps>0$ sufficiently small, we have $B_{\eps R}\subset \Omega \setminus \cup_{\ell\not=1,2}\overline \Omega_{u_\ell}$. 
We also observe that:

\begin{equation}\label{eq:scaling2}
\|u_{i}\|_1 = \eps^{N+1}\|u_{i,\eps}\|_1, \; 1=\|u_i\|_2^2 = \eps^{N+2}\|u_{i,\eps}\|_2^2,\;\| \nabla u_i\|_2^2 = \eps^{N}\|\nabla u_{i,\eps}\|_2^2\,\, \mbox{ and }\,\, |\Omega_{u_i}| = \eps^N|\Omega_{u_{i,\eps}}|.
\end{equation}

\begin{lemma}\label{lemma:inequalities_classS_blowup}
    Under the previous notations, there exists $C>0$ (depending on $\varphi$) such that
    \begin{align*}
    0\leq & 2t \int_{B_R} \nabla (u_{1,\eps}-u_{2,\eps})\cdot \nabla \varphi + C t \eps^2 + t^2 \|\nabla \varphi\|_2^2\\
    &+\frac{1}{\eps^N} \left[f_\eta\left(\eps^N|\Omega_{(\hat{u}_{1,\eps}+t\varphi)^+}|+ \eps^N\sum_{i=2}^k|\Omega_{(\hat{u}_{i,\eps}-t\varphi)^+}|\right) -f_\eta\left(\eps^N\sum_{i=1}^k|\Omega_{{u}_{i,\eps}}|\right)\right]
    \end{align*}
    and
    \begin{align*}
    0\leq & 2t \int_{B_R} \nabla (u_{2,\eps}-u_{1,\eps})\cdot \nabla \varphi + C t \eps^2 + t^2 \|\nabla \varphi\|_2^2\\
    &+\frac{1}{\eps^N} \left[f_\eta\left(\eps^N|\Omega_{(\hat{u}_{2,\eps}+t\varphi)^+}|+ \eps^N\sum_{i\neq 2}^k|\Omega_{(\hat{u}_{i,\eps}-t\varphi)^+}|\right) -f_\eta\left(\eps^N\sum_{i=1}^k|\Omega_{{u}_{i,\eps}}|\right)\right] 
    \end{align*}
    for $\eps>0$ sufficiently small (depending on $R>0$) and $t\in (0,1)$. Here, for each $i$, $\hat{u}_{i,\eps}:=u_{i,\eps}-\sum_{j\neq i} u_{j,\eps}$.
\end{lemma}

\begin{proof}
We prove only the first inequality, as the other case is analogous. Using the minimality of ${u}_1, \ldots, {u}_k$, we have
\[
J_\eta({u}_1, \ldots, {u}_k) \leq J_\eta(u_{1,t,\eps}, \ldots, u_{k,t,\eps}),
\]
where, as in \eqref{eq:deformation}, 
\[
u_{t,\eps}=\left(u_{1,t,\eps}, \ldots, u_{k,t,\eps}\right):=\left(\left(\hat{u}_1+t \varphi_\eps\right)^{+}, \left(\hat{u}_2-t \varphi_\eps \right)^{+}, \ldots, \left(\hat{u}_k-t \varphi_\eps\right)^{+}\right),
\]
recalling that $\hat{u}_i=u_i - \sum_{j\neq i} u_j$, for $i=1,\ldots, k$.
This yields
\begin{align*}
\sum_{i=1}^k\int_{\Omega}|\nabla {u}_i|^2 + f_\eta\left(\sum_{i=1}^k|\Omega_{{u}_i}|\right)  \leq & \;\displaystyle \dfrac{\int_{\Omega}|\nabla (\hat{u}_1 + t\varphi_\eps)^+|^2}{\|(\hat{u}_1 + t\varphi_\eps)^+\|^2_2} + \sum_{i=2}^k\dfrac{\int_{\Omega}|\nabla (\hat{u}_i - t\varphi_\eps)^+|^2}{\|(\hat{u}_i - t\varphi_\eps)^+\|^2_2} \\
 & + f_\eta\left(|\Omega_{u_{1,t,\eps}}|+\sum_{i=2}^k|\Omega_{u_{i,t,\eps}}|\right).
\end{align*}
By scaling the inequality above and using both \eqref{eq:scaling} and \eqref{eq:scaling2} we obtain
\begin{align}
\eps^N\sum_{i=1}^k\int_{\Omega_\eps}|\nabla {u}_{i,\eps}|^2 + f_\eta\left(\eps^N\sum_{i=1}^k|\Omega_{{u}_{i,\eps}}|\right)  \leq & \;\displaystyle\dfrac{\int_{\Omega}|\nabla (\hat{u}_1 + t\varphi_\eps)^+|^2}{\|(\hat{u}_1 + t\varphi_\eps)^+\|^2_2} + \sum_{i=2}^k\dfrac{\int_{\Omega}|\nabla (\hat{u}_i - t\varphi_\eps)^+|^2}{\|(\hat{u}_i - t\varphi_\eps)^+\|^2_2} \nonumber \\
 & + f_\eta\left(\eps^N|\Omega_{(\hat{u}_{1,\eps}+t\varphi)^+}|+\eps^N\sum_{i=2}^k|\Omega_{(\hat{u}_{i,\eps}-t\varphi)^+}|\right), \label{eq:scaling_classS_1}
\end{align}
where $\Omega_\eps = \Omega/\eps$. 
By \cite[Lemma A.1]{ASST}, since $\|u_i\|_2=1$ for every $i=1,\ldots, k$, there exists $C_1>0$ such that, for $t>0$ small,
\[
\frac{1}{\|(\hat u_1+t\varphi)^+\|_2^2}\leq 1-2t \int_\Omega u_1 \varphi_\eps + C_1 \|\varphi_\eps\|_2^2t^2,\quad \frac{1}{\|(\hat u_i-t\varphi)^+\|_2^2}\leq 1+2t \int_\Omega u_1 \varphi_\eps + C_1 \|\varphi_\eps\|_2^2t^2\ (i\geq 2),
\]
and so there exists $C_2$ (depending on $\varphi$) such that
\begin{align}
\dfrac{\int_{\Omega}|\nabla (\hat{u}_1 + t\varphi_\eps)^+|^2}{\|(\hat{u}_1 + t\varphi_\eps)^+\|^2_2}\leq &\int_{\Omega}|\nabla (\hat{u}_1 + t\varphi_\eps)^+|^2\left(1-2t\eps^{N+2}\int_{\Omega_\eps} u_{1,\eps}\varphi+ C_1 t^2 \eps^{N+2}\|\varphi\|_2^2\right)\nonumber \\
=& \eps^N \int_{\Omega_\eps}|\nabla (\hat{u}_{1,\eps} + t\varphi)^+|^2-2t \eps^{N+2}\int_{B_R} u_{1,\eps} \varphi   \int_{\Omega}|\nabla (\hat{u}_1 + t\varphi_\eps)^+|^2 \nonumber \\
&+ C_1 t^2 \eps^{N+2}\|\varphi\|_2^2 \int_{\Omega}|\nabla (\hat{u}_1 + t\varphi_\eps)^+|^2\label{eq:scaling_classS_2}\nonumber \\
\leq & \eps^N \int_{\Omega_\eps}|\nabla (\hat{u}_{1,\eps}+t\varphi)^+|^2 + C_2 t \eps^{N+2},
\end{align}
for every $t\in (0,1)$. Here, we have used again \eqref{eq:scaling}, \eqref{eq:scaling2}, as well as the locally uniform convergence of $u_{1,\eps}$. Similarly, for $i\geq 2$,
\begin{align}
\dfrac{\int_{\Omega}|\nabla (\hat{u}_i - t\varphi_\eps)^+|^2}{\|(\hat{u}_i - t\varphi_\eps)^+\|^2_2}\leq & \eps^N \int_{\Omega_\eps}|\nabla (\hat{u}_{i,\eps}-t\varphi)^+|^2 + C_2 t \eps^{N+2}.\label{eq:scaling_classS_3}
\end{align}
Therefore, by plugging \eqref{eq:scaling_classS_2} and \eqref{eq:scaling_classS_3} into \eqref{eq:scaling_classS_1}, and given that  
\[
|\nabla (\hat{u}_{1,\eps}+t\varphi)^+|^2 + \sum_{i=2}^k |\nabla (\hat{u}_{i,\eps}-t\varphi)^+|^2=|\nabla (\hat{u}_{1,\eps}+t\varphi)|^2,
\]
we have
\begin{align*}
\eps^N\sum_{i=1}^k\int_{\Omega_\eps}|\nabla {u}_{i,\eps}|^2  \leq & \,\eps^{N}\int_{\Omega_\eps}|\nabla (\hat{u}_{1,\eps} + t\varphi)|^2 +2C_2 t \eps^{N+2}\\
&+f_\eta\left(\eps^N|\Omega_{(\hat{u}_{1,\eps}+t\varphi)^+}|+ \eps^N\sum_{i=2}^k|\Omega_{(\hat{u}_{i,\eps}-t\varphi)^+}|\right) -f_\eta\left(\eps^N\sum_{i=1}^k|\Omega_{{u}_{i,\eps}}|\right).
\end{align*}
Dividing both sides by $\eps^N$, and using the fact that
\begin{align*}
\int_{\Omega_\eps}|\nabla (\hat{u}_{1,\eps} + t\varphi)|^2&=\sum_{i=1}^k \int_{\Omega_\eps} |\nabla u_{i,\eps}|^2 + 2t \int_{\Omega_\eps} \nabla \hat{u}_{1,\eps}\cdot \nabla\varphi + t^2 \int_\Omega |\nabla \varphi|^2\\
&=\sum_{i=1}^k \int_{\Omega_\eps} |\nabla u_{i,\eps}|^2 + 2t \int_{B_R} \nabla ({u}_{1,\eps}-u_{2,\eps})\cdot \nabla\varphi + t^2 \int_\Omega|\nabla \varphi|^2,
\end{align*}
for $\eps>0$ sufficiently small, we obtain the desired conclusion.
\end{proof}

\begin{lemma}\label{lemma:limitmeas} Under the previous notations, for a.e. $t\in (0,1)$ and $\ell=1,2$, we have
\begin{equation*}\label{eq:limitmeas}
\frac{1}{\eps^N}\left[
f_\eta\left(\eps^N|\Omega_{(\hat{u}_{\ell,\eps}+t\varphi)^+}|+ \eps^N\sum_{i\neq \ell}^k|\Omega_{(\hat{u}_{i,\eps}-t\varphi)^+}|\right) -f_\eta\left(\eps^N\sum_{i=1}^k|\Omega_{{u}_{i,\eps}}|\right)\right] \to 0,    
\end{equation*}
as $\eps\to 0$.
\end{lemma}
\begin{proof} 
We present the proof for $\ell=1$, as for $\ell=2$ is analogous. First, recall that, from \eqref{eq:scaling2} and Lemma \ref{le:measOK}, we have     $\eps^N\sum_{i=1}^k|\Omega_{{u}_{i,\eps}}| = \sum_{i=1}^k|\Omega_{{u}_i}| = a$, which yields to 
\[
\frac{1}{\eps^N}f_\eta\left(\eps^N\sum_{i=1}^k|\Omega_{{u}_{i,\eps}}|\right) = \frac{\eta}{\eps^N}\left(\eps^N\sum_{i=1}^k|\Omega_{{u}_{i,\eps}}| - a\right) = 0. 
\]
Let us denote $m_{1,\eps}=m_{1,\eps}(t):= \eps^N|\Omega_{(\hat{u}_{1,\eps}+t\varphi)^+}|+ \eps^N\sum_{i=2}^k|\Omega_{(\hat{u}_{i,\eps}-t\varphi)^+}|$. We have two possibilities, either $m_{1, \eps} \leq \eps^N\sum_{i=1}^k|\Omega_{{u}_{i,\eps}}|=a$ or $m_{1, \eps} > \eps^N\sum_{i=1}^k|\Omega_{{u}_{i,\eps}}|=a$.\\

\noindent{\it Case 1:} If $m_{1, \eps} \leq a$, then 
\begin{align*}
\frac{1}{\eps^N}f_\eta\left(\eps^N|\Omega_{(\hat{u}_{1,\eps}+t\varphi)^+}|+ \eps^N\sum_{i=2}^k|\Omega_{(\hat{u}_{i,\eps}-t\varphi)^+}|\right) & = \frac{\eta}{\eps^N}\left(\eps^N|\Omega_{(\hat{u}_{1,\eps}+t\varphi)^+}|+ \eps^N\sum_{i=2}^k|\Omega_{(\hat{u}_{i,\eps}-t\varphi)^+}|- a\right) \\
& = \eta\left(|\Omega_{(\hat{u}_{1,\eps}+t\varphi)^+}|+ \sum_{i=2}^k|\Omega_{(\hat{u}_{i,\eps}-t\varphi)^+}|- \frac{a}{\eps^N}\right), 
\end{align*}
and hence
\begin{align*}
\frac{1}{\eps^N}f_\eta\left(m_{1,\eps}\right) - \frac{1}{\eps^N}f_\eta\left(\eps^N\sum_{i=1}^k|\Omega_{{u}_{i,\eps}}|\right) & = \eta\left( |\Omega_{(\hat{u}_{1,\eps}+t\varphi)^+}|+ \sum_{i=2}^k|\Omega_{(\hat{u}_{i,\eps}-t\varphi)^+}| - \sum_{i=1}^k|\Omega_{{u}_{i,\eps}}|\right).\\    
\end{align*}

\noindent{\it Case 2:} On the other hand, if $m_{1, \eps} > a$, then  we have 
\begin{align*}
\frac{1}{\eps^N}f_\eta\left(m_{1,\eps}\right) - \frac{1}{\eps^N}f_\eta\left(\eps^N\sum_{i=1}^k|\Omega_{{u}_{i,\eps}}|\right) & =\frac{1}{\eps^N\eta}\left(\eps^N|\Omega_{(\hat{u}_{1,\eps}+t\varphi)^+}|+ \eps^N\sum_{i=2}^k|\Omega_{(\hat{u}_{i,\eps}-t\varphi)^+}| - \eps^N\sum_{i=1}^k|\Omega_{{u}_{i,\eps}}|\right)\\
& = \frac{1}{\eta}\left(|\Omega_{(\hat{u}_{1,\eps}+t\varphi)^+}|+ \sum_{i=2}^k|\Omega_{(\hat{u}_{i,\eps}-t\varphi)^+}| - \sum_{i=1}^k|\Omega_{{u}_{i,\eps}}|\right).
\end{align*}

In both cases, we need to examine the quantity 
\[
\tilde{m}_{1,\eps}(t) := |\Omega_{(\hat{u}_{1,\eps}+t\varphi)^+}|+ \sum_{i=2}^k|\Omega_{(\hat{u}_{i,\eps}-t\varphi)^+}| - \sum_{i=1}^k|\Omega_{\bar{u}_{i,\eps}}|,
\]
and show that, for a.e. $t>0$, $\lim_{\eps\to 0}\tilde{m}_{1, \eps}(t) = 0$, so that the claim of the lemma holds. Observe that,  since $\supp\varphi \subset B_R \subset \Omega\setminus\cup_{\ell \geq 3}^k \Omega_{u_{\ell,\eps}}$, we have $\Omega_{(\hat{u}_{i,\eps} - t\varphi)^+} = \Omega_{{u}_{i,\eps}}$, for $i=3,\ldots, k$. 
Moreover, 
\[
\Omega_{(\hat{u}_{1,\eps} + t\varphi)^+} = \left(\Omega_{({u}_{1, \eps}-{u}_{2,\eps}+t\varphi)^+}\cap B_R\right)\cup\left(\Omega_{{u}_{1,\eps}} \cap(B_R)^c\right),
\]
and
\[
\Omega_{(\hat{u}_{2,\eps} - t\varphi)^+} = \left(\Omega_{({u}_{2, \eps}-{u}_{1,\eps}-t\varphi))^+}\cap B_R\right)\cup\left(\Omega_{{u}_{2,\eps}} \cap(B_R)^c\right).
\]
Hence, 
\begin{align*}
\tilde{m}_{1,\eps}(t) = &\, |\Omega_{({u}_{1, \eps}-{u}_{2,\eps}+t\varphi)^+}\cap B_R|+|\Omega_{{u}_{1,\eps}} \cap(B_R)^c| + |\Omega_{({u}_{2, \eps}-{u}_{1,\eps}-t\varphi)^+}\cap B_R| + |\Omega_{{u}_{2,\eps}} \cap(B_R)^c| \\
 & + \sum_{i=3}^k|\Omega_{{u}_{i,\eps}}| - \sum_{i=1}^k|\Omega_{{u}_{i,\eps}}| \\
= & \, |\Omega_{({u}_{1, \eps}-{u}_{2,\eps}+t\varphi)^+}\cap B_R| + |\Omega_{({u}_{2, \eps}-{u}_{1,\eps}-t\varphi)^+}\cap B_R| -|\Omega_{{u}_{1,\eps}} \cap B_R| - |\Omega_{{u}_{2,\eps}} \cap B_R| \\
= & \, |\Omega_{({u}_{1, \eps}-{u}_{2,\eps}+t\varphi)^+}\cap B_R| + |\Omega_{({u}_{1, \eps}-{u}_{2,\eps}+t\varphi)^-}\cap B_R| -|\Omega_{{u}_{1,\eps}} \cap B_R| - |\Omega_{{u}_{2,\eps}} \cap B_R| \\
= & \, |\{x \in B_R \,:\, ({u}_{1, \eps}-{u}_{2,\eps}+t\varphi)(x) \neq 0\}| -|\Omega_{{u}_{1,\eps}} \cap B_R| - |\Omega_{{u}_{2,\eps}} \cap B_R|.
\end{align*}
By passing the limit as $\eps\to 0$, and using Lemmas \ref{lemma:twoplanesolution} and \ref{lemma:L1_convergence_of_posit.}, we obtain
\begin{align*}
 \lim_{\eps\to 0}\tilde{m}_{1,\eps}(t) &= |\{x \in B_R \,:\, (a(x\cdot \nu)^+-b(x\cdot \nu)^-+t\varphi)(x) \neq 0\}| - |B_R|\\
 &=|\{x \in B_R \,:\, a(x\cdot \nu)^+-b(x\cdot \nu)^-+t\varphi(x) = 0\}|.
\end{align*}
Finally, since $\varphi\geq 0$,
\begin{align*}
\{x \in B_R \,:\, a(x\cdot \nu)^+-&b(x\cdot \nu)^-+t\varphi(x) = 0\}  = \{x \in B_R \,:\, a(x\cdot \nu)^+-b(x\cdot \nu)^- = -t\varphi(x)\} \\
& \subseteq \{x \in B_R \,:\,  (x\cdot \nu)< 0,\ t\varphi(x)=b(x\cdot \nu)\}\cup \{x \in B_R \,\big|\, x\cdot \nu=0\} \\
& = \left\{x \in B_R \,\big|\, (x\cdot \nu)< 0,\ \frac{\varphi(x)}{(x\cdot \nu)}=\frac{b}{t}\right\}\cup \{x\in B_R\,:\, x\cdot \nu=0\}. 
\end{align*}
Since $|\{x\in B_R\,:\, x\cdot \nu=0\}|=0$ and, by Sard's Theorem, $\left|\left\{x \in B_R \,:\, (x\cdot \nu)< 0,\ \frac{\varphi(x)}{(x\cdot \nu)}=\frac{b}{t}\right\}\right|=0$ for a.e. $t>0$, the conclusion follows. 
\end{proof}

\begin{lemma}\label{lemma:blowup_harmonic}
Let $\bar u_1,\bar u_2$ be as in Lemma \ref{lemma:blowuplimits_def}. Then
\[
\Delta (\bar u_{1}-\bar u_{2})=0 \quad \text{ in } \R^N
\]
and, in particular, $a=b$.
\end{lemma}
\begin{proof}

By Lemma \ref{lemma:limitmeas}, there exists $t_n\to 0$ such that, for each $n$, 
\[
\frac{1}{\eps^N}\left[
f_\eta\left(\eps^N|\Omega_{(\hat{u}_{1,\eps}+t_n\varphi)^+}|+ \eps^N\sum_{i=2}^k|\Omega_{(\hat{u}_{i,\eps}-t_n\varphi)^+}|\right) -f_\eta\left(\eps^N\sum_{i=1}^k|\Omega_{{u}_{i,\eps}}|\right)\right] \to 0  
\]
as $\eps\to 0$. By Lemma \ref{lemma:inequalities_classS_blowup}, 
\begin{align*}
    0\leq & 2t_n \int_{B_R} \nabla (u_{1,\eps}-u_{2,\eps})\cdot \nabla \varphi + C t_n \eps^2 +t^2\|\nabla \varphi\|_2^2\\
    & + \frac{1}{\eps^N} \left[f_\eta\left(\eps^N|\Omega_{(\hat{u}_{1,\eps}+t_n\varphi)^+}|+ \eps^N\sum_{i=2}^k|\Omega_{(\hat{u}_{i,\eps}-t_n\varphi)^+}|\right) -f_\eta\left(\eps^N\sum_{i=1}^k|\Omega_{{u}_{i,\eps}}|\right)\right]
    \end{align*}
By letting $\eps\to 0$, and using the convergences of the blowup sequences provided by Lemma \ref{lemma:blowuplimits_def}, we have
\[
0\leq 2t_n \int_{B_R} \nabla (\bar u_{1}-\bar u_{2})\cdot \nabla \varphi + t_n^2\|\nabla \varphi\|_2^2. 
\]
Finally, we divide the inequality above by $t_n$ and let $n\to \infty$ to see that
\[
0\leq \int_{B_R}\nabla(\bar{u}_{1}-\bar{u}_2)\cdot\nabla\varphi.
\]
Since $R>0$ and $\varphi$ are arbitrary, this means that
\[
-\Delta(\bar{u}_{1}-\bar{u}_2) \geq 0 \quad \mbox{ in } \R^N.
\]
By doing a similar computation, we can also show that 
\[
-\Delta(\bar{u}_{2}-\bar{u}_1) \geq 0 \quad \mbox{ in } \R^N,
\]
hence $\bar u_1-\bar u_2$ is harmonic in $\R^N$.
\end{proof}

\begin{proof}[Conclusion of the proof of Theorem \ref{thm:blowuplimits}]
This result is now a direct consequence of Lemmas \ref{lemma:blowuplimits_def}, \ref{lemma:L1_convergence_of_posit.}, \ref{lemma:nontrivial_blowups} and \ref{lemma:blowup_harmonic}.
\end{proof}

Finally, we relate the information obtained at one-phase reduced boundaries (Proposition \ref{prop:optimalitycond}) with the one obtained at all two-phase points (Theorem \ref{thm:blowuplimits}).

\begin{proposition}\label{prop:m(x0)>=m_eta}
Let $i\neq j$ and  take $x_0\in \Gamma_{TP}(\partial \Omega_i)\cap \Gamma_{TP}(\partial \Omega_j)$.   
Then $m(x_0)\geq m_\eta$, where $m_\eta$ is the constant appearing in \eqref{eq:m_eta}.
\end{proposition}
\begin{proof}
Take $x_0\in \partial\Omega_{u_i}\cap \partial\Omega_{u_j}\cap \Omega$, and let $x_1\in \partial^*\Omega_{u_l}$ be a one-phase point for some $l\in \{1,\dots, k\}$ - which exists, by Lemma \ref{le:onephase}.
Take the blowup sequences
\[
u_{i,\eps}(x)=\frac{u_i(x_0+\eps x)}{\eps},\quad u_{l,\eps}(x)=\frac{u_l(x_1+\eps x)}{\eps}.
\]
By Propositions \ref{prop:blowupconv} and \ref{prop:optimalitycond}, there exists $\nu_{u_l}(x_1)$, the measure theoretical inner normal to the boundary of $\partial \Omega_{u_l}$ at $x_1$, and
\[
u_{l,\eps}\to m_\eta(x\cdot \nu_{u_l}(x_1))_+ \quad \text{ in } L_{loc}^\infty\cap H^1_{loc},\qquad \frac{\Omega_{u_l}-x_1}{\eps}\to \{x\cdot \nu_{u_l}(x_1)>0\}
\quad \text{a.e. and $L^1_{loc}$} .
\]
On the other hand, by Theorem \ref{thm:blowuplimits}, up to considering a subsequence, there exists $\nu(x_0)\in \partial \Omega_{u_i}$ such that 
\[
u_{i,\eps}\to m(x_0)(x\cdot \nu(x_0))_+ \quad \text{ in } L_{loc}^\infty\cap H^1_{loc},\qquad \frac{\Omega_{u_i}-x_0}{\eps}\to \{x\cdot \nu(x_0)>0\}
\quad \text{a.e. and $L^1_{loc}$},
\]
(observe that while $\nu(x_0)$ may depend on the sequence, this is not the case for $m(x_0)$). Now we simply repeat the same procedure of the proof of Proposition~\ref{prop:optimalitycond}, considering a diffeomorphism that pushes inward $\Omega_{u_i}$ at $x_0$ (using $\nu(x_0)$ as normal) and pushes outward $\Omega_{u_l}$ at $x_1$ (using $\nu_{u_l}(x_1)$). More precisely, we set
\[
\tau_{\rho,\kappa}(x)=\tau(x):=x+\kappa\rho\varphi\left(\frac{|x-x_0|}{\rho}\right)\nu(x_0) \chi_{B_\rho(x_0)\cap \Omega_{u_i}}-\kappa\rho\varphi\left(\frac{|x-x_1|}{\rho}\right)\nu_{u_l}(x_1)\chi_{B_\rho(x_1)}.
\]
and consider $\widetilde u_{\rho,i}(z)=u_{i}(\tau^{-1}(z))$, $\widetilde u_{\rho,l}(z)=u_{l}(\tau^{-1}(z))$. For sufficiently small $\kappa,\rho$, we use 
\[
(u_1,\dots, \widetilde u_{\rho,i},\dots,\widetilde u_{\rho,l},\dots, u_k)\in \overline H
\] as a test function for $J_\eta$. By repeating the proof of Proposition \ref{prop:optimalitycond} (which relies only on the convergence statements for blowup sequences!) we obtain this time
\[
0\leq \kappa \rho^N C(\varphi)((m(x_0))^2-(m_\eta)^2)+o(\rho^N)+\rho^No(\kappa)
\] 
and so $m(x_0)\geq m_\eta$, as wanted.
\end{proof}

\section{Optimality conditions in the viscosity sense. Conclusion of the proof}\label{sec:viscosity}

From now on, we fix $\eta>0$ small enough, and take an optimal vector $(u_1,\dots,u_k)\in \overline{H}$ for problem~\eqref{eq:ceta}, such that
$(\Omega_{u_1},\ldots, \Omega_{u_k})$  is an optimal partition of the original constrained optimal partition problem  \eqref{eigenvalue_problem} (recall Remark~\ref{rmk:etafixed}).
In this section, we show that $(u_1,\dots, u_k)$  satisfies, in the viscosity sense, suitable optimality conditions at the free boundary. This will allow to apply the regularity result of~\cite{DePhilippisSpolaorVelichkov} at two-phase points, and the regularity result of~\cite{RussTreyVelichkov} at one-phase points, concluding the proof of Theorem \ref{thm:main}.

First of all, we introduce the notion of viscosity solution of a PDE.
\begin{definition}
Let $\Omega\subset\R^N$ be an open set. 
We say that a continuous function $Q\colon \Omega\to \R$ \emph{touches from below} (resp. \emph{touches from above}) a function $w\colon \Omega\to \R$ at $x_0\in \Omega$ if $Q(x_0)=w(x_0)$ and \[
Q(x)\leq w(x), \qquad (\text{resp. }Q(x)\geq w(x)),\qquad\text{for all $x$ in a neighborhood of $x_0$}.
\]
\end{definition}
\begin{definition}
We say that $Q\colon \Omega\to \R$ is an admissible \emph{comparison function} in $\Omega$ if 
\begin{itemize}
\item $Q\in C^1(\overline{\{Q>0\}}\cap \Omega)\cap C^1(\overline{\{Q<0\}}\cap \Omega)$;
\item $Q\in C^2({\{Q>0\}}\cap \Omega)\cap C^2({\{Q<0\}}\cap \Omega)$;
\item $\partial\{Q>0\}\cap \Omega$ and $\partial\{Q<0\}\cap \Omega$ are smooth manifolds in $\Omega$.
\end{itemize}
\end{definition}

We start with the results for the one-phase points. Recall that $m_\eta$ is the same constant of Proposition~\ref{prop:optimalitycond} and~\ref{prop:m(x0)>=m_eta}.

\begin{proposition}\label{prop:viscosityonephase}
Let $(u_1,\dots,u_k)$ be optimal for~\eqref{eq:ceta}. Let $i\in\{1,\dots,k\}$ and $x_0\in  \left(\partial\Omega_{u_i}\setminus \cup_{j\not=i}\partial\Omega_{u_j}\right)\cap \Omega$ be an interior one-phase point. We have the following:
\begin{enumerate}
\item[(i)] Suppose that $Q$ is a comparison function and $Q_+$ touches $u_i$ from below at $x_{0}$, then $\left|\nabla Q_{+}\left(x_{0}\right)\right|^2 \leq m_\eta$;
\item[(ii)] Suppose that $Q$ is a comparison function  and $Q_+$ touches $u_i$ from above at $x_{0}$, then $\left|\nabla Q_{+}\left(x_{0}\right)\right|^2 \geq m_\eta$.
\end{enumerate}
Moroever, $\Gamma_{OP}(\partial\Omega_{i})$, that is the one-phase free boundary of $\Omega_{u_i}\cap \Omega$, can be decomposed in a regular part $\Reg_i$ and in a singular part $\Sing_i$ such that
\begin{itemize}
\item $\Reg_i$ is locally the graph of a $C^{1,\alpha}$ function for any $\alpha\in [0,1)$,
\item There exists a universal constant $N^*\in\{5,6,7\}$ such that $\Sing_i$ is empty if $N<N^*$, discrete if $N=N^*$, and $\dim_H(\Sing_i)\leq N-N^*$ otherwise.
\end{itemize}
\end{proposition}
When the claim of Proposition~\ref{prop:viscosityonephase} holds, we say that $|\nabla u_i|^2=m_\eta$ on $(\partial \Omega_{u_i}\setminus \cup_{j\not=i}\partial\Omega_{u_j})\cap \Omega$ in the sense of viscosity solutions.

\begin{remark}
The constant $N^*$ is the lowest dimension at which there exists singular local minimizers for the Alt-Caffarelli one-phase functional in $\R^N$. 
\end{remark}

\begin{proof}
In view of Remark~\ref{rem:connected}, we note that the set $\Omega_{u_i}$ is a solution to the shape optimization problem:\begin{equation}\label{eq:minonephasepenalized}
\min \left\{\lambda_1(\omega) : \omega\subset \Omega\setminus \cup_{j\not=i}\overline \omega_j,\; |\omega|=a-\sum_{j\not=i}|\omega_j|\right\}.
\end{equation}

Thanks to~\cite[Lemma~5.30]{RussTreyVelichkov} we deduce the claim with a certain constant $m_i>0$ on the one-phase free boundary, namely $(\partial \Omega_{u_i}\setminus \cup_{j\not=i}\partial\Omega_{u_j})\cap \Omega$.
The fact that the constant $m_i$ is exactly $m_\eta$ (and does not depend on $i$) follows thanks to Proposition~\ref{prop:optimalitycond}, and performing a blowup at a one-phase point of the reduced boundary.
The last claim concerning the regular and singular part also follows from~\cite[Theorem~1.2]{RussTreyVelichkov}.
\end{proof}

We now focus on the two-phase points.
Let $(u_1,\dots, u_k)$ be a solution for problem~\eqref{eq:ceta} and $x_0\in \partial \Omega_{u_i}\cap \partial\Omega_{u_j}\setminus \cup_{l\not=i,j}\Omega_{u_l}$ be a two-phase point. Up to relabeling the indexes, we can assume that $i=1,\;j=2$. We denote $u=u_1-u_2$ and since there are no triple points (and no two-phase points at the boundary of the box $\Omega$), we take $R>0$ sufficiently small such that $B_R(x_0) \subset \Omega\setminus\cup_{\ell \geq 3}^k \Omega_{u_\ell}$.
We can now prove that, at two-phase points (branching and not branching), the following optimality conditions, in the sense of viscosity solutions, hold true:
\[
|\nabla u_1|^2\ge m_\eta,\quad |\nabla u_2|^2\geq m_\eta, \quad\text{and}\quad |\nabla u_1|^2-|\nabla u_2|^2=0\quad\text{on}\quad \partial\Omega_{u_1}\cap \partial\Omega_{u_2}\cap \Omega,
\]
where $m_\eta>0$ is the constant from Proposition~\ref{prop:m(x0)>=m_eta}, see also Theorem~\ref{thm:blowuplimits}.
Let us state it in a more precise way.
\begin{proposition}\label{prop:viscositysol}
With the notation above, we have the following:
\begin{enumerate}
\item[(i)] If $Q$ is a comparison function touching from \emph{below} $u=u_1-u_2$ at a two-phase point $x_0$, then \[
|\nabla Q^-(x_0)|^2\geq m_\eta,\qquad\text{and},\qquad |\nabla Q^+(x_0)|^2-|\nabla Q^-(x_0)|^2\leq 0.
\] 
\item[(ii)] If $Q$ is a comparison function touching from \emph{above} $u=u_1-u_2$ at a two-phase point $x_0$, then \[
|\nabla Q^+(x_0)|^2\geq m_\eta,\qquad\text{and},\qquad |\nabla Q^+(x_0)|^2-|\nabla Q^-(x_0)|^2\geq 0.
\] 
\end{enumerate}
\end{proposition}
\begin{proof}
This proof, since we now have the information on the behavior of the blowup limits at two-phase points (see Theorem~\ref{thm:blowuplimits}), follows as~\cite[Lemma~2.5]{DePhilippisSpolaorVelichkov}. 
For the sake of completeness, we report here the proof of part $(i)$ (the other case is analogous). Let $x_{0}\in \partial \Omega_{u_1}\cap \partial\Omega_{u_2}\setminus \cup_{l\geq 3}\Omega_{u_l}$ and $Q$ be an admissible function that touches $u=u_1-u_2$ from below at $x_{0}$. Let $u_{x_{0}, r_{n}}$ and $Q_{x_{0}, r_{n}}$ be blowup sequences of $u$ and $Q$ at $x_{0}$. Then, up to extracting a subsequence, we can assume that $u_{x_{0}, r_{n}}$ converges uniformly to a blowup limit of the form (see Theorem~\ref{thm:blowuplimits})
\[
u_{\infty}(x)=m(x_0)(x\cdot \nu)^+- m(x_0)(x\cdot \nu)^-,
\]
for some $\nu\in \partial B_1$. Recall also from Proposition \ref{prop:m(x0)>=m_eta} that $m(x_0)\geq m_\eta$.

On the other hand, since $Q^{+}$and $Q^{-}$are differentiable at $x_{0}$ (respectively in $\overline{\{Q>0\}}$ and $\overline{\{Q<0\}}$), we deduce that $Q_{x_{0}, r_{n}}$ converges to the function
\[
Q_\infty(x)=\left|\nabla Q^{+}\left(x_{0}\right)\right|\left(x \cdot \nu^{\prime}\right)_{+}-\left|\nabla Q^{-}\left(x_{0}\right)\right|\left(x \cdot \nu^{\prime}\right)_{-},
\]
where $\nu^{\prime}=\left|\nabla Q^{+}\left(x_{0}\right)\right|^{-1} \nabla Q^{+}\left(x_{0}\right)=-\left|\nabla Q^{-}\left(x_{0}\right)\right|^{-1} \nabla Q^{-}\left(x_{0}\right)$. Now since, $Q_\infty$ touches $u_\infty$ from below, we have that $\nu^{\prime}=\nu$,
\[
\begin{aligned}
\left|\nabla Q^{+}\left(x_{0}\right)\right|^{2}-\left|\nabla Q^{-}\left(x_{0}\right)\right|^{2} \leq 0 
 \quad \text { and } \quad\left|\nabla Q^{+}\left(x_{0}\right)\right| \leq m_\eta, \quad\left|\nabla Q^{-}\left(x_{0}\right)\right| \geq m_\eta,
\end{aligned}
\]
thus $(i)$ is proved.
\end{proof}

We are now in position to prove the regularity result at two-phase points, again following~\cite{DePhilippisSpolaorVelichkov} (where the most difficult step, namely the improvement of flatness, is proved).
\begin{theorem}\label{thm:regtwophase}
Let $(u_1,\dots, u_k)$ be a solution for problem~\eqref{eq:ceta} and $x_0\in \partial \Omega_{u_i}\cap \partial\Omega_{u_j}\setminus \cup_{l\not=i,j}\Omega_{u_l}$ be a two-phase point.
There exists $r_0>0$ (depending possibly also on $x_0$) such that $\partial \Omega_{u_i}\cap B_{r_0}(x_0)$ and $\partial \Omega_{u_j}\cap B_{r_0}(x_0)$ are $C^{1,\alpha}$ graphs for all $\alpha\in[0,1/2)$.
\end{theorem}
\begin{proof}
 As usual, up to relabeling the indexes, we can assume that $i=1,j=2$. We denote $u=u_1-u_2$ and since there are no triple points, we take $R>0$ sufficiently small such that $B_R(x_0) \subset \Omega\setminus\cup_{\ell \geq 3}^k \Omega_{u_\ell}$.

We divide the proof in some steps.

{\bf Step 1.} 
 There is a unique blowup of $u$ at $x_0$.
Moreover there exists $\alpha>0$ such that for every open set $D^{\prime} \Subset B_R(x_0)$ there is a constant $C\left(D^{\prime}, N\right)>0$ such that, for every two-phase points $x_{0}, y_{0} \in \Big(\partial \Omega_{u_1}\cap \partial\Omega_{u_2}\setminus \cup_{l\geq 3}\Omega_{u_l}\Big)\cap  D^{\prime}$, we have
\begin{equation*}
\left|m\left(x_{0}\right)-m\left(y_{0}\right)\right| \leq C\left|x_{0}-y_{0}\right|^{\alpha} \quad \text { and } \quad\left|\nu\left(x_{0}\right)-\nu\left(y_{0}\right)\right| \leq C_{0}\left|x_{0}-y_{0}\right|^{\alpha} \text {, } \tag{4.1}
\end{equation*}
where \[
m\left(x_{0}\right) (x\cdot \nu(x_0))^+-m(x_0)(x\cdot \nu(x_0))^-\qquad\text{and}\qquad m\left(y_{0}\right) (x\cdot \nu(y_0))^+-m(y_0)(x\cdot \nu(y_0))^-\] 
denote the blowup limits of $u$ at $x_{0}$ and $y_{0}$ respectively. 
In particular, $ \Big(\partial \Omega_{u_1}\cap \partial\Omega_{u_2}\setminus \cup_{l\geq 3}\Omega_{u_l}\Big)\cap D^{\prime}$ is locally a closed subset of the graph of a $C^{1, \alpha}$ function.
\begin{proof}[Proof of Step 1]
This follows from the improvement of flatness result at two-phase points~\cite[Theorem~4.3]{DePhilippisSpolaorVelichkov}.
\end{proof}

{\bf Step 2.} There are $C^{0, \alpha}$ continuous functions $m_1: \partial \Omega_{u}^{+} \rightarrow \mathbb{R}, m_2: \partial \Omega_{u}^{-} \rightarrow \mathbb{R}$ such that $m_1 \geq m_\eta, m_2 \geq m_\eta$, and $u_1,u_2$ are viscosity solutions of the one-phase problem
\[
-\Delta u_1=\lambda_1(\Omega_{u_1})u_1 \quad \text { in } \quad \Omega_{u_1}, \quad\left|\nabla u_1\right|=m_1 \quad \text { on } \quad \partial \Omega_{u_1}
\]
and
\[
-\Delta u_2=\lambda_1(\Omega_{u_2})u_2 \quad \text { in } \quad \Omega_{u_2}, \quad\left|\nabla u_2\right|=m_2 \quad \text { on } \quad \partial \Omega_{u_2} .
\]
\begin{proof}[Proof of Step 2]
We show the proof only for $u_1$, being the other case analogous. We already know the equation satisfied by $u_1$ in $\Omega_{u_1}$, thanks to Lemma~\ref{lemma:otherproperties}.
By~\cite[Theorem~4.3]{DePhilippisSpolaorVelichkov}, we infer that for all two-phase points $x_0\in D'$, we have \begin{equation}\label{eq:4.3}
|u_1(x)-m_1\left(x_{0}\right)\left(x-x_{0}\right) \cdot \nu\left(x_{0}\right)| \leq C_{0}\left|x-x_{0}\right|^{1+\gamma}, 
\end{equation}
where $x\in B_{r_0}(x_0)\cap \Omega_{u_1}$, with $\gamma\in(0,1/2)$, $0<r_0<R$ and $C_0>0$ depending only on $D'$.

To conclude we only need to prove that $m_1\in C^{0, \alpha}\left(\partial \Omega_{u_1}\right)$. Since $m_1$ is $\alpha$-H\"older continuous on the set of two-phase points by Step 1 and constant on the set of one-phase points, thanks to Proposition~\ref{prop:viscosityonephase}, we just need to show that if $x_{0}$ is a two-phase point such that there is a sequence $x_{n}$ of one-phase points converging to $x_{0}$,
then $m_1\left(x_{0}\right)=m_\eta$. To this end, let $y_{n}\in \partial \Omega_{u_1}\cap \partial\Omega_{u_2}\cap B_R(x_0)$ be such that
\[
\operatorname{dist}\left(x_{n}, \partial \Omega_{u_1}\cap \partial\Omega_{u_2}\cap B_R(x_0)\right)=\left|x_{n}-y_{n}\right| .
\]
Then we set
\[
r_{n}=\left|x_{n}-y_{n}\right| \quad \text { and } \quad u_{1,n}(x)=\frac{1}{r_{n}} u_1\left(x_{n}+r_{n} x\right)
\]
and note that $u_{1,n}$ is a viscosity solution of the free boundary problem
\[
-\Delta u_{1,n}=r_n\lambda_1(\Omega_1)u_{1,n} \quad \text { in } \quad \Omega_{u_{1,n}} \cap B_{1}, \quad\left|\nabla u_{1,n}\right|=m_\eta \quad \text { on } \quad \partial\Omega_{u_{1,n}} \cap B_{1} .
\]
Since $u_{1,n}$ are uniformly Lipschitz they converge to a function $u_{1,\infty}$ which is also a viscosity solution of the same problem (thanks to De Silva regularity paper for viscosity solutions~\cite{DeSilva}). On the other hand, by~\eqref{eq:4.3}, we have that
\[
u_{1,\infty}(x)=m_1\left(x_{0}\right)\left(x \cdot \nu\left(x_{0}\right)\right)^{+},
\]
which gives that $m_1\left(x_{0}\right)=m_\eta$.
\end{proof}

{\bf Step 3. Conclusion of the proof of Theorem~\ref{thm:regtwophase}.}
Let $x_{0}\in \partial \Omega_{u_1} \cap \partial \Omega_{u_2}\cap B_R(x_0)$ and let $\bar{\varepsilon}$ be the constant in~\cite[Theorem 1.1]{DeSilva}. Thanks to the classification of blowups at two-phase points, we can choose $r_{0}>0$, depending on $x_{0}$, such that
\[
\left\|u_{x_{0}, r_{0}}-m_\eta (x\cdot \nu)^\pm\right\|_{\infty}<\bar{\varepsilon},
\]
 so that thanks to Step 2, we can apply~\cite[Theorem 1.1]{DeSilva} to conclude that locally at $x_{0}$ the free boundaries $\partial \Omega_{u_1}$ and $\partial \Omega_{u_2}$ are $C^{1, \alpha}$ graphs. By the arbitrariness of $x_{0}$ this concludes the proof for some $\alpha$. The fact that the result follows for all $\alpha\in [0,1/2)$ is pointed out in \cite{FerreriVelichkov}.
\end{proof}

To conclude, we present a summary with the complete proof of the main theorem of this paper.

\begin{proof}[Conclusion of the proof of Theorem \ref{thm:main}]

By Lemma~\ref{le:measOK}, we may fix $\eta>0$ sufficiently small to that, by taking an optimal vector $(u_1,\dots,u_k)\in \overline{H}$ of problem~\eqref{eq:ceta}, then the associated nodal sets
\[
(\Omega_{u_1},\ldots, \Omega_{u_k})\text{ form an optimal partition of the original constrained optimal partition problem } \eqref{eigenvalue_problem},
\] 
and every possible optimizer of \eqref{eigenvalue_problem} is of this form (being therefore connected - see Remark \ref{rem:connected}). We check that this partition satisfies the statement of Theorem \ref{thm:main}. First of all, given $i$,
\[
\partial \Omega_{u_i}=\Gamma_{OP}(\partial\Omega_{u_i})\cup \Gamma_{TP}( \partial\Omega_{u_i})\cup  \Gamma_{B} (\partial\Omega_{u_i})
\]
as a consequence of Theorems \ref{thm:notriplepoints} and \ref{thm:notpbdry}.

The fact that $\Gamma_{TP}(\Omega_{u_i})\subset \text{Reg}(\partial \Omega_i)$, which is of class $C^{1,\alpha}$, follows from Theorem \ref{thm:regtwophase}. All the remaining properties in (1) follow from Proposition \ref{prop:viscosityonephase}.

Finally, assuming $C^{1,1}$ regularity of $\Omega$, since the boundary of the box $\Omega$ only admits one-phase points and \eqref{eq:minonephasepenalized} holds,  \cite[Theorem 1.2]{RussTreyVelichkov} shows that there are no singular points on $\partial \Omega$ and that the regular part is $C^{1,1/2}$. In particular, item (2) is proved.
 \end{proof}

\begin{remark}
As a final remark, we point out that the proof of Lemma~\ref{le:onephase} was already showing that branching points had necessarily to be cusps. In fact, with the same argument therein, if a two-phase branching point were not a cusp, then we could find a new box $\Omega'$ for which the point is on the boundary of the box. But then, thanks to Theorem~\ref{thm:notpbdry}, we know that there can not be two-phase points in the boundary of the box and we have a contradiction.
\end{remark}


 \appendix
 \section{Some useful results and remarks}

\begin{lemma}[{{\cite[Lemma A.1]{ASST}}}]\label{lemma:expansion_L^2} Let $u \in L^2(\Omega)$ with $u^+\not\equiv 0$. Then, for all $\varphi \in L^2(\Omega)$, 
\begin{equation*}\label{inverse-norm-lemma}
\displaystyle \frac{1}{\|(u \pm t\varphi)^{+}\|_2^2 }= \dfrac{1}{\|u^+\|_2^2} \mp \frac{2t}{\|u^+\|_2^4} \int_{\Omega}u^+\varphi + o(t) \,\qquad \text{ as } t\to 0^+.
\end{equation*}
\end{lemma}

\begin{lemma}[{{\cite[Lemma 4.5]{BucurVelichkov}}}]\label{le:potentialestimate} There exists $C>0$  such that, for every $R>0$ and  $u \in H^1\left(B_r\right)$,
\[
\frac{1}{r^2}\left|\{u=0\} \cap B_r\right|\left(\dashint_{\partial B_r} u \, d \mathcal H^{N-1}\right)^2 \leq C_N \int_{B_r}|\nabla u|^2\, d x,
\]
where $C_N$ is a constant that depends only on the dimension $N$.
\end{lemma}

\begin{lemma}\label{lemma:kind_of_meanvalueth} 
Let $u$ be an $H^1_{loc}(\R^N)$ nonnegative function,  subharmonic in $\R^N$, and let $\omega\Subset B_\rho\subset \mathbb{R}^N$. Then
\[
\sup_{\omega} u\leq \left(\frac{\rho}{dist(\omega,\partial B_\rho)}\right)^N \dashint_{\partial B_\rho} u(y)\, d\mathcal H^{N-1}(y).
\]
In particular, given $\theta\in (0,1)$, 
\[
\sup_{B_{\theta\rho}\cap \Omega} u \leq \frac{1}{(1-\theta)^N}\dashint_{\partial B_\rho} u(y)\, d\mathcal H^{N-1}(y).
\]
\end{lemma}
\begin{proof}
Let $v$ be the harmonic extension of $u|_{\partial B_\rho}$ in $B_\rho$, that is, the unique solution of
\[
\Delta v=0 \text{ in } B_\rho, \qquad v=u \text{ on } \partial B_\rho.
\]
Then, by the maximum principle and Poisson's formula, for $x\in \omega$ and since $|x-y|\geq d:=dist(\omega,\partial B_\rho)$ for every $y\in \partial B_\rho$,
\begin{align*}
u(x)\leq v(x) &= \frac{\rho^2-|x|^2}{\rho|\partial B_1|}\int_{\partial B_\rho} \frac{v(y)}{|x-y|^N}\, d\mathcal H^{N-1}(y)=\frac{\rho^2-|x|^2}{\rho|\partial B_1|}\int_{\partial B_\rho} \frac{u(y)}{|x-y|^N}\, d\mathcal H^{N-1}(y)\\
            &\leq \frac{\rho}{|\partial B_1| d^N} \int_{\partial B_\rho} u(y)\, d\mathcal H^{N-1}(y)=\left(\frac{\rho}{dist(\omega,\partial B_\rho)}\right)^N \dashint_{\partial B_\rho} u(y)\, d\mathcal H^{N-1}(y).
\end{align*}
In the particular case that $\omega=B_{\theta \rho}$, simply observe that $dist(\omega, \partial B_{\rho})=(1-\theta)\rho$. 
\end{proof}

\begin{lemma}[{{\cite[Lemma 2.14]{BucurVelichkov}}}: three-phase monotonicity lemma]\label{tpml} 
Let $u_i \in H^1_{loc}\left(\R^N\right), i=1,2,3$, be three nonnegative Sobolev functions such that $\Delta u_i \geq-1$ for each $i=1,2,3$, and $ u_i\cdot u_j \equiv 0$ in $\R^N$ for each $i \neq j$. Then there are dimensional constants $\varepsilon>0$ and $C>0$ such that
\[
\prod_{i=1}^3\left(\frac{1}{r^{2+\varepsilon}} \int_{B_r} \frac{\left|\nabla u_i\right|^2}{|x|^{N-2}}\, d x\right) \leq C\left(1+\sum_{i=1}^3 \int_{B_1} \frac{\left|\nabla u_i\right|^2}{|x|^{N-2}}\, d x\right)^3\qquad \text{for every $r \in\left(0, \frac{1}{2}\right)$.}
\]
\end{lemma}

\begin{remark}
It is well known that, under the assumptions of the previous lemma,
\[
\int_{B_1}\frac{\left|\nabla u_i\right|^2}{|x|^{N-2}}\, d x\leq C\left(1+\int_{B_2}u_i^2\right)
\]
(see \cite[Remark 1.2]{CJK2002}, or \cite[Lemma 2.1]{Velichkov_CJK} for a proof that does not require the continuity of the functions $u_i$). Therefore,
\begin{equation}\label{CJK_remark}
\prod_{i=1}^3\left(\frac{1}{r^{2+\varepsilon}} \int_{B_r} \frac{\left|\nabla u_i\right|^2}{|x|^{N-2}}\, d x\right) \leq C\left(1+\sum_{i=1}^3 \int_{B_2} u_i^2\right)^3\qquad \text{for every $r \in\left(0, \frac{1}{2}\right)$.}
\end{equation}
\end{remark}

\subsection{Some fact about Geometric Measure Theory}
Here we recall some facts about geometric measure theory, referring for example to~\cite[Chapter~2]{AmbrosioFuscoPallara} for more information.
The De Giorgi perimeter of a Borel set $E\subset\R^N$ is the quantity
	\[
	\Per(E):=\sup\left\{\int_{E}\div(\phi) \, dx \,:\,\phi\in C^1_c(\R^N,\R^N), \,\, \|\phi\|_\infty\le 1 \right\}\,.
	\]
	If $\Per(E)<+\infty$, then we say that $E$ has finite perimeter. 
 
	Equivalently, it can be defined in the setting of  functions of bounded variation as the total variation of the distributional derivative of characteristic functions. We recall that if $\Om\subset\R^N$ is an open set, $u\in L^1(\Om)$ is a function of bounded variation, and we write $u\in BV(\Om)$, when the distributional derivative $Du$ of $u$ is an $\R^N$-valued finite Radon measure.
	Then $E\subset\R^N$ is a set of finite perimeter, if and only if $\chi_E\in BV(\R^N)$ and 
	\begin{equation}\label{eq:perdistrib}
	\Per(E)=|D\chi_E|(\R^N)=:\|D\chi_E\|_{TV(\R^N)}.
	\end{equation}	
	Whenever it exists, the quantity
	\[
	\theta_E(x):=\lim_{r\to0}\frac{|E\cap B_r(x)|}{|B_r(x)|}\in [0,1]\,,
	\]
	is called the density of a Borel set $E$ at $x$. Given $t\in [0,1]$, we denote by $E^t$ the subset of points of $\R^N$ such that $\theta_E(x)=t$, and we call essential boundary of $E$ the set $\partial^e E=E\setminus (E^0\cup E^1)$. Eventually, we define the reduced boundary of $E$ as the set $\partial^*E\subset\partial^eE$ of points of the essential boundary such that the measure theoretic inner unit normal 
	\[
	\nu_E(x):=\lim_{r\to0}\frac{D\chi_E(B_r(x))}{|D\chi_E|(B_r(x))}
	\] 
	exists.
	
	The geometry  of the boundary of sets of finite perimeter is described by the following two fundamental results.
\begin{theorem}[De Giorgi's Structure Theorem]\label{thm:deGiorgi}
Let $E$ be a set of finite perimeter. Then $\partial^*E$ is $\mathcal H^{N-1}-$rectifiable,  $P(E)=\mathcal H^{N-1}(\partial^*E)$ and, if $x\in \partial^*E$, then $(E-x)/r$ converges in $L^1_{\rm loc}$ to the hyperspace orthogonal to $\nu_E(x)$, as $r\to 0$. Moreover the following divergence formula holds:
\[
\int_{E}\div\phi\,dx=-\int_{\partial^*E}\phi\cdot\nu_E \, d \mathcal H^{N-1}\,,
\]  
for any vector field $\phi\in C^1_c(\R^N,\R^N)$.
\end{theorem}
	
\begin{theorem}[Federer's Structure Theorem]\label{Federer_thm}
Let $E$ be a set of finite perimeter. Then $\partial^*E\subset E^{1/2}$ and $\mathcal H^{N-1}(\partial^eE\setminus\partial^*E)=0$. In particular $\partial^*E$, $E^{1/2}$, and $\partial^eE$ are equivalent, up to a $\mathcal H^{N-1}$-negligible set. 
\end{theorem}

\subsection{Some facts about quasi-open sets}\label{ssec:quasiopen}
Finally, we recall some definitions and facts about quasi-open sets which will be needed in the paper.
\begin{definition}\label{def:quasiopen}
A \emph{quasi-open} set is a measurable set $\Omega\subset \R^N$ such that for all $\eps>0$ there exists $K_\eps$ compact such that its (Newtonian) capacity ${\rm cap}(K_\eps)<\eps$ and $\Omega\setminus K_\varepsilon$ is open.  
Similarly, a function $u:\Omega\to\R$ is \emph{quasi-continuous} if for all $\varepsilon>0$ there exists a compact set $K_\varepsilon$ such that ${\rm cap}(K_\varepsilon)<\varepsilon$ and the restriction of $u$ to $\Omega\setminus K_\varepsilon$ is continuous. Eventually, we say that a property holds \emph{quasi-everywhere} on a set if it holds up to a set of null-capacity.
\end{definition}  
It is well-known that every $u\in H^1(\R^N)$ admits a quasi-continuous representative $\widetilde u$. Moreover, if $\widetilde u$ and $\widehat u$ are two quasi-continuous representatives of $u$, then they are equal quasi-everywhere. Therefore, in this paper, for every $u\in H^1(\R^N)$ we identify it with its quasi-continuous representative.  A quasi-open set is then simply a superlevel set of (the quasi-continuous representative of) a function $u\in H^1(\R^N)$.  For more details on quasi-open sets and
quasi-continuous functions, and the definition of capacity we refer
for example to~\cite[Chapter~3]{HenrotPierre}.
		
Let us also stress that it is standard to define the Sobolev space $H^1_0$ on a quasi-open set $\Omega\subset \R^N$ as \[
H^1_0(\Omega)=\{u\in H^1(\R^N) : u=0\text{ quasi-everywhere in }\R^N\setminus \Omega\}.
\]
If $\Omega$ is an open set, this definition coincides with the usual one, see~\cite[Section~3.3.5]{HenrotPierre}  for more details.

\vspace{1cm}
\textbf{Acknowledgements:} Dario Mazzoleni has been partially supported by the MUR via PRIN projects, funded by the European Union -- Next Generation EU, ``$NO^3$'' P2022R537CS (CUP\_F53D23002810006) and ``Mathematics for Industry 4.0'' P2020F3NCPX (CUP\_F19J21017440001), and by the INdAM-GNAMPA project 2024 CUP\_E53C23001670001. Hugo Tavares and Makson S. Santos are partially supported by the Portuguese government through FCT - Funda\c c\~ao para a Ci\^encia e a Tecnologia, I.P., under the projects UID/MAT/04459/2020 and PTDC/MAT-PUR/1788/2020.  Hugo Tavares is also partially supported by FCT - Funda\c c\~ao para a Ci\^encia e a Tecnologia, I.P.  within the scope of the
project \emph{Spectral Optimal Partitions: geometric and numerical analysis}, reference 2023.13921.PEX

\bibliographystyle{abbrv}
\bibliography{MST}

\noindent\textsc{Dario Mazzoleni}\\
Dipartimento di Matematica ``F. Casorati''\\
University of Pavia\\
Via Ferrata 5, 27100 Pavia (Italy)\\
\noindent\texttt{dario.mazzoleni@unipv.it}

\vspace{.15in}

\noindent\textsc{Makson S. Santos}\\
Center for Mathematical Studies (CEMS.UL)\\
University of Lisbon\\
1749-016 Lisboa, Portugal\\
\noindent\texttt{msasantos@ciencias.ulisboa.pt}

\vspace{.15in}

\noindent\textsc{Hugo Tavares}\\
Departamento de Matem\'atica do Instituto Superior T\'ecnico\\
Universidade de Lisboa\\
1049-001 Lisboa, Portugal\\
\noindent\texttt{hugo.n.tavares@tecnico.ulisboa.pt}

\end{document}